\documentclass[a4paper]{amsart}

\usepackage{graphicx}
\usepackage{bbm, caption, color, enumitem, fourier}
\usepackage{tikz}
\def\fdd{\stackrel{\textnormal{fdd}}{\Rightarrow}}
\def\longfdd{\stackrel{\textnormal{fdd}}{\Longrightarrow}}
\def\d{\mathrm{d}}
\def\zero{{\bf{z}}}
\def\dualoperator{\vartheta}
\def\shift{\theta}
\def\spine{{\mathbb{S}}}
\newcommand{\mrca}[2]{#1 {\wedge} #2}

\newcommand{\length}[1]{{\rm{Len}}\left(#1\right)}

\def\C{{\mathbb{C}}}

\def\E{{\mathbb{E}}}
\def\F{{\mathbb{F}}}
\def\H{{\mathbb{H}}}

\def\L{{\mathbb{L}}}
\def\M{{\mathbb{M}}}
\def\N{{\mathbb{N}}}
\def\P{{\mathbb{P}}}

\def\R{{\mathbb{R}}}
\def\S{{\mathbb{S}}}

\def\Z{{\mathbb{Z}}}

\def\y{{\mathbb{Y}}}

\def\cA{{\mathcal A}}
\def\cC{{\mathcal C}}
\def\cG{{\mathcal G}}
\def\cH{{\mathcal H}}
\def\cM{{\mathcal M}}
\def\cP{{\mathcal P}}
\def\cQ{{\mathcal Q}}

\def\cV{{\mathcal V}}

\def\eps{\varepsilon}

\def\vp{\varphi}
\def\bvp{{\bar \varphi}}
\def\bTm{{\widetilde T}^{-1}}
\newcommand{\indicator}[1]{\mathbbm{1}_{\{#1\}}}
\newcommand{\Indicator}[1]{\mathbbm{1}(#1)}

\newcommand{\stub}[1]{
	\draw (-.1,-.1+#1) -- (.1,.1+#1);
	\draw (.1,-.1+#1) -- (-.1,.1+#1);
}
\newcommand{\ancestor}[1]{
	\filldraw (0,#1) circle [radius=3pt];
	\draw[line width=2pt] (0,0) -- (0,#1);
}
\newcommand{\laststub}[1]{
	\stub{#1}
	\draw[line width=2pt] (0,0) -- (0,#1);
}
\newcommand{\done}[1]{
	\filldraw (-.1,#1-.1) rectangle (.1,#1+.1);
}
\theoremstyle{plain}

\newtheorem{lemma}{Lemma}[section]
\newtheorem{theorem}[lemma]{Theorem}
\newtheorem{prop}[lemma]{Proposition}
\newtheorem{corollary}[lemma]{Corollary}
\newtheorem*{claim*}{Claim}

\newtheorem{innercondition}{Condition}
\newenvironment{condition}[1]
 {\renewcommand\theinnercondition{#1}\innercondition}
 {\endinnercondition}

\theoremstyle{remark}

\newtheorem{remark}[lemma]{Remark}

\theoremstyle{definition}

\title[Height and contour processes of Crump-Mode-Jagers forests (I)]{Height and contour processes of Crump-Mode-Jagers forests (I): general distribution and scaling limits in the case of short edges}

\author{Emmanuel Schertzer}
\address{LPMA/UMR 7599\\Universit\'e Pierre et Marie Curie (P6) -- Bo\^ite courrier 188\\75252 PARIS Cedex 05 (FRANCE)}
\email{emmanuel.schertzer@upmc.fr}

\author{Florian Simatos}
\address{ISAE\\10 avenue Edouard Belin\\31055 Toulouse Cedex 4\\France}
\email{florian.simatos@isae.fr}

\date{\today}
\numberwithin{equation}{section}
\begin{document}

\begin{abstract}
Crump--Mode--Jagers (CMJ) trees generalize Galton--Watson trees by allowing individuals to live for an arbitrary duration and give birth at arbitrary times during their life-time. In this paper, we are interested in the height and contour processes encoding a general CMJ tree.

We show that the one-dimensional distribution of the height process can be expressed in terms of a random transformation of the ladder height process associated with the underlying Lukasiewicz path. As an application of this result, when edges of the tree are ``short'' we show that, asymptotically, (1) the height process is obtained by stretching by a constant factor the height process of the associated genealogical Galton--Watson tree, (2) the contour process is obtained from the height process by a constant time change and (3) the CMJ trees converge in the sense of finite-dimensional distributions.
\end{abstract}
\maketitle

\setcounter{tocdepth}{1}
\tableofcontents 

\newpage

\section{Introduction}

\subsection{Galton--Watson forests and their scaling limits.}
A planar discrete rooted tree is a rooted tree where edges have unit length and which is endowed with an ordering on siblings, in such a way that it can be naturally embedded in the plane. Since the seminal work of Aldous, Neveu, Pitman and others~\cite{Aldous91:0, Aldous91:1, Aldous93:0, Le-Gall93:0, Neveu89:0, Neveu89:1}, it is well known that such a tree is conveniently encoded by its height and contour processes. To generate these processes, one can envision a particle starting from the root and traveling along the edges of the tree at unit speed, from left to right. The contour process is simply constructed by recording the distance of the particle from the root of the tree. To generate the height process, we start by labeling the vertices of the tree according to their order of visit by the exploration particle (i.e., from left to right): the height process evaluated at $k$ is then given by the distance from the root of the $k$th vertex. 

From a probabilistic standpoint, a particularly interesting case is the Galton--Watson case where each individual $u$ in the tree begets a random number of offspring $\xi_u$, these random variables being i.i.d.\ with common distribution $\xi$. In the critical and subcritical cases -- i.e., when $\E(\xi)\leq 1$ -- the tree is almost surely finite. Considering an infinite sequence of such i.i.d.\ random rooted planar trees, we can generate a random (planar) forest with its corresponding contour and height processes -- respectively denoted by $\mathcal C$ and $\cH$ -- obtained by pasting sequentially the height and contour processes of the trees composing the forest. 

When $\E(\xi^2)<\infty$, Aldous~\cite{Aldous93:0} proved that the large time behavior of those processes (properly normalized in time and space) can be described in terms of a reflected Brownian motion. More precisely, in the critical case $\E(\xi)=1$ and if $0<\sigma = \mbox{Var}(\xi^2)<\infty$, we have
\[ \left( \frac{1}{\sqrt{p}} \cH([pt]), \frac{1}{\sqrt{p}} \cC(pt) \right) \ \Longrightarrow \ \frac{2}{\sigma} \ \left( |w(t)|, |w(t/2)| \right) \]
with $w$ a standard Brownian motion and the convergence holds weakly (in the functional sense). 

When the second moment of the offspring distribution is infinite and the offspring distribution is in the domain of attraction an $\alpha$-stable law with $\alpha\in(1,2)$, Le Gall and Le Jan~\cite{Le-Gall98:0} and then Duquesne and Le Gall~\cite{Duquesne02:0} proved the existence of a scaling sequence $(\eps_p, p\in\N)$ and a limiting continuous path $\cH_\infty$ such that 
\[ \left( \eps_p \cH([pt]), \eps_p \cC(pt) \right) \ \Longrightarrow \ \ \left( \cH_\infty(t), \cH_\infty(t/2) \right) \]
where $\cH_\infty$ can be expressed as a functional of a spectrally positive L\'evy process. As in the finite second moment case alluded above, we note that the height and contour processes are asymptotically related by a simple deterministic and constant time change.

\subsection{Crump-Mode-Jagers forests}

The subject of the present paper is the study of the height and contour processes of planar Crump--Mode--Jagers (CMJ) forests, which are random instances of \emph{chronological forests}. Chronological trees generalize discrete trees in the following way: each individual $u$ is endowed with a pair $(V_u, \cP_u)$ such that:
\begin{enumerate}
	\item[(1)] $V_u \in (0,\infty)$ represents the life-length of $u$;
	\item[(2)] $\cP_u$ is a point measure which represents the age of $u$ at childbearing. In particular, we enforce $\mbox{Supp}(\cP_u)\subset (0,V_u]$, so that individuals produce their offspring during their lifetime.
\end{enumerate}
Note that $\lvert \cP_u \rvert = \cP_u(0,V_u]$ is the number of children of $u$. As noted by Lambert in \cite{Lambert10:0}, a chronological tree can be regarded as a tree satisfying the rule ``edges always grow to the right''. 
This is illustrated in Figure~\ref{fig:sequential-construction} where we present a sequential construction of a planar chronological forest from a sequence of ``sticks'' $\omega=(\omega_n, n\geq0)$, where $\omega_n=(V_n, \cP_n)$.

\begin{figure}[p!]
	\centering
		\begin{tikzpicture}
			[color=black]
			% \foreach \n/\x/\shift in {0/0/-.5,1/0/1,2/.5/3,3/1/6,4/1.5/9}{
			% 	\begin{scope}[shift={(\shift,0)}]
			% 		\input{v3-pic/\n}
			% 		\node[anchor=north] at (\x,0) {$n = \n$};
			% 	\end{scope}
			% }
			% \n/\x/\shift = 0/0/-.5
			\begin{scope}[shift={(-.5,0)}]
				\node[anchor=north] at (0,0) {$n = 0$};
			\end{scope}
			% \n/\x/\shift = 1/0/1
			\begin{scope}[shift={(0,0)}]
				% \input{v3-pic/1}
				% root
				\draw (0,0) -- (0,2);
				\stub{1.5}
				\stub{.5}
				\node[anchor=north] at (1,0) {$n = 1$};
			\end{scope}
			% \n/\x/\shift = 2/.5/3
			\begin{scope}[shift={(3,0)}]
				% \input{v3-pic/2}
				% root
				\draw (0,0) -- (0,2);
				\stub{1.5}
				\draw[dashed] (0,1.5) -- (1,1.5);
				\begin{scope}[shift={(1,1.5)}]
					\draw (0,0) -- (0,1.5);
					\stub{1.2}
					\stub{.5}
				\end{scope}
				\stub{.5}
				\node[anchor=north] at (.5,0) {$n = 2$};
			\end{scope}
			% \n/\x/\shift = 3/1/6
			\begin{scope}[shift={(6,0)}]
				% \input{v3-pic/3}
				% root
				\draw (0,0) -- (0,2);
				\stub{1.5}
				% 1
				\draw[dashed] (0,1.5) -- (1,1.5);
				\begin{scope}[shift={(1,1.5)}]
					\draw (0,0) -- (0,1.5);
					\stub{1.2}
					% 2
					\draw[dashed] (0,1.2) -- (1,1.2);
					\begin{scope}[shift={(1,1.2)}]
						% 3
						\draw (0,0) -- (0,1.5);
						\stub{.9}
					\end{scope}
					\stub{.5}
				\end{scope}
				\stub{.5}
				\node[anchor=north] at (1,0) {$n = 3$};
			\end{scope}
			% \n/\x/\shift = 4/1.5/9
			\begin{scope}[shift={(9,0)}]
				% \input{v3-pic/4}
				% root
				\draw (0,0) -- (0,2);
				\stub{1.5}
				% 1
				\draw[dashed] (0,1.5) -- (1,1.5);
				\begin{scope}[shift={(1,1.5)}]
					\draw (0,0) -- (0,1.5);
					\stub{1.2}
					% 2
					\draw[dashed] (0,1.2) -- (1,1.2);
					\begin{scope}[shift={(1,1.2)}]
						% 3
						\draw (0,0) -- (0,1.5);
						\stub{.9}
						% 4
						\draw[dashed] (0,.9) -- (1,.9);
						\begin{scope}[shift={(1,.9)}]
							\draw (0,0) -- (0,1);
						\end{scope}
					\end{scope}
					\stub{.5}
				\end{scope}
				\stub{.5}
				
				\node[anchor=north] at (1.5,0) {$n = 4$};
			\end{scope}

			% \foreach \n/\x/\shift in {5/2/0,6/2.5/5}{
			% 	\begin{scope}[shift={(\shift,-6)}]
			% 		\input{v3-pic/\n}
			% 		\node[anchor=north] at (\x,0) {$n = \n$};
			% 	\end{scope}
			% }
			% \n/\x/\shift = 5/2/0
			\begin{scope}[shift={(0,-6)}]
				% \input{v3-pic/5}
				% root
				\draw (0,0) -- (0,2);
				\stub{1.5}
				% 1
				\draw[dashed] (0,1.5) -- (1,1.5);
				\begin{scope}[shift={(1,1.5)}]
					\draw (0,0) -- (0,1.5);
					\stub{1.2}
					% 2
					\draw[dashed] (0,1.2) -- (1,1.2);
					\begin{scope}[shift={(1,1.2)}]
						% 3
						\draw (0,0) -- (0,1.5);
						\stub{.9}
						% 4
						\draw[dashed] (0,.9) -- (1,.9);
						\begin{scope}[shift={(1,.9)}]
							\draw (0,0) -- (0,1);
						\end{scope}
					\end{scope}
					\stub{.5}
					% 5
					\draw[dashed] (0,.5) -- (3,.5);
					\begin{scope}[shift={(3,.5)}]
						\draw (0,0) -- (0,2);
					\end{scope}
				\end{scope}
				\stub{.5}
				\node[anchor=north] at (2,0) {$n = 5$};
			\end{scope}
			% \n/\x/\shift = 6/2.5/5
			\begin{scope}[shift={(5,-6)}]
				% \input{v3-pic/6}
				% root
				\draw (0,0) -- (0,2);
				\stub{1.5}
				% 1
				\draw[dashed] (0,1.5) -- (1,1.5);
				\begin{scope}[shift={(1,1.5)}]
					\draw (0,0) -- (0,1.5);
					\stub{1.2}
					% 2
					\draw[dashed] (0,1.2) -- (1,1.2);
					\begin{scope}[shift={(1,1.2)}]
						% 3
						\draw (0,0) -- (0,1.5);
						\stub{.9}
						% 4
						\draw[dashed] (0,.9) -- (1,.9);
						\begin{scope}[shift={(1,.9)}]
							\draw (0,0) -- (0,1);
						\end{scope}
					\end{scope}
					\stub{.5}
					% 5
					\draw[dashed] (0,.5) -- (3,.5);
					\begin{scope}[shift={(3,.5)}]
						\draw (0,0) -- (0,2);
					\end{scope}
				\end{scope}
				\stub{.5}
				% 6
				\draw[dashed] (0,.5) -- (5,.5);
				\begin{scope}[shift={(5,.5)}]
					\draw (0,0) -- (0,4);
					\stub{3.5}
					\stub{2.5}
					\stub{1}
				\end{scope}
				
				\node[anchor=north] at (2.5,0) {$n = 6$};
			\end{scope}

			% \foreach \n/\x/\shift in {10/4/1}{
			% 	\begin{scope}[shift={(\shift,-13.5)}]
			% 		\input{v3-pic/\n}
			% 		\node[anchor=north] at (\x,0) {$n = \n$};
			% 	\end{scope}
			% }
			% \n/\x/\shift = 10/4/1
				\begin{scope}[shift={(1,-13.5)}]
					% \input{v3-pic/10}
					% root
					\draw (0,0) -- (0,2);
					\stub{1.5}
					% 1
					\draw[dashed] (0,1.5) -- (1,1.5);
					\begin{scope}[shift={(1,1.5)}]
						\draw (0,0) -- (0,1.5);
						\stub{1.2}
						% 2
						\draw[dashed] (0,1.2) -- (1,1.2);
						\begin{scope}[shift={(1,1.2)}]
							% 3
							\draw (0,0) -- (0,1.5);
							\stub{.9}
							% 4
							\draw[dashed] (0,.9) -- (1,.9);
							\begin{scope}[shift={(1,.9)}]
								\draw (0,0) -- (0,1);
							\end{scope}
						\end{scope}
						\stub{.5}
						% 5
						\draw[dashed] (0,.5) -- (3,.5);
						\begin{scope}[shift={(3,.5)}]
							% 6
							\draw (0,0) -- (0,2);
						\end{scope}
					\end{scope}
					\stub{.5}
					% 7
					\draw[dashed] (0,.5) -- (5,.5);
					\begin{scope}[shift={(5,.5)}]
						\draw (0,0) -- (0,4);
						\stub{3.5}
						% 8
						\draw[dashed] (0,3.5) -- (1,3.5);
						\begin{scope}[shift={(1,3.5)}]
							\draw (0,0) -- (0,2);
						\end{scope}
						\stub{2.5}
						% 9
						\draw[dashed] (0,2.5) -- (2,2.5);
						\begin{scope}[shift={(2,2.5)}]
							\draw (0,0) -- (0,1);
						\end{scope}
						\stub{1}
						% 10
						\draw[dashed] (0,1) -- (3,1);
						\begin{scope}[shift={(3,1)}]
							\draw (0,0) -- (0,1);
							\stub{1}
							% 11
							\draw[dashed] (0,1) -- (1,1);
							\begin{scope}[shift={(1,1)}]
								\draw (0,0) -- (0,1);
							\end{scope}
						\end{scope}
					\end{scope}
					\node[anchor=north] at (4,0) {$n = 10$};
				\end{scope}
		\end{tikzpicture}
	\caption{We start at $n = 0$ with nothing, then add $\omega_0$ at time $n=1$. At this time, there are two stubs and so the next stick $\omega_1$ is grafted to the highest stub, and we repeat until time $n = 10$ at which time no more stub is available and the tree is built. Then, the next step proceeds with the construction of the next tree, thus constructing the second tree of the forest, etc.}
	\label{fig:sequential-construction}
\end{figure}
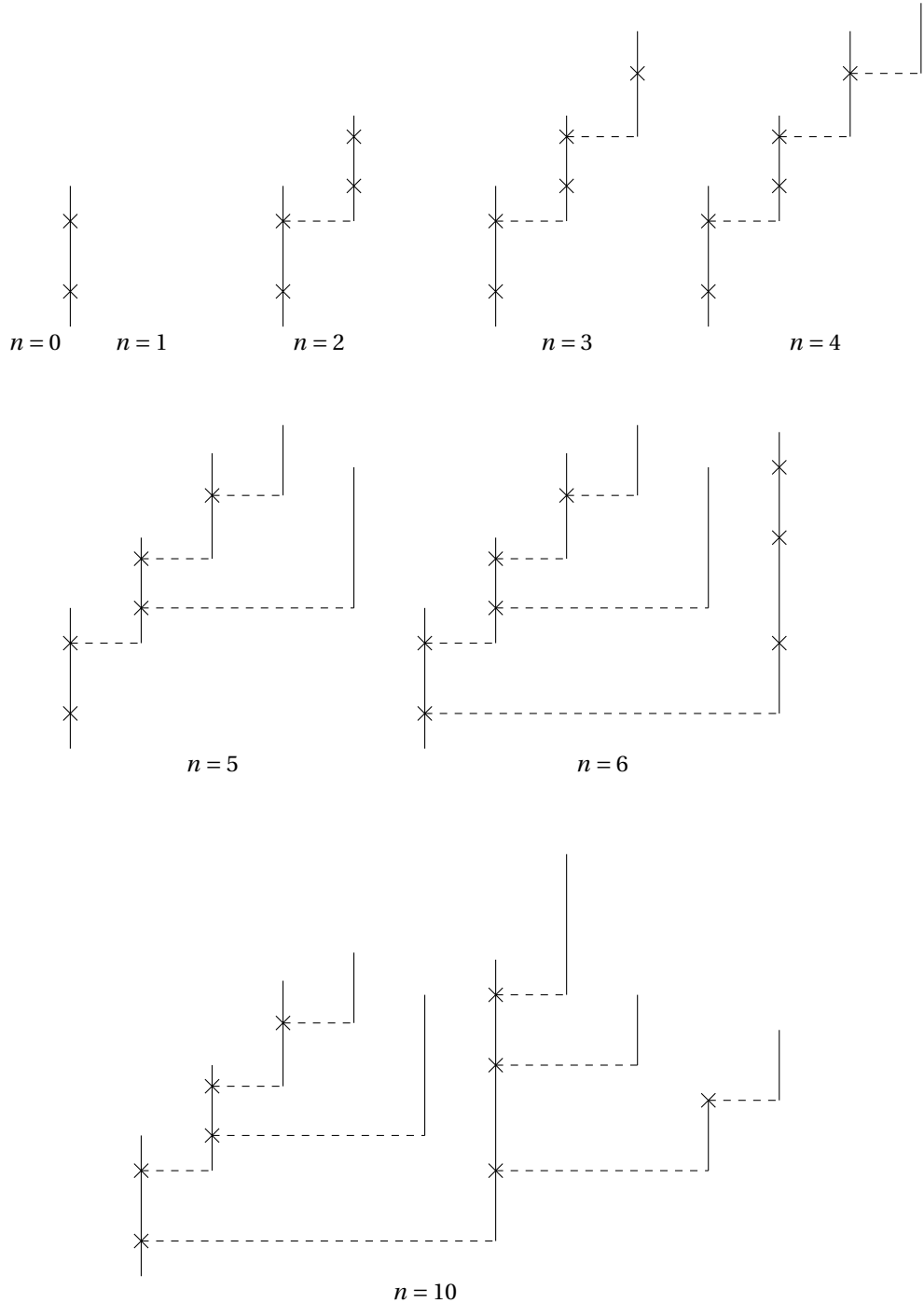

At time $n = 0$ we start with the empty forest and we add the stick $\omega_0$ at time $n = 1$. In the case considered in Figure~\ref{fig:sequential-construction}, $\cP_0$ has two atoms which correspond to birth times of individuals, but these two atoms are not yet matched with the sticks corresponding to these individuals. These unmatched atoms are called \emph{stubs}, and when there is at least one stub we apply the following rule:
\begin{description}
	\item[Rule \#$1$] if there is at least one stub, we graft the next stick to the highest stub.
\end{description}

Thus, we iteratively apply this rule until there is no more stub, at which point we have built a complete chronological tree with a natural planar embedding. Figure~\ref{fig:sequential-construction} illustrates a particular case where at time $10$ there is no more stub, in which case we apply the following rule:
\begin{description}
	\item[Rule \#$2$] if there is no stub, we start a new tree with the next stick.
\end{description}

Thus, starting at time $n = 0$ from the empty forest and iterating these two rules, we build in this way a forest $\F^\infty$, possibly consisting of infinitely many chronological trees. By definition, a CMJ forest is obtained when the initial sticks are i.i.d., and throughout the paper we will denote their common distribution by $(V^*,\cP^*)$.

\subsection{ Chronological height and contour processes of CMJ forests} As for discrete trees, the contour process of a CMJ forest is obtained by recording the position of an exploration particle traveling at unit speed along the edges of the forest from left to right, moving, when a chronological tree is represented as in Figure~\ref{fig:sequential-construction}, at infinite speed along dashed lines. This process will be referred to as the {\it chronological} contour process associated to the CMJ forest, and the chronological height of the $n$th individual is defined as its date of birth. We define the genealogical contour and height processes as the contour and height processes associated to the discrete forest encoding the genealogy of $\F^\infty$.

Contour processes of CMJ forests have been considered by Lambert in~\cite{Lambert10:0} in the particular setting where
birth events are distributed in a Poissonian way along the sticks independently of the life-length -- the so-called binary, homogeneous case. Under this assumption, the author showed that the (jumping) contour process is a spectrally positive L\'evy process. See also~\cite{Delaporte:15, Felipe:15, Lambert15:0, Lambert13:0, Mathieu:13, Richard:14} for related works.

To our knowledge, little is known in the general case and in the present study, we determine in full generality: 
\begin{enumerate}
\item[(1)] the distribution of the contour/height process of a CMJ forest;
\item[(2)] the correlation between the height/contour process of a CMJ forest and the height/contour process of its underlying genealogy.
\end{enumerate}
One of our first result is a description of the one-dimensional marginal of the height processes of a CMJ forest in terms of a bivariate renewal process. This two-dimensional process is constructed as a random functional of the weak ascending ladder height process associated to the dual Lukasiewicz path starting from $n$. This is is the subject of Section~\ref{sec:dec}.

\subsection{Scaling limits} In the near-critical case it is well-known that, properly scaled in time and space, the genealogical height and contour processes associated to Galton--Watson trees converge toward a continuous process. Except for the binary, homogeneous case and to the best of our knowledge, little is known outside this case: we claim that our results highlighting the distribution of the chronological height process can be used to deal with a broad class of CMJ forests.

To support this claim, we treat in details in the present paper the case of short edges where the genealogical and chronological structures become deterministically proportional to one another. Moreover, current work in progress~\cite{Schertzer:2} suggests that our techniques can be extended to a broader class of CMJ forests including cases where the genealogical and chronological structures are not deterministically obtained from one another, see Section~\ref{sub:perspectives} below for more details.
\\

To explain our results in the short edge case, let $\y^*$ be the random number obtained by first size-biasing the random variable $\lvert \cP^* \rvert$ (i.e., the number of atoms in the point measure $\cP^*$) and then by recording the age of the individual when giving birth to a randomly chosen child. The mean of $\y^*$ has a simple expression, namely
\[ \E(\y^*) = \E \left( \int u \cP^*(\d u) \right). \]
As noticed by Nerman~\cite{Nerman84:0}, this random variable describes the age of an ancestor of a typical individual $u$ when giving birth to the next ancestor of $u$. For this reason, $\y^*$ and in particular the condition $\E(\y^*) < \infty$ -- which is one way to formalize the ``short edge'' condition -- plays a major role in previous works on CMJ processes, see for instance~\cite{Sagitov86:0, Sagitov90:0, Sagitov94:0, Sagitov94:1, Sagitov95:0, Sagitov97:0}. In the present paper we prove that if $\E(\y^*) < \infty$, then in the near-critical regime the asymptotic behavior of the chronological height process is obtained by stretching the genealogical height process by the deterministic factor $\E(\y^*)$. This result is stated and proved in Section~\ref{sect:cv-height}.

The analysis of the contour process is more delicate (see Section~\ref{tech:challenges} below for more details). Our main result shows that when $\E(V^*)<\infty$ -- another way to formalize the ``short edge'' condition -- the chronological contour process is obtained from the chronological height process by rescaling time by the deterministic factor $1/(2\E(V^*))$. Hence, again provided that edges are short enough, this result provides a relation between the height and contour processes which is analogous to the discrete case. This result is stated in Section~\ref{sec:cv:contour} where the general structure of the proof is given, and details are provided in Section~\ref{sec:proofs}.

Finally, we prove that when both $\y^*$ and $V^*$ have finite means, the minimum of the chronological contour process is obtained by scaling the minimum of the genealogical height process, in space by $\E(\y^*)$ and in time by $1/\E(V^*)$. This shows that the genealogical and chronological trees, and not only the height/contour processes, are asymptotically close to one another. In particular, under these assumptions the CMJ trees themselves converge in the sense of finite-dimensional distributions.

\subsection{Technical challenges}\label{tech:challenges} As already discussed, Duquesne and Le Gall~\cite{Duquesne02:0} showed under rather mild conditions that the contour and height processes of Galton--Watson trees converge weakly to a continuous function. In the CMJ framework, we establish convergence in the sense of finite-dimensional distributions to a limiting object provided that edges are short enough. In Section~\ref{sec:example} we present simple examples where finite-dimensional distributions of the scaled contour and height processes converge, but the processes themselves fail to converge in a functional sense. To be more precise, in this example the contour process becomes unbounded on any finite time-interval. This gap between convergence of finite-dimensional distributions and weak convergence also exists in the Galton--Watson case, however we argue in Section~\ref{sec:example} that it is more significant in the CMJ case.

The main steps of the proof of our result on the relation between the contour and height processes in the case of short edges (i.e., when $\E(V^*)<\infty$) are highlighted in Section~\ref{sec:overview}. Due to the potential existence of pathological times when the contour/height process becomes degenerate (as illustrated by the example in Section~\ref{sec:example}), the convergence of the contour process raises new technical challenges that are absent in the discrete setting. In order to overcome those difficulties, we develop new tools presented in Sections~\ref{sec:preliminary-results} and~\ref{sec:proofs}.

\subsection{Perspectives} \label{sub:perspectives} The present paper aims at initiating the systematic study of scaling limits of CMJ forests. Most of the present paper is devoted to developing fundamental tools which, we believe, have the potential to tackle a broad class of CMJ forests and which will be the basis of subsequent papers.

The cornerstone of our approach is Proposition~\ref{prop:formula-H} below, which indicates how to recover a CMJ forest from its underlying genealogy by a random stretching. At the discrete level, this stretching operation is correlated with the genealogical structure in intricate ways but it suggests three possible universality classes:
\begin{description}
	\item[First class] the random stretching becomes asymptotically deterministic (this is the class to which Galton--Watson forests belong);
	\item[Second class] the stretching remains random in the limit, but uncorrelated with the genealogy;
	\item[Third class] the stretching remains random and correlated with the underlying genealogical structure. 
\end{description}
To show the potential of our techniques, we deal in the present paper with the first class, which corresponds to the ``short edge'' condition discussed earlier.

In current work in progress~\cite{Schertzer:2} we are dealing with the second class. Starting from the limiting genealogical structure, encoded by a continuous path, the chronological height process is obtained by marking the branches of the forest with a Poisson point process: each mark carries a random number encoding the chronological contribution of the vertex under consideration. We conjecture that the limiting object should be related to the Poisson snake (see e.g., \cite{Abraham02:0} and \cite{Bertoin97:2}).

Finally, studying the third class will presumably require new ideas given that the correlation structure may be quite involved: this will be the subject of further study.

\section{Spine, height and contour processes} \label{sec:presentation}

In this section, we introduce the spine process, that can be thought of as a generalization of
the exploration process first defined by Le Gall and Le Jan in \cite{Le-Gall98:0}. 

The idea underlying the definition relies 
on the decomposition of the ``spine'' -- or ``ancestral line'' -- lying below the point of the tree corresponding to the birth of the $n$th individual.
In the $n$th step of the sequential construction presented on Figure \ref{fig:sequential-construction}, this corresponds to the path in the forest
starting from the root and reaching up to $n$ (which also corresponds to the right-most path in the planar forest constructed at step $n$).
As can be seen from the figure, this path 
is naturally decomposed into finitely many segments that correspond
to each ancestor's contribution to the spine.

The spine process at $n$ is then
defined as a sequence of measures that encodes this decomposition. 
More precisely, 
we start by
labeling
ancestors from highest to lowest. Then, 
the $k$th element of the spine process (evaluated at $n$) is simply the measure that
records the location of the stubs on the $k$th segment -- crosses on Figure \ref{fig:sequential-construction-with-spine} -- and the age of
the $k$th ancestor
upon giving birth to the $(k-1)$st ancestor -- circles on Figure \ref{fig:sequential-construction-with-spine}.

\begin{figure}[p!]
	\captionsetup{singlelinecheck=off}
	\centering
		\begin{tikzpicture}
			[color=black]
			% \foreach \n/\x/\shift in {0/0/-.5,1/0/1,2/.5/3,3/1/6,4/1.5/9}{
			% 	\begin{scope}[shift={(\shift,0)}]
			% 		\input{v3-pic/spine-\n}
			% 		\node[anchor=north] at (\x,0) {$n = \n$};
			% 	\end{scope}
			% }
			% \n/\x/\shift = 0/0/-.5
			\begin{scope}[shift={(-.5,0)}]

				\node[anchor=north] at (0,0) {$n = 0$};
			\end{scope}
			% \n/\x/\shift = 1/0/1
			\begin{scope}[shift={(1,0)}]
				% \input{v3-pic/spine-1}
				% root
				\draw (0,0) -- (0,2);
				\laststub{1.5}
				\stub{.5}
				\node[anchor=north] at (0,0) {$n = 1$};
			\end{scope}
			% \n/\x/\shift = 2/.5/3
			\begin{scope}[shift={(3,0)}]
				% \input{v3-pic/spine-2}
				% root
				\draw (0,0) -- (0,2);
				\ancestor{1.5}
				\draw[dashed] (0,1.5) -- (1,1.5);
				\begin{scope}[shift={(1,1.5)}]
					\draw (0,0) -- (0,1.5);
					\laststub{1.2}
					\stub{.5}
				\end{scope}
				\stub{.5}
				\node[anchor=north] at (.5,0) {$n = 2$};
			\end{scope}
			% \n/\x/\shift = 3/1/6
			\begin{scope}[shift={(6,0)}]
				% \input{v3-pic/spine-3}
				% root
				\draw (0,0) -- (0,2);
				\ancestor{1.5}
				% 1
				\draw[dashed] (0,1.5) -- (1,1.5);
				\begin{scope}[shift={(1,1.5)}]
					\draw (0,0) -- (0,1.5);
					\ancestor{1.2}
					% 2
					\draw[dashed] (0,1.2) -- (1,1.2);
					\begin{scope}[shift={(1,1.2)}]
						% 3
						\draw (0,0) -- (0,1.5);
						\laststub{.9}
					\end{scope}
					\stub{.5}
				\end{scope}
				\stub{.5}
				\node[anchor=north] at (1,0) {$n = 3$};
			\end{scope}
			% \n/\x/\shift = 4/1.5/9
			\begin{scope}[shift={(9,0)}]
				% \input{v3-pic/spine-4}
				% root
				\draw (0,0) -- (0,2);
				\ancestor{1.5}
				% 1
				\draw[dashed] (0,1.5) -- (1,1.5);
				\begin{scope}[shift={(1,1.5)}]
					\draw (0,0) -- (0,1.5);
					\done{1.2}
					% 2
					\draw[dashed] (0,1.2) -- (1,1.2);
					\begin{scope}[shift={(1,1.2)}]
						% 3
						\draw (0,0) -- (0,1.5);
						\done{.9}
						% 4
						\draw[dashed] (0,.9) -- (1,.9);
						\begin{scope}[shift={(1,.9)}]
							\draw (0,0) -- (0,1);
						\end{scope}
					\end{scope}
					\laststub{.5}
				\end{scope}
				\stub{.5}
				
				\node[anchor=north] at (1.5,0) {$n = 4$};
			\end{scope}

			% \foreach \n/\x/\shift in {5/2/0,6/2.5/5}{
			% 	\begin{scope}[shift={(\shift,-6)}]
			% 		\input{v3-pic/spine-\n}
			% 		\node[anchor=north] at (\x,0) {$n = \n$};
			% 	\end{scope}
			% }
			% \n/\x/\shift = 5/2/0
			\begin{scope}[shift={(0,-6)}]
				% \input{v3-pic/spine-5}
				% root
				\draw (0,0) -- (0,2);
				\done{1.5}
				% 1
				\draw[dashed] (0,1.5) -- (1,1.5);
				\begin{scope}[shift={(1,1.5)}]
					\draw (0,0) -- (0,1.5);
					\done{1.2}
					% 2
					\draw[dashed] (0,1.2) -- (1,1.2);
					\begin{scope}[shift={(1,1.2)}]
						% 3
						\draw (0,0) -- (0,1.5);
						\done{.9}
						% 4
						\draw[dashed] (0,.9) -- (1,.9);
						\begin{scope}[shift={(1,.9)}]
							\draw (0,0) -- (0,1);
						\end{scope}
					\end{scope}
					\done{.5}
					% 5
					\draw[dashed] (0,.5) -- (3,.5);
					\begin{scope}[shift={(3,.5)}]
						\draw (0,0) -- (0,2);
					\end{scope}
				\end{scope}
				\laststub{.5}
				\node[anchor=north] at (2,0) {$n = 5$};
			\end{scope}
			% \n/\x/\shift = 6/2.5/5
			\begin{scope}[shift={(5,-6)}]
				% \input{v3-pic/spine-6}
				% root
				\draw (0,0) -- (0,2);
				\done{1.5}
				% 1
				\draw[dashed] (0,1.5) -- (1,1.5);
				\begin{scope}[shift={(1,1.5)}]
					\draw (0,0) -- (0,1.5);
					\done{1.2}
					% 2
					\draw[dashed] (0,1.2) -- (1,1.2);
					\begin{scope}[shift={(1,1.2)}]
						% 3
						\draw (0,0) -- (0,1.5);
						\done{.9}
						% 4
						\draw[dashed] (0,.9) -- (1,.9);
						\begin{scope}[shift={(1,.9)}]
							\draw (0,0) -- (0,1);
						\end{scope}
					\end{scope}
					\done{.5}
					% 5
					\draw[dashed] (0,.5) -- (3,.5);
					\begin{scope}[shift={(3,.5)}]
						\draw (0,0) -- (0,2);
					\end{scope}
				\end{scope}
				\ancestor{.5}
				% 6
				\draw[dashed] (0,.5) -- (5,.5);
				\begin{scope}[shift={(5,.5)}]
					\draw (0,0) -- (0,4);
					\laststub{3.5}
					\stub{2.5}
					\stub{1}
				\end{scope}
				
				\node[anchor=north] at (2.5,0) {$n = 6$};
			\end{scope}

			% \foreach \n/\x/\shift in {10/4/1}{
			% 	\begin{scope}[shift={(\shift,-13.5)}]
			% 		\input{v3-pic/spine-\n}
			% 		\node[anchor=north] at (\x,0) {$n = \n$};
			% 	\end{scope}
			% }
			% \n/\x/\shift = 10/4/1
			\begin{scope}[shift={(1,-13.5)}]
				% \input{v3-pic/spine-10}
				% root
				\draw (0,0) -- (0,2);
				\done{1.5}
				% 1
				\draw[dashed] (0,1.5) -- (1,1.5);
				\begin{scope}[shift={(1,1.5)}]
					\draw (0,0) -- (0,1.5);
					\done{1.2}
					% 2
					\draw[dashed] (0,1.2) -- (1,1.2);
					\begin{scope}[shift={(1,1.2)}]
						% 3
						\draw (0,0) -- (0,1.5);
						\done{.9}
						% 4
						\draw[dashed] (0,.9) -- (1,.9);
						\begin{scope}[shift={(1,.9)}]
							\draw (0,0) -- (0,1);
						\end{scope}
					\end{scope}
					\done{.5}
					% 5
					\draw[dashed] (0,.5) -- (3,.5);
					\begin{scope}[shift={(3,.5)}]
						% 6
						\draw (0,0) -- (0,2);
					\end{scope}
				\end{scope}
				\done{.5}
				% 7
				\draw[dashed] (0,.5) -- (5,.5);
				\begin{scope}[shift={(5,.5)}]
					\draw (0,0) -- (0,4);
					\done{3.5}
					% 8
					\draw[dashed] (0,3.5) -- (1,3.5);
					\begin{scope}[shift={(1,3.5)}]
						\draw (0,0) -- (0,2);
					\end{scope}
					\done{2.5}
					% 9
					\draw[dashed] (0,2.5) -- (2,2.5);
					\begin{scope}[shift={(2,2.5)}]
						\draw (0,0) -- (0,1);
					\end{scope}
					\done{1}
					% 10
					\draw[dashed] (0,1) -- (3,1);
					\begin{scope}[shift={(3,1)}]
						\draw (0,0) -- (0,1);
						\done{1}
						% 11
						\draw[dashed] (0,1) -- (1,1);
						\begin{scope}[shift={(1,1)}]
							\draw (0,0) -- (0,1);
						\end{scope}
					\end{scope}
				\end{scope}
				\node[anchor=north] at (4,0) {$n = 10$};
			\end{scope}
		\end{tikzpicture}
	\caption[foo bar]{Same construction as in Figure~\ref{fig:sequential-construction}, but now with the spine highlighted in thick line. This allows to differentiate three kinds of atoms:
	\begin{description}[leftmargin=20pt]
		\item[Cross] represents a stub and corresponds to an atom on the spine whose subtree has not been explored yet;
		\item[Circle] represents an atom on the spine whose subtree is being explored;
		\item[Square] represents an atom whose subtree has been explored and that is no longer on the spine.
	\end{description}}
	\label{fig:sequential-construction-with-spine}
\end{figure}
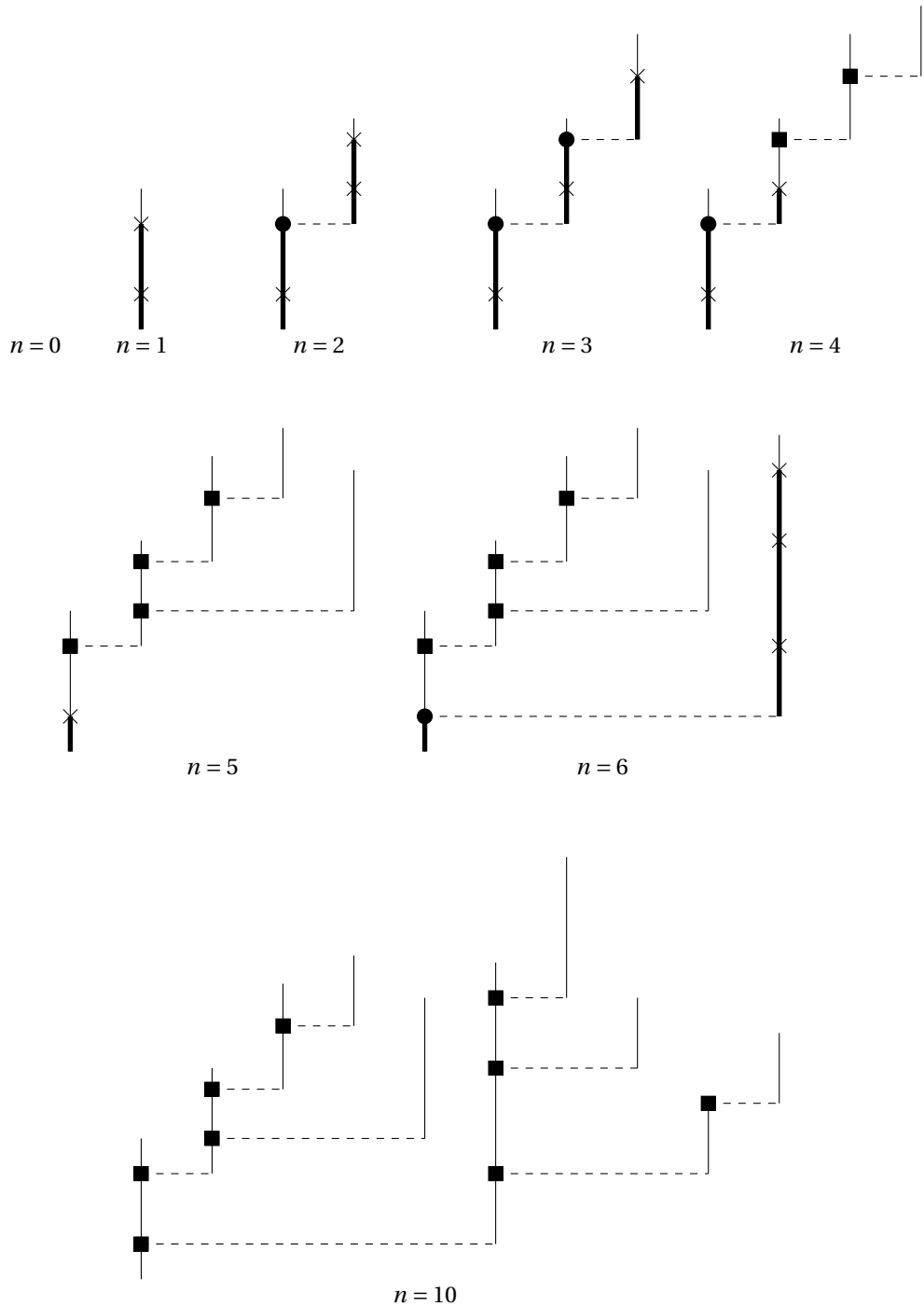

\subsection{Notation} Let $\Z$ denote the set of integers and $\N$ the set of non-negative integers. For $x \in \R$ let $[x] = \max\{n \in \Z: n \leq x\}$ and $x^+ = \max(x, 0)$ be its integer and positive parts, respectively. If $A \subset \R$ is a finite set we denote by $\lvert A \rvert$ its cardinality. Throughout we adopt the convention $\max \emptyset = \sup \emptyset = -\infty$, $\min \emptyset = \inf \emptyset = +\infty$ and $\sum_{k=a}^b u_k = 0$ if $b < a$, with $(u_k)$ any real-valued sequence.

\subsubsection{Measures}

Let $\cM$ be the set of finite point measures on $(0,\infty)$ endowed with the weak topology, $\epsilon_x \in \cM$ for $x > 0$ be the Dirac measure at $x$ and $\zero$ be the zero measure, the only measure with mass $0$. For a measure $\nu \in \cM$ we denote its mass by $\lvert \nu \rvert = \nu(0,\infty)$ and the supremum of its support by $\pi(\nu) = \inf \{ x > 0: \pi(x, \infty) = 0 \}$ with the convention $\pi(\zero) = 0$. For $k \in \N$ we define $\Upsilon_k(\nu) \in \cM$ as the measure obtained by removing the $k$ largest atoms of $\nu$, i.e., $\Upsilon_k(\nu) = \zero$ for $k \geq \lvert \nu \rvert$ and, writing $\nu = \sum_{i=1}^{\lvert \nu \rvert} \epsilon_{a(i)}$ with $0 < a(\lvert \nu \rvert) \leq \cdots \leq a(1)$, $\Upsilon_k(\nu) = \sum_{i=k+1}^{\lvert \nu \rvert} \epsilon_{a(i)}$ for $k = 0, \ldots, \lvert \nu \rvert-1$.

\subsubsection{Finite sequences of measures}

We let $\cM^* = \cup_{n \in \N} (\cM \setminus \{\zero\})^n$ be the set of finite sequences of non-zero measures in $\cM$. For $Y \in \cM^*$ we denote by $\length{Y}$ the only integer $n \in \N$ such that $Y \in (\cM \setminus \{\zero\})^n$, which we call the length of $Y$, and identify $\zero$ with the only sequence of length $0$. For two sequences $Y_1 = (Y_1(1), \ldots, Y_1(H_1))$ and $Y_2 = (Y_2(1), \ldots, Y_2(H_2))$ in $\cM^*$ with lengths $H_1, H_2 \geq 1$, we define $[Y_1, Y_2] \in \cM^*$ as their concatenation:
\[ [Y_1, Y_2] = \big( Y_1(1), \ldots, Y_1(H_1), Y_2(1), \ldots, Y_2(H_2) \big). \]
Further, by convention we set $[\zero, Y] = [Y, \zero] = Y$ for any $Y \in \cM^*$ and we then define inductively
\[ [Y_1, \ldots, Y_N] = \big[ [Y_1, \ldots, Y_{N-1}], Y_N \big], \ N \geq 2. \]
Note that, with these definitions, we have $\length{[Y_1, \ldots, Y_N]} = \length{Y_1} + \cdots + \length{Y_N}$ for any $N \geq 1$ and $Y_1, \ldots, Y_N \in \cM^*$.

Identifying a measure $\nu \in \cM \setminus \{\zero\}$ with the sequence of length one $(\nu) \in \cM^*$, the above definitions give sense to, say, $[Y, \nu]$ with $Y \in \cM^*$ and $\nu \in \cM \setminus \{\zero\}$. The operator $\pi$ defined on $\cM$ is extended to $\cM^*$ through the relation
\[ \pi(Y) = \sum_{k=1}^{\length{Y}} \pi(Y(k)), \ Y = (Y(1), \ldots, Y(\length{Y}) \in \cM^*. \]
Recalling the convention $\sum_{k=1}^0 = 0$, we see that $\pi(\zero) = 0$ and further, it follows directly from the above relation that $\pi([Y_1, \ldots, Y_N]) = \pi(Y_1) + \cdots + \pi(Y_N)$.

\subsubsection{Measurable space} \label{subsub:space}

We define $\L = \{ (v, \nu) \in (0,\infty) \times \cM: v \geq \pi(\nu) \}$ and call an element $s \in \L$ either a \emph{stick} or a \emph{life descriptor}. 
We work on the measurable space $(\Omega, \mathcal{F})$ with $\Omega = \L^\Z$ the space of doubly infinite sequences of sticks and $\mathcal{F}$ the $\sigma$-algebra generated by the coordinate mappings. An elementary event $\omega \in \Omega$ is written as $\omega = (\omega_n, n \in \Z)$ and $\omega_n = (V_n, \cP_n)$. For $n \in \Z$ we consider the three operators $\shift_n, \dualoperator^n, \cG: \Omega \to \Omega$ defined as follows:
\begin{itemize}
	\item $\shift_n$ is the shift operator, defined by $\shift_n(\omega) = (\omega_{n + k}, \in \Z)$;
	\item $\dualoperator^n$ is the dual (or time-reversal) operator, defined by $\dualoperator^n(\omega) = (\omega_{n - k - 1}, \in \Z)$;
	\item $\cG$ is the genealogical operator, mapping the sequence $((V_n, \cP_n), n \in \Z)$ to the sequence $((1, \lvert \cP_n \rvert \epsilon_1), n \in \Z)$.
\end{itemize}
We say that a mapping $\Gamma: \Omega \to {\mathfrak X}$ is a genealogical mapping if it is invariant by the genealogical operator, i.e., if $\Gamma \circ \cG = \Gamma$. The shift and dual operators are related by the following relations:
\begin{equation} \label{eq:sigma-dual}
	\dualoperator^m \circ \dualoperator^n = \shift_{n - m} \ \text{ and } \ \dualoperator^n \circ \shift_m = \dualoperator^{n+m}, \ m,n \in \Z,
\end{equation}
and for any random time $\Gamma: \Omega \to \Z$ we have
\begin{equation} \label{eq:dual-relations}
	\cP_\Gamma \circ \dualoperator^n = \cP_{n - 1 - \Gamma \circ \dualoperator^n}.
\end{equation}

\subsection{Spine, height and contour processes} \label{sub:exploration-process}

We now proceed to a formal definition of the various processes which will be studied.

\subsubsection{Spine process}
Consider the operator $\Phi: \cM^* \times \cM \to \cM^*$ defined for $\nu \in \cM$ and $Y = (Y(1), \ldots, Y(\length{Y}) \in \cM^*$ by
\begin{equation} \label{eq:Phi}
	\Phi(Y, \nu) = \begin{cases}
		[Y, \nu] & \text{ if } \nu \neq \zero,\\
		\big( Y(1), \ldots, Y(H-1), \Upsilon_1(Y(H)) \big) & \text{ if } \nu = \zero \ \text{ and } \ H \geq 1,\\
		\zero & \text{ else},
	\end{cases}
\end{equation}
where $H = \max \{k \geq 1: \lvert Y(k) \rvert \geq 2\}$. Note that by definition, we have $\Phi(Y, \nu) \in \cM^*$ for $Y \in \cM^*$ and $\nu \in \cM$ and that further, if $\nu \neq \zero$ then $\Phi(Y, \nu) \neq \zero$. Next, we consider the $\cM^*$-valued sequence $\spine_0 = (\spine^n_0, n \geq 0)$ (the subscript $0$ will be justified below, see~\eqref{eq:S^n_m}) defined recursively by
\begin{equation} \label{eq:dynamic-spine}
	\spine^0_0 = \zero \ \text{ and } \ \spine^{n+1}_0 = \Phi(\spine^n_0, \cP_n), \ n \geq 0.
\end{equation}
This dynamic is illustrated on Figure~\ref{fig:sequential-construction-with-spine}.
As already discussed in the introduction,
the $k$th
element of $\spine_0^n$ (ordered from top to bottom)
records (1) the location of the stubs on the $k$th segment in the spine decomposition
illustrated in Figure \ref{fig:sequential-construction-with-spine}, and (2) the
age of the $k$th ancestor (of $n$) when begetting the $(k-1)$st ancestor (identifying, for $k = 1$, the individual with its $0$th ancestor). In words, the recursive relation~\eqref{eq:dynamic-spine} encodes 
the fact that the birth event corresponding to the $(n+1)$st individual coincides with the next available stub
after grafting the $n$th stick on top of $\spine_0^n$. 
In particular, if no stub is available, a new spine is started from scratch (third relation).

We note that when $\spine_0^n\neq {\bf z}$, any element of the sequence $\spine_0^n$ contains at least one atom: the one corresponding
the birth of an ancestor, which is not counted as a {\it stub}: In particular, the condition $H = \max \{k \geq 1: \lvert Y(k) \rvert \geq 2\}$ in~\eqref{eq:Phi} reads ``look for the first available segment with a stub''.

\begin{remark}\label{rem:exploration}
The definition of the spine process is similar, but not completely analogous to the exploration
process of Le Gall and Le Jan in \cite{Le-Gall98:0}. Therein, the authors only consider the stubs attached to the spine.
However, in the chronological case, not only do we need to keep track of the number of available stubs, but one needs to also record 
the length of the segments carrying those stubs (in the discrete case, this is always equal to $1$). This is done by adding the additional atom
corresponding to the birth of the ``previous'' ancestor (when ancestors are labelled from top to bottom), and whose location coincides with the length of 
the corresponding segment.
\end{remark}

\subsubsection{Chronological height and contour processes} \label{subsub:chronological-processes}

We define the \emph{chronological height process} $\H = (\H(n), n \geq 0)$ by the relation
\[ \H(n) = \pi(\spine^n_0), \ n \geq 0. \]
Informally, $\H(n)$ is the birth time of the $n$th individual. We consider the associated \emph{chronological contour process}, which is the continuous-time process $\C = (\C(t), t \geq 0)$ with continuous sample paths defined inductively as follows. In the sequel, we define
\[ \cV(-1) = 0, \ \cV(n) = V_0 + \cdots + V_n \ \text{ and } \ K_n = 2 \cV(n-1) - \H(n), \ n \geq 0. \]

Note that the sequence $(K_n, n \geq 0)$ is non-decreasing, and we will assume that its terminal value is infinite.
(This assumption will hold a.s.\ for (sub)critical CMJ forests).
We start the initialization by setting $\C(K_0) = 0$. 
Assume that $\C$ has been built on $[0, K_n]$, we extend the construction to $[0, K_{n+1}]$ in the following way: $\C$ first increases at rate $+1$ up to $\H(n) + V_n$ and then decreases at rate $-1$ to $\H(n+1)$. Since $\H(n+1) \leq \H(n) + V_n$, this is well-defined and this extends the construction up to the time $K_n + 2V_n + \H(n) - \H(n+1) = K_{n+1}$ as desired.
 
It is not hard to prove that $\C$ is the usual contour process associated with the forest $\F^\infty$ seen as a forest of continuous trees. Indeed, our definition coincides with the usual definition of $\C(t)$ as the distance to the origin of a particle going up along the left side of an edge and going down along the right side, see for instance Le Gall~\cite{Le-Gall05:0} for a formal and general definition in the realm of real trees.

\subsubsection{Genealogical height and contour processes and exploration process} \label{subsub:genealogical-processes}

We define $\cH = \H \circ \cG$ and $\cC = \C \circ \cG$ which we call \emph{genealogical height and contour processes}, respectively, and $\rho^n_0 = \spine^n_0 \circ \cG$ the \emph{exploration process}. As explained in Remark~\ref{rem:exploration}, it is closely related to the classical exploration process introduced by Le Gall and Le Jan~\cite{Le-Gall98:0}.

\section{The spine process and the Lukasiewicz path.}
\label{sec:dec}

In this section, we relate the spine process to the well-known Lukasiewicz path. 
More precisely, the spine process is expressed in terms of a random functional of the weak ascending ladder height process associated to the dual Lukasiewicz path. This is the content of Proposition~\ref{prop:formula-spine} below. In a forthcoming section,
this result will allow us to express the one-dimensional marginal of the
spine process in terms of a bivariate renewal process, and will be instrumental
in proving our main scaling limit results for height and contour processes
in the short edges case.

\subsection{Lukasiewicz path} \label{subsub:Lukasiewicz}

We define the Lukasiewicz path $S = (S(n), n \in \Z)$ by $S(0) = 0$ and, for $n \geq 1$,
\[ S(n) = \sum_{k=0}^{n-1} \left( \lvert \cP_k \rvert - 1 \right) \ \text{ and } \ S(-n) = -\sum_{k=-n}^{-1} \left( \lvert \cP_k \rvert - 1 \right). \]
%%%%%%%%%%%%%%%%%%%%%%%%%%%%%%%
% S DIRECTLY IN TERMS OF \spine %
%%%%%%%%%%%%%%%%%%%%%%%%%%%%%%%
% \[ S(n) = \lvert \spine^n \rvert - \sum_{k = 1}^n \Indicator{\spine^k = \zero}, \ n \geq 1. \]
%%%%%%%%%%%%%%%%%%%%%%%%%%%%%%%
% END S DIRECTLY IN TERMS OF \spine %
%%%%%%%%%%%%%%%%%%%%%%%%%%%%%%%
Note that if $\Gamma$ is a random time, the dual operator acts as follows:
\begin{equation} \label{eq:S(Gamma)-dual}
	S(\Gamma) \circ \dualoperator^n = S(n) - S(n - \Gamma \circ \dualoperator^n), \ n \in \Z.
\end{equation}

It is well known that in the discrete case, the height process is directly related to the 
sequence of weak ascending ladder times. As we shall see, in the chronological case,
more structure of the ladder height process is needed.
In particular, the height (and spine) process
will be expressed not only in terms of the ladder height times, but also in terms of
the undershoot 
upon reaching the successive records of $S$ (through the quantity $\cQ$
defined below).
In order to make this more precise, 
we consider the following functionals associated to $S$, which will be used repeatedly in the rest of the paper:
\begin{itemize}
	\item the sequence of weak ascending ladder height times: $T(0)=0$ and for $k \geq 0$,
	\[ T(k+1) = \inf \big\{ \ell > T(k): S(\ell) \geq S(T(k)) \big\} = T(1) \circ \shift_{T(k)} + T(k); \]
	\item the hitting times upward and downward:
	\[ \tau_\ell = \inf \left\{ k > 0: S(k) \geq \ell \right\} \ \text{ and } \ \tau^-_\ell = \inf \left\{ k \geq 0: S(k) = -\ell \right\}, \ \ell \geq 0, \]
	so that in particular $\tau_0 = T(1)$;
	\item for $\ell \in \N$ with $\tau_\ell < \infty$,
	\[ \zeta_\ell = \ell - S(\tau_\ell-1) \ \text{ and } \ \mu_{\ell} = \Upsilon_{\zeta_\ell}(\cP_{\tau_\ell-1}),\]
	so that $\zeta_\ell$ is the undershoot upon reaching level $\ell$;
	\item and the backward maximum
	\[ L(m) = \max_{k = 0, \ldots, m} S(-k), \ m \geq 0. \]
	\end{itemize}
	Note that, since $S(\tau_0) \geq 0$, $\zeta_0 = - S(\tau_0 - 1) \leq S(\tau_0) - S(\tau_0 - 1) = \lvert \cP_{\tau_0 - 1} \rvert - 1$, so that $\mu_0 \neq \zero$. We will pay special attention to the following functionals of the ladder height process:
	\begin{itemize}
	\item for $k \geq 1$ with $T(k) < \infty$,
	\[ \cQ(k) = \mu_0\circ \shift_{T(k-1)} \ \text{ and } \ \y(k) = \pi \circ \cQ(k); \]
	\item the following two inverses associated to the sequence $(T(k), k \geq 0)$:
	\[ T^{-1}(n) = \min \left\{ k \geq 0: T(k) \geq n \right\} \ \text{ and } \ \bTm(n) = \max \left\{ k \geq 0: T(k) \leq n \right\}, \ n \geq 0. \]
\end{itemize}
The fact that $\mu_0 \neq \zero$ implies that $\cQ(k) \neq \zero$ whenever it is well-defined, a simple fact that will be used later on. If $n$ is a weak ascending ladder height time, then $\bTm(n) = T^{-1}(n)$ with $T(\bTm(n)) = n = T(T^{-1}(n))$, while if $n$ is not a weak ascending ladder height time, then $\bTm(n) +1 = T^{-1}(n)$ with $T(\bTm(n)) < n < T(T^{-1}(n))$. Define
\[ \cA(n) = \left\{ n - T(k): k \geq 0 \right\} \circ \vartheta^n, \ n \geq 0. \]
It is well-known that $\cA(n) \cap \R_+$ is the set of $n$'s ancestors, see for instance Duquesne and Le Gall~\cite{Duquesne02:0}.
This property relates the height process and the weak ascending ladder height times $T$ through the following identity:
\begin{equation} \label{eq:identity-H}
	\cH(n) = \bTm(n) \circ \dualoperator^n, \ n \geq 0.
\end{equation}
The genealogical height is also given by the length of $\spine^n_0$ as we show now.
\begin{lemma}
	For any $n \geq 0$ we have $\length{\spine^n_0} = \cH(n)$.
\end{lemma}

\begin{proof}
	As highlighted in Remark \ref{rem:exploration}, 
	the exploration process $\rho^n_0 = \spine^n_0 \circ \cG$ 
	slightly differs from the classical definition of the exploration process in Le Gall and Le Jan~\cite{Le-Gall98:0}: however, this slight difference does not alter the length of the sequence, which remains unchanged between the two definitions.
	
	Since the length of the sequence in the classical exploration process coincides with the height process, this implies that
	$\length{\rho^n_0} = \cH(n)$. Thus, $\length{\spine^n_0} = \length{\spine^n_0} \circ \cG = \length{\spine^n_0 \circ \cG}$,
	which proves the desired result.
\end{proof}

Define
\[ \mrca{m}{n} = \max \big( \cA(m) \cap \cA(n) \big), \ m, n \geq 0. \]
Then $\mrca{m}{n} \in \Z$ and $m$ and $n$ have an ancestor in common (i.e., belong to the same tree) if and only if $\mrca{m}{n} \geq 0$ in which case $\mrca{m}{n}$ is the lexicographic index of their most recent common ancestor -- see for instance~\cite{Duquesne02:0}. We end this section by listing the following identities, which are proved in the Appendix~\ref{appendix:proof-useful-identities}. The second identity involves the condition $L(n-m) \circ \dualoperator^m > 0$: it is readily checked that
\begin{equation} \label{eq:formula-L}
	L(n-m) \circ \dualoperator^m = S(m) - \min_{\{m, \ldots, n\}} S, \ 0 \leq m \leq n,
\end{equation}

\begin{lemma} \label{lemma:condition-mrca}
	For any $0 \leq m \leq n$, $m$ is an ancestor of $n$, i.e., $\mrca{m}{n}=m$, if and only if $L(n-m)\circ\dualoperator^m=0$.
\end{lemma}

\begin{lemma}\label{lemma:useful-identities}
	For any $n \geq m \geq 0$ with $L(n-m) \circ \dualoperator^m > 0$, we have
	\begin{equation} \label{eq:identity-mrca}
		\mrca{m}{n} = n - T(T^{-1}(n-m)) \circ \dualoperator^n = m - \tau_{L(n-m)} \circ \dualoperator^m
	\end{equation}
	and
	\begin{equation}\label{eq:R-mrca}
		\cQ(T^{-1}(n-m)) \circ \dualoperator^n = \mu_{L(n-m)} \circ \dualoperator^m.
	\end{equation}
\end{lemma}

\subsection{Fundamental formula for $\H(n)$}

As mentioned earlier, $\cA(n) \cap \R_+$ is the set of ancestors of $n$. More precisely, $n-T(k)\circ\dualoperator^n$ is the index of the $k$th ancestor of $n$, assuming that ancestors are ordered from highest to lowest date of birth (or height). Further, interpreting $\y(k) \circ \dualoperator^n$ as the age of the $k$th ancestor when giving birth to the $(k{-}1)$st ancestor motivates the following result.

\begin{prop} \label{prop:formula-H}
	For every $n \geq 0$, we have
	\begin{equation} \label{eq:formula-H}
		\big( \cH(n), \ \H(n) \big) = \left( \bTm(n) \ , \ \sum_{k = 1}^{\bTm(n)} \y(k) \right) \circ \dualoperator^n.
	\end{equation}
\end{prop}

Since $\H(n) = \pi(\spine^n_0) = \sum_{k=1}^{\length{\spine^n_0}} \pi(\spine^n_0(k))$, 
Proposition \ref{prop:formula-H}
is an immediate corollary of the following result which proves a more general relation between the spine process and the Lukasiewicz path.

\begin{prop} \label{prop:formula-spine}
	We have
	\begin{equation} \label{eq:formula-spine}
		\spine^n_0 = \left( \cQ(\bTm(n)), \ldots, \cQ(1) \right) \circ \dualoperator^n, \ n \geq 0.
	\end{equation}
\end{prop}

The rest of this section is devoted to proving Proposition ~\ref{prop:formula-spine}. We prove it through several lemmas, several of which will be used in the sequel. To prove these results, for $m \geq 0$ and $k \in \{0, \ldots, \lvert \cP_m \rvert\}$ we introduce
\[ \chi(m,k) = \tau^-_k \circ \theta_{m+1} + m+1 = \inf \left\{ n \geq m+1: S(n) = S(m+1) - k \right\} \]
and define $\chi(m) = \chi(m,\lvert \cP_m \rvert)$ so that
\[ \chi(m) = \inf \left\{ n \geq m+1: S(n) = S(m+1) - \lvert \cP_m \rvert \right\} = \inf \left\{ n \geq m+1: S(n) = S(m) - 1 \right\} \]
which is also equal to $\tau^-_1 \circ \theta_m + m$. Intuitively, for $k\in\{0,\cdots,|\cP_m|-1\}$, $\chi(m,k)$ corresponds to the index of $(k+1)$st child of the $m$th individual (with the convention that children are ranked from youngest to oldest); whereas $\chi(m)$ is the index of the highest stub on $\S_0^m$ (i.e., right before attaching the $m$th individual). In particular, any individual 
$n\in\{m+1,\dots,\chi(m)-1\}$ belongs to a subtree attached to $m$. In view of this interpretation,
the two following lemmas seem quite natural. (On Figure \ref{fig:sequential-construction-with-spine}, 
and $\chi(1)=4$). For the proof of Lemma~\ref{lemma:snake-3} we will need the following identity, whose proof is defered to Appendix~\ref{appendix:proof-identity-chi}.

\begin{lemma}\label{lemma:chi}
	Let $n \geq 0$, $m = n - \tau_0 \circ \dualoperator^n$ and $i = \zeta_0 \circ \dualoperator^n$. If $m \geq 0$, then it holds that $i \in \{0, \ldots, \lvert \cP_m \rvert - 1 \}$ and $\chi(m,i) = n$.
\end{lemma}

\begin{lemma} \label{lemma:snake-1}
	For any $m \geq 0$ such that $|\cP_m|>0$, $n \in \{m+1, \ldots, \chi(m)-1\}$ and $\ell \in \{1, \ldots, \cH(m)\}$ we have
	\[ \cH(n) > \cH(m) \ \text{ and } \ \spine^n_0(\ell) = \spine^m_0(\ell). \]
\end{lemma}

\begin{proof}
	Let $\ell, m$ and $n$ be as in the statement: we first prove that $\cH(n) > \cH(m)$. Since $S$ only makes negative jumps of size $-1$, we have by definition of $\chi(m)$
	\[ \min_{\{m+1, \ldots, \chi(m)-1\}} S \geq S(m). \]
	This inequality implies that, since $n \in \{m+1, \ldots, \chi(m)-1\}$, there is at least one more ladder height time for the dual Lukasiewicz process seen from $n$ as compared to the dual Lukasiewicz process seen from $m$. In view of the relation~\eqref{eq:identity-H} which expresses $\cH(n) = \bTm(n) \circ \dualoperator^n$ as the number of weak ascending ladder height times of the dual Lukasiewicz process, this means precisely that $\cH(n) > \cH(m)$.
	
	We now prove that $\spine^n_0(\ell) = \spine^m_0(\ell)$. Since $n \in \{m+1, \ldots, \chi(m)-1\}$, in order to prove this it is enough to prove that $\chi' \geq \chi(m)$ where we define
	\[ \chi' = \inf \left\{ k \geq m+1: \spine^k_0(\ell) \neq \spine^m_0(\ell) \right\}. \]
	In view of the definition~\eqref{eq:Phi} of $\Phi$ and the dynamic~\eqref{eq:dynamic-spine}, we see that the 
	$\ell$th element of the spine between $m$ and $n$ is modified
	only if the length of the spine goes below $\ell$ between $m$ and $n$.
	Since the length of the spine coincides with $\cH$, this implies
	$\cH(\chi') = \ell \leq \cH(m)$. Finally, since $\cH(m) < \min_{\{m+1, \ldots, \chi(m)-1\}} \cH$, this implies that $\chi' \geq \chi(m)$ and concludes the proof.
\end{proof}

\begin{lemma} \label{lemma:snake-2}
	For $m \geq 0$ such that $|\cP_m|>0$ and $k \in \{0, \ldots, \lvert \cP_m \rvert - 1 \}$ we have
	\[ \cH(\chi(m, k)) = \cH(m)+1 \ \text{ and } \ \spine^{\chi(m,k)}_0(\cH(m)+1) = \Upsilon_k(\cP_m). \]
\end{lemma}

\begin{proof}
	By definition of $\chi(m,k)$ and the fact that $S$ only makes jumps of negative size $-1$, we have
	\[ S(\chi(m,k)) = \min_{m+1, \ldots, \chi(m,k)} S \geq S(m). \]
	A similar argument as in the proof of the previous lemma then leads to the conclusion $\cH(\chi(m,k)) = \cH(m)+1$ 
	(i.e., by showing that there is exactly one extra ladder height time for the dual walk seen from $\chi(m,k)$).
	
	We now prove that $\spine^{\chi(m,k)}_0(\cH(m)+1) = \Upsilon_k(\cP_m)$. 
	For $k = 0$ this is seen to be true by looking at the dynamic (\ref{eq:dynamic-spine}). We now prove that this is true by induction: so assume this is true for $k \in \{0, \ldots, \lvert \cP_m \rvert - 2\}$ and let us prove that this continues to hold for $k+1$. In order to do so, it is sufficient to combine the induction hypothesis with the following claim:
	\[ \spine^{\chi(m,k+1)}_0(\cH(m)+1) = \Upsilon_1 \left( \spine^{\chi(m,k)}_0(\cH(m)+1) \right). \]
	In order to prove this identity, we first note
 that (again, this is seen by comparing the number of ladder height times of the dual processes seen from the two times)
 \[
 \cH(n) > \cH(\chi(m,k))=\cH(m)+1 \ \mbox{for} \ n = \chi(m,k)+1, \ldots, \chi(m,k+1)-1.
 \] 
 Finally, we already know that $\cH(\chi(m,k+1)) = \cH(m)+1$. 
 From the dynamic~\eqref{eq:dynamic-spine}, this implies that
 the $(\cH(m)+1)$st element of $\spine_0^n$ remains unchanged
 for $n = \chi(m,k)+1, \ldots, \chi(m,k+1)-1$, but that one stub is removed 
 at time $\chi(m,k+1)$, i.e.,
	\[ \spine^{\chi(m,k+1)}_0(\cH(m)+1) = \Upsilon_1 \left( \spine^{\chi(m,k)}_0(\cH(m)+1) \right). \]
	This proves the claim made earlier and ends the proof of Lemma \ref{lemma:snake-2}.
\end{proof}

\begin{lemma} \label{lemma:snake-3}
	For any $n, k \geq 0$ with $T(k) \circ \dualoperator^n \leq n$ we have
	\[ \spine^n_0 = \left[ \spine^{n - T(k) \circ \dualoperator^n}_0, \cQ(k) \circ \dualoperator^n, \ldots, \cQ(1) \circ \dualoperator^n \right]. \]
\end{lemma}

Recall the convention $[\zero, Y] = Y$ for any $Y \in \cM^*$: in particular,
\[ \left[ \spine^{n - T(k) \circ \dualoperator^n}_0, \cQ(k) \circ \dualoperator^n, \ldots, \cQ(1) \circ \dualoperator^n \right] = \left[ \cQ(k) \circ \dualoperator^n, \ldots, \cQ(1) \circ \dualoperator^n \right]\]
when $\spine^{n - T(k) \circ \dualoperator^n}_0={\bf z}$.

\begin{proof} [Proof of Lemma~\ref{lemma:snake-3}]
	Let us first prove the result for $k = 1$, so we consider $n \geq 0$ with $\tau_0 \circ \dualoperator^n \leq n$ and we prove that $\spine^n_0 = [ \spine^{n - T(1) \circ \dualoperator^n}_0, \cQ(1) \circ \dualoperator^n ]$. Combining the two previous lemmas, we see that
	\[ \spine^{\chi(m,i)}_0 = \left[ \spine^m_0, \Upsilon_i(\cP_m) \right] \]
	for any $m \geq 0$ and any $i \in \{0, \ldots, \lvert \cP_m \rvert - 1\}$. In particular, Lemma~\ref{lemma:chi} shows that we can apply this to $m = n - \tau_0 \circ \dualoperator^n$ and $i = \zeta_0 \circ \dualoperator^n$, which gives
	\[ \spine^{\chi(m,i)}_0 = \left[ \spine^{n - \tau_0 \circ \dualoperator^n}_0, \Upsilon_{\zeta_0 \circ \dualoperator^n}(\cP_{n - \tau_0 \circ \dualoperator^n}) \right]. \]
	On the one hand, we have $\chi(m,i) = n$ (again by Lemma~\ref{lemma:chi}) and so in particular $\spine^{\chi(m,i)}_0 = \spine^n_0$, while on the other hand, we have
	\[ \Upsilon_{\zeta_0 \circ \dualoperator^n}(\cP_{n - \tau_0 \circ \dualoperator^n}) = \Upsilon_{\zeta_0}(\cP_{\tau_0 - 1}) \circ \dualoperator^n = \cQ(1) \circ \dualoperator^n. \]
	Combining the above arguments concludes the proof for $k = 1$. The general case follows by induction left to the reader.
\end{proof}

We can now prove Proposition~\ref{prop:formula-spine}.

\begin{proof} [Proof of Proposition~\ref{prop:formula-spine}]
	By definition, $T(\bTm(n)) \leq n$ and so Lemma~\ref{lemma:snake-3} with $k = \bTm(n)$ yields
	\[ \spine^{n}_0 = \left[ \spine_0^{n-T(\bTm(n))\circ\dualoperator^n}, \cQ(\bTm(n))\circ\dualoperator^n, \cdots, \cQ(1)\circ\dualoperator^n \right]. \]
	Since $\cQ(k) \neq \zero$ whenever it is well-defined, in particular for $k \in \{1, \ldots, \bTm(n)\}$, it follows that
	\[ \length{\spine^n_0} = \bTm(n) \circ \dualoperator^n + \length{\spine_0^{n-T(\bTm(n))\circ\dualoperator^n}}. \]
	However, $\length{\spine^n_0} = \bTm(n) \circ \dualoperator^n$ by~\eqref{eq:identity-H}, and thus $\length{\spine_0^{n-T(\bTm(n))\circ\dualoperator^n}} = 0$ which means $\spine_0^{n - T(\bTm(n))\circ\dualoperator^n} = {\bf z}$. This achieves the proof of Proposition~\ref{prop:formula-spine}.
\end{proof}

\subsection{Right decomposition of the spine}
\label{sub:right-decomposition}
In the case of i.i.d.\ life descriptors, the spine process is easily seen to be a Markov process.
In the forthcoming Section \ref{sec:proba:spine}, Proposition \ref{prop:formula-spine} will allow us to express the one-dimensional marginal of this process 
in terms of a bivariate renewal process. 

The present section can be seen as a description of the transition probabilities of the spine process:
we show that for $m\leq n$, the spine at $n$ is deduced from the spine at $m$
by truncating $\spine_0^m$ and then by concatenating a spine that is independent of the past up to $m$, a
construction reminiscent of the snake property -- see Duquesne and Le Gall~\cite{Duquesne02:0}.
As we shall now see, the independent ``increment'' will be given by
\begin{equation} \label{eq:S^n_m}
	 \spine_m^n := \spine_0^{n-m}\circ \shift_m, \ 0 \leq m \leq n,
 \end{equation}
	 which, when life descriptors are i.i.d., is distributed as the original spine at time $n-m$. In particular, since $\pi(\spine^n_m) = \pi(\spine_0^{n-m}\circ \shift_m) = \pi(\spine_0^{n-m})\circ \shift_m = \H(n-m) \circ \shift_m$, we note that an immediate consequence of~\eqref{eq:sigma-dual} and~\eqref{eq:formula-H} is that
\begin{equation} \label{eq:formula-pi(rhonm)}
	\pi(\spine^n_m) = \left( \sum_{k = 1}^{\bTm(n-m)} \y(k) \right) \circ \dualoperator^n, \ 0 \leq m \leq n.
\end{equation}

\begin{prop} \label{prop:snake-det}
	Let $n \geq m \geq 0$. If $\mrca{m}{n} \geq 0$, then $\spine^n_0 = [\spine^{\mrca{m}{n}}_0, \spine^n_{\mrca{m}{n}}]$ and
	\begin{equation} \label{eq:snake-det}
		\spine^n_{\mrca{m}{n}} = \begin{cases}
			\left[ \mu_{L(n-m)}\circ\dualoperator^m , \spine^n_{m} \right] & \text{ if } {L(n-m)}\circ\dualoperator^m > 0,\\
			\spine^n_{m} & \text{ else.}
		\end{cases}
	\end{equation}	
\end{prop}

In order to prove Proposition \ref{prop:snake-det}, we will need the following lemma.

\begin{lemma} \label{cor:two-S}
	For any $n \geq m \geq 0$ we have
	\begin{equation} \label{sp-1}
		\spine^n_m = \left( \cQ(\bTm(n-m)), \ldots, \cQ(1) \right) \circ \dualoperator^n.
	\end{equation}
	If in addition $\mrca{m}{n} \geq 0$, then
	\begin{equation} \label{sp-2}
		\spine_{\mrca{m}{n}}^n = \left( \cQ\circ T^{-1}(n-m),\ldots,\cQ(1) \right) \circ \dualoperator^n.
	\end{equation}
\end{lemma}

\begin{proof}
	By definition we have $\spine^n_m = \spine^{n-m}_0 \circ \shift_m$ and so Proposition~\ref{prop:formula-spine} implies that
	\begin{equation} \label{eq:manipulation-1}
		\spine^n_m = \left( \cQ(\bTm(n-m)), \ldots, \cQ(1) \right) \circ \dualoperator^{n-m} \circ \shift_m.
	\end{equation}
	The first relation~\eqref{sp-1} thus follows from the identity $\dualoperator^{n-m} \circ \shift_m = \dualoperator^n$ of~\eqref{eq:sigma-dual}. To prove the other relation~\eqref{sp-2}, we use~\eqref{sp-1} with $m$ random, which in this case reads as follows: for any random time $\Gamma$, the relation
	\begin{equation} \label{eq:manipulation}
		\spine^n_{\Gamma \circ \dualoperator^n} = \left( \cQ(\bTm(n-\Gamma)), \ldots, \cQ(1) \right) \circ \dualoperator^n.
	\end{equation}
	holds in the event $0 \leq \Gamma \circ \dualoperator^n \leq n$ (see remark below). Apply now this relation to $\Gamma = n - T(T^{-1}(n-m))$, so that $\mrca{m}{n} = \Gamma \circ \dualoperator^n$ by~\eqref{eq:identity-mrca}. Then we always have $\Gamma \leq n$ and so under the assumption $\mrca{m}{n} \geq 0$, we obtain
	\[ \spine^n_{\mrca{m}{n}} = \left( \cQ(\bTm(T(\Gamma'))), \ldots, \cQ(1) \right) \circ \dualoperator^n \]
	with $\Gamma' = T^{-1}(n-m)$. Since $\bTm(T(k)) = k$ for any $k \geq 0$, we obtain the result.
\end{proof}

\begin{remark} \label{rk:manipulation}
	Let us comment on~\eqref{eq:manipulation} as similar identities will be used in the sequel. To see how it follows from~\eqref{eq:manipulation-1}, write~\eqref{eq:manipulation-1} in the form $\spine^n_m = (U \circ \dualoperator^n) (m)$ for some mapping $U$ with domain $\Omega$ and values in the space of $\cM^*$-valued sequence, so that $(U \circ \dualoperator^n) (m)$ is the $m$th element of the dual sequence. With this notation, we can directly plug in a random time, i.e., if $m = \Gamma$ is random then we have $\spine^n_\Gamma = (U \circ \dualoperator^n) (\Gamma)$ and in particular, $\spine^n_{\Gamma \circ \dualoperator^n} = (U \circ \dualoperator^n) (\Gamma \circ \dualoperator^n) = U(\Gamma) \circ \dualoperator^n$.
\end{remark}

\begin{proof} [Proof of Proposition~\ref{prop:snake-det}]

By~\eqref{eq:identity-mrca}, $\mrca{m}{n}\geq0$ implies that $T(T^{-1}(n-m))\circ\vartheta^n \leq n$ 
and so Lemma \ref{lemma:snake-3} with $k=T^{-1}(n-m)$ gives
$$
\spine_0^n = \left[\spine_0^{n-T(T^{-1}(n-m))\circ\vartheta^n}, \cQ(T^{-1}(n-m))\circ\vartheta^n, \ldots, \cQ(1)\circ\vartheta^n \right]. 
$$
Combining~\eqref{eq:identity-mrca}, which shows that $\spine_0^{n-T(T^{-1}(n-m))\circ\vartheta^n} = \spine^{\mrca{m}{n}}_0$, and the expression for $\spine^n_{\mrca{m}{n}}$ given in~\eqref{sp-2} under the assumption $\mrca{m}{n} \geq 0$ gives the first part of the result, namely that $\spine^n_0 = \left[ \spine^{\mrca{m}{n}}_0,\ \spine^n_{\mrca{m}{n}} \right]$. In order to show~\eqref{eq:snake-det} and thus complete the proof, we distinguish between the two cases $L(n-m) \circ \dualoperator^m = 0$ and $L(n-m) \circ \dualoperator^m > 0$. 

If $L(n-m) \circ \dualoperator^n = 0$, then $\mrca{m}{n} = m$ according to Lemma~\ref{lemma:condition-mrca} which proves~\eqref{eq:snake-det}.

Assume now that $L(n-m) \circ \dualoperator^n > 0$: in view of~\eqref{eq:formula-L}, this means that $n-m$ is not a weak ascending ladder height time of $S \circ \dualoperator^n$ and so $T^{-1}(n-m)\circ\vartheta^n = \bTm(n-m)\circ\vartheta^n + 1$. We then obtain by Lemma~\ref{cor:two-S} the relation $\spine^n_{\mrca{m}{n}} = [\cQ(T^{-1}(n-m)) \circ \dualoperator^n, \spine^n_m]$ and since $\cQ(T^{-1}(n-m)) \circ \dualoperator^n = \mu_{L(n-m)} \circ \dualoperator^m$ in this case by~\eqref{eq:R-mrca}, we obtain the result.
\end{proof}

\subsection{Probabilistic description of the spine}\label{sec:proba:spine}

In this paper we are interested in the chronological height and contour processes associated to CMJ forests, which corresponds to the case where the planar forest is constructed from an i.i.d.\ sequence of sticks. Formally, let $(V^*, \cP^*)$ be a random variable with values in $\L$, and let $\P$ be the probability distribution on $\Omega$ such that $\omega$ under $\P$ is i.i.d.\ with common distribution $(V^*, \cP^*)$. In this paper we consider the subcritical and critical cases, i.e., we assume that
\[ \E(\lvert \cP^* \rvert) \leq 1. \]

Under this (sub)critical assumption, $S$ under $\P$ is a random walk with step distribution $\lvert \cP^* \rvert - 1$, which therefore does not drift to $+\infty$. In particular, all the trees considered in the informal sequential construction of the Introduction are finite and the sequence $K_n$ almost surely grows to $\infty$.

In this case, for any $n \in \Z$ the dual operator $\dualoperator^n$ leaves $\P$ invariant, i.e., $\P = \P \circ (\dualoperator^n)^{-1}$. In the rest of the paper, this property will be called \emph{duality}, it implies for instance that $S$ and $S \circ \dualoperator^n$ under $\P$ are equal in distribution, and the same goes with $\cH(m)$ and $\bTm(m) \circ \dualoperator^n$ for any $m, n \geq 0$.
\\

The fundamental result which makes it possible to study the asymptotic behavior of the height process is the following lemma. It entails in particular that
\[ \left( \big(T(k), \sum_{i=0}^k \y(i) \big), k \geq 1 \right) \]
under $\P$ is a bivariate renewal process stopped at some independent geometric random variable, which thus describes the law of $(\cH(n), \H(n))$ in view of~\eqref{eq:formula-H}. Recall that $\tau^-_\ell = \inf \{k \geq 0: S(k) = -\ell\}$ for $\ell \geq 0$.

\begin{lemma} \label{lemma:y}
	Let $G = \inf\{k \geq 0: T(k) = \infty\}$. Then under $\P$, the sequence
	\[ \left( \big(T(k) - T(k-1), \cQ(k) \big), k = 1, \ldots, G-1 \right) \]
	is equal in distribution to $(( T^*(k), \cQ^*(k)), k = 1, \ldots, G^*-1)$, where the random variables $(( T^*(k), \cQ^*(k)), k \geq 1)$ are i.i.d.\ with common distribution $( T^*, \cQ^*)$ satisfying
	\begin{equation} \label{eq:def-y^*}
		\E \left[ f(\cQ^*) g( T^* ) \right] = \frac{1}{\E(\lvert \cP^* \rvert)} \sum_{t\geq 1} \sum_{x\geq0} \E\left[ f \circ \Upsilon_x(\cP^*) ; \lvert \cP^* \rvert \geq x + 1 \right] g(t) \ \P\left( \tau^-_x = t - 1 \right)
	\end{equation}
	for every bounded and measurable functions $f: \cM \to \R_+$ and $g: \Z_+ \to \R_+$, and $G^*$ is an independent geometric random variable with parameter $1 - \E(\lvert \cP^* \rvert)$.
\end{lemma}

By duality, this result describes the law of $((T(k) - T(k-1), \cQ(k)), k < G) \circ \dualoperator^n$ under $\P$ and justifies the claim made before the statement of the lemma. By combining this result with the spine decomposition of Proposition~\ref{prop:formula-H}, we thus get that the genealogical height process at a fixed time can be expressed as a functional of an explicit bivariate renewal process.
\\

Moreover, we note that the random variable $\y^* = \pi(\cQ^*)$ admits a natural interpretation. Indeed, the previous result implies that
\begin{equation} \label{eq:law-y^*}
	\E \left[ f(\y^*) \right] = \sum_{k=1}^\infty \ \sum_{r=0}^{k-1} \frac{1}{k} \E \left[ f \circ \pi \circ \Upsilon_r(\cP^*) \mid \lvert \cP^* \rvert = k \right] \times \frac{k\P(\lvert \cP^* \rvert=k)}{\E(\lvert \cP^* \rvert)}.
\end{equation}

Identifying $( k \P(\lvert \cP^* \rvert =k)/\E(\lvert\cP^*\rvert), k \geq 0)$ as the size-biased distribution of $\lvert \cP^* \rvert$, we see that if we bias the life descriptor $\cP^*$ by its number of children, then $\y^*$ is the age of the individual when its begets a randomly chosen child. As mentioned in the introduction, in the critical case $\E(\y^*) = 1$, the random variable $\y^*$ and its genealogical interpretation can already be found in Nerman~\cite{Nerman84:0}.

\begin{proof}[Proof of Lemma~\ref{lemma:y}] 
	The strong Markov property implies that $G$ is a geometric random variable with parameter $\P(\tau_0=T(1) = \infty)$ and that conditionally on $G$, the random variables $((T(k)-T(k-1),\cQ(k)), k =1, \ldots, G-1)$ are i.i.d.\ with common distribution $(\tau_0, \cQ(1))$ conditioned on $\{\tau_0 < \infty\}$. Thus in order to prove Lemma~\ref{lemma:y}, we only have to show that $(\tau_0, \cQ(1))$ under $\P( \, \cdot \mid \tau_0<\infty)$ is equal in distribution to $( T^*, \cQ^*)$. Recalling that $\cQ(1) = \Upsilon_{\zeta_0}(\cP_{\tau_0-1})$, we will actually show a more complete result and characterize the joint distribution of $(\cP_{\tau_0-1}, \tau_0, \zeta_0)$ under $\P( \, \cdot \mid \tau_0<\infty)$.
	
	Fix in the rest of the proof $x, t \in \N$ with $t \geq 1$ and $h: \cM \to [0,\infty)$ measurable: we will prove that
	\begin{equation} \label{erb0}
		\E \left[ h\left(\cP_{\tau_0-1} \right) \indicator{\zeta_0 = x} \indicator{\tau_0 = t} \right] = \E\left[ h(\cP^*) ; \lvert \cP^* \rvert \geq x + 1 \right] \P\left( \tau^-_x = t - 1 \right).
	\end{equation}
	By standard arguments, this characterizes the law of $(\cP_{\tau_0-1}, \tau_0, \zeta_0)$ and implies for instance that for any bounded measurable function $F: \cM \times \N \times \N \to [0,\infty)$, we have
	\begin{multline*}
		\E \left[ F\left(\cP_{\tau_0-1}, \zeta_0, \ \tau_0 \right) \mid \tau_0 <\infty \right]\\
		= \frac{1}{\P(\tau_0<\infty)} \sum_{t\geq 1} \sum_{x \geq 0} \E\left[ F(\cP^*,x,t) ; \lvert \cP^* \rvert \geq x + 1 \right] \P\left( \tau^-_x = t - 1 \right).
	\end{multline*}
	Since $\tau^-_x$ is $\P$-almost surely finite, the above relation for $F(\nu, x, t) = 1$ entails the relation $\P(\tau_0 < \infty) = \E(\lvert \cP^* \rvert)$ which implies in turn the desired result by taking $F(\nu, x, t) = f(\Upsilon_x(\nu)) g(t)$. Thus we only have to prove~\eqref{erb0}, which we do now. First of all, note that if
	\[ B = \big\{ S(t-1) = -x \ \text{ and } \ S(k) < 0 \ \text{ for } \ k = 1, \ldots, t-1 \big\}, \]
	then the two events $\{\zeta_0 = x, \tau_0 = t\}$ and $B \cap \{ \lvert \cP_{t-1} \rvert \geq x+1\}$ are equal. It follows from this observation that
	\[ \E\left[ h(\cP_{\tau_0-1}) \indicator{\zeta_0 = x} \indicator{\tau_0 = t} \right] = \E\left[ h(\cP_{t-1}) \indicator{\lvert \cP_{t-1} \rvert \geq x+1} ; B \right] \]
	and since $\cP_{t-1}$ and the indicator function of the event $B$ are independent and $\cP_{t-1}$ under $\P$ is equal in distribution to $\cP^*$, we obtain
	\[ \E\left[ h(\cP_{\tau_0-1}) \indicator{\zeta_0 = x} \indicator{\tau_0= t} \right] = \E\left[ h(\cP^*) ; \lvert \cP^* \rvert \geq x+1 \right] \P(B). \]
	Since $\P(B) = \P(\tau^-_x = t-1)$ by duality, this proves Lemma~\ref{lemma:y}.
\end{proof}

\section{Convergence of the height process}\label{sect:cv-height}

\subsection{Probabilistic set-up.} \label{sub:probabilistic-set-up}

For each $p \geq 1$, let $(V^*_p, \cP^*_p)$ be an $\L$-valued random variable corresponding to a (sub)critical CMJ branching process, i.e., which satisfies
\begin{equation} \label{eq:(sub)critical}
	0 \leq \E(\lvert \cP^*_p \rvert) \leq 1.
\end{equation}
We further assume that the sequence $(\cP^*_p)$ is near-critical in the sense that
\begin{equation} \label{eq:near-critical}
	\lim_{p \to \infty} \E(\lvert \cP^*_p \rvert) = 1.
\end{equation}

Let $\y^*_p$ be the random variable with distribution prescribed by~\eqref{eq:law-y^*} with $\cP^* = \cP^*_p$, and $\P_p$ be the probability distribution on $\Omega$ under which $\omega$ is an i.i.d.\ sequence with common distribution $(V^*_p, \cP^*_p)$. We let $\Rightarrow$ denote weak convergence under $\P_p$ and $\fdd$ denote convergence in the sense of finite-dimensional distributions under $\P_p$. For instance, $B_p \fdd B_\infty$ if and only if $(B_p(t), t \in I)$ under $\P_p$ converges weakly to $(B_\infty(t), t \in I)$ for any finite set $I \subset [0,\infty)$.

\subsection{Convergence of the height process.} We now state our main results concerning the convergence of the chronological height process: we fix a sequence $\varepsilon_p \to 0$ and consider the rescaled processes
\begin{equation} \label{eq:scaling-1}
	\cH_p(t) = \varepsilon_p \cH([pt]), \ \H_p(t) = \varepsilon_p \H([pt]) \ \text{ and } \ S_p(t) = \frac{1}{p\varepsilon_p} S([pt]), \ t \geq 0.
\end{equation}

Our results will involve the following condition. Except for the first integrability condition, it is automatically satisfied in the non-triangular case where the law of $\y^*_p$ does not depend on $p$.

\begin{condition}{T-H} \label{cond-H}
	For every $p \geq 1$, $\E(\y^*_p) < \infty$. Moreover, there exists an integrable random variable $\bar \y$ with $\E\bar \y = 0$ such that $\y^*_p - \E(\y^*_p) \Rightarrow \bar \y$ and $\E[(\y^*_p - \E(\y^*_p))^+] \to \E(\bar \y^+)$.
\end{condition}

\begin{theorem} \label{thm:H-fdd}
	Fix some $t > 0$. If Condition~\textnormal{\ref{cond-H}} holds and the sequence $(\cH_p(t), p \geq 1)$ is tight, then $\H_p(t) - \E(\y^*_p) \cH_p(t) \Rightarrow 0$.
\end{theorem}

\begin{proof}
	First of all, note that $\cH([pt]) \Rightarrow \infty$ since $\cH(n)$ and $\bTm(n)$ are equal in distribution by duality. Further, the fundamental formula~\eqref{eq:formula-H} gives
	\[ \H_p(t) - \E(\y^*_p) \cH_p(t) = \cH_p(t) \times \left( \frac{1}{\bTm([pt])} \sum_{k = 1}^{\bTm([pt])} \big( \y(k) - \E(\y^*_p) \big) \right) \circ \dualoperator^{[pt]}. \]
	Let in the sequel $W_p(n) = \bar \y_p(1) + \cdots + \bar \y_p(n)$ and $W(n) = \bar \y(1) + \cdots + \bar \y(n)$, where the two sequences $(\bar \y_p(k), k \geq 1)$ and $(\bar \y(k), k \geq 1)$ are i.i.d.\ with common distribution $\y^*_p - \E(\y^*_p)$ and $\bar \y$ introduced in Condition~\textnormal{\ref{cond-H}}, respectively. Fix $\eta > 0$ and $M, N \geq 1$: by duality, it follows from Lemma~\ref{lemma:y} and standard manipulations that
	\begin{multline*}
		\P_p \left( \left \lvert \H_p(t) - \E(\y^*_p) \cH_p(t) \right \rvert \geq \eta \right) \leq \P_p \left( \cH_p(t) \geq M \right) + \P_p\left(\cH([pt]) \leq N \right)\\
		+ \P \left( \sup_{n \geq N} \frac{1}{n} \left \lvert W_p(n) \right \rvert \geq \eta/M \right).
	\end{multline*}
	
	Letting first $p \to \infty$, then $N \to \infty$ and finally $M \to \infty$ makes the two first terms of the above upper bound vanish: the first one because the sequence $(\cH_p(t), n \geq 1)$ is tight and the second one because $\cH([pt]) \Rightarrow \infty$, and so we end up with
	\begin{equation} \label{eq:bound-H}
		\limsup_{p \to \infty} \P_p \left( \left \lvert \H_p(t) - \E(\y^*_p) \cH_p(t) \right \rvert \geq \eta \right) \leq \limsup_{N \to \infty} \limsup_{p \to \infty} \P \left( \sup_{n \geq N} \frac{1}{n} \left \lvert W_p(n) \right \rvert \geq 2\eta' \right)
	\end{equation}
	with $\eta' = \eta / (2M)$. We omit the $\limsup_{M \to \infty}$ because, as we now show, the previous limit is equal to $0$ for each fixed $M > 0$. In the non-triangular case where the law of $\y^*_p$ (and thus $W_p$) does not depend on $p$, this follows from the strong law of large numbers, and we now extend this to the triangular setting under Condition~\textnormal{\ref{cond-H}}. Writing
	\[ \sup_{n \geq N} \frac{1}{n} \left \lvert W_p(n) \right \rvert \leq \frac{1}{N} \left \lvert W_p(N) \right \rvert + \sup_{n \geq N} \frac{1}{n} \left \lvert W_p(n) - W_p(N) \right \rvert \]
	and using that $(W_p(n) - W_p(N), n \geq N)$ is equal in distribution to $W_p$, we get
	\[ \P \left( \sup_{n \geq N} \frac{1}{n} \left\lvert W_p(n) \right \rvert \geq 2\eta' \right) \leq \P \left( \frac{1}{N} \lvert W_p(N) \rvert \geq \eta' \right) + \P \left( \sup_{n \geq 0} \frac{1}{n + N} \left\lvert W_p(n) \right \rvert \geq \eta' \right). \]

	By the Portmanteau Theorem, we have
	\[ \limsup_{p\rightarrow\infty} \P\left(\frac{1}{N} \lvert W_p(N) \rvert \geq \eta' \right) \leq \P \left( \frac{1}{N} \lvert W(N) \rvert \geq \eta' \right), \]
	which entails
	\[ \limsup_{p \to \infty} \P \left( \frac{1}{N} \lvert W_p(N) \rvert \geq \eta' \right) \mathop{\longrightarrow}_{N \to \infty} 0. \] 
	
	As for the second term, if we define $W^\pm_p(n) = W_p(n) \pm \eta' n$ and $W^\pm(n) = W(n) \pm \eta' n$, then simple manipulations lead to
	\[ \P \left( \sup_{n \geq 0} \frac{1}{n + N} \left\lvert W_p(n) \right \rvert \geq \eta' \right) \leq \P \left( \sup_{n\geq0} W^-_p \geq \eta' N \right) + \P \left( \inf_{n\geq0} W^+_p \leq - \eta N \right). \]
	%%%%%%%%%%%
	% DETAILS %
	%%%%%%%%%%%
	% \begin{align*}
	% 	\P \left( \sup_{n \geq 0} \frac{1}{n + N} \left\lvert W_p(n) \right \rvert \geq \eta \right) & = \P \left( \exists n \geq 0: \left\lvert W_p(n) \right \rvert \geq \eta N + \eta n \right)\\
	% 	& = \P \left( \exists n \geq 0: W_p(n) \geq \eta N + \eta n \right) + \P \left( \exists n \geq 0: W_p(n) \leq - \eta N - \eta n \right)\\
	% 	& = \P \left( \exists n \geq 0: W^{-\eta}_p(n) \geq \eta N \right) + \P \left( \exists n \geq 0: W^{+\eta}_p(n) \leq - \eta N \right)
	% \end{align*}
	%%%%%%%%%%%%%%%
	% END DETAILS %
	%%%%%%%%%%%%%%%

	Under Condition~\textnormal{\ref{cond-H}}, we have $\sup W^-_p \Rightarrow \sup W^-$ and $\inf W^+_p \Rightarrow \inf W^+$, see for instance Theorem~$22$ in Borovkov~\cite{Borovkov76:0}. The result thus follows from the fact that, since $W^+$ (resp.\ $W^-$) is a random walk drifting to $+\infty$ (resp.\ $-\infty$), its infimum (resp.\ supremum) is finite.
\end{proof}

\begin{remark}\label{rem:conv-seq}
By the exact same argument, we leave the reader convince herself that if $t_p$ is a deterministic sequence such that $t_p/p\rightarrow 0$, then $\varepsilon_p \H(t_p) \Rightarrow 0$. This fact will be used later in proving the convergence of the contour process.
\end{remark}

We now state one immediate corollary of this result, which states that under mild conditions on the $\y^*_p$'s, the paths $\cH_p$ and $\H_p$ converge jointly in the sense of finite-dimensional distributions.

\begin{corollary} \label{cor:H-fdd}
	Assume that Condition~\textnormal{\ref{cond-H}} holds and that:
	\begin{enumerate}[label={\textnormal{(H\arabic*)}}]
		\item \label{HH} $p \varepsilon_p \to \infty$;
		\item $\E(\y^*_p) \to \alpha^*$ for some $\alpha^* \in (0,\infty)$;
		\item \label{H} $\cH_p \Rightarrow \cH_\infty$ for some $\cH_\infty$ satisfying $\P(\cH_\infty(t) > 0) = 1$ for every $t > 0$.
	\end{enumerate}
	Then
	\begin{equation}
	\label{eq:cvfdd} \left( \cH_p, \H_p \right) \longfdd \left( \cH_\infty, \, \alpha^* \cH_\infty \right).
	 \end{equation}
\end{corollary}
Condition~\ref{HH} is essentially a non-degeneracy condition: when $|\cP|=1$ a.s.\
it is not satisfied. Theorem~$2.3.1$ in Duquesne and Le Gall~\cite{Duquesne02:0} provides explicit conditions for Condition~\ref{H} to hold. Namely, the following three conditions together imply~\ref{H}:
\begin{enumerate}
	\item[(H3a)] $S_p \Rightarrow S_\infty$ for some L\'evy process $S_\infty$ with infinite variation;
	\item[(H3b)] the Laplace exponent $\psi$ of $S_\infty$ satisfies $\int_1^\infty \d u / \psi(u) < \infty$;
	\item[(H3c)] if $(Z^p_k, k \geq 0)$ is a Galton-Watson process with offspring distribution $\lvert \cP^*_p \rvert$ and started with $[p \varepsilon_p]$ individuals, then for every $\delta > 0$,
	\[ \liminf_{p \to \infty} \ \P \left( Z^p_{[\delta / \varepsilon_p]} = 0 \right) > 0. \]
\end{enumerate}

\section{Convergence of the contour process}\label{sec:cv:contour}

\subsection{Main results}

The probabilistic set-up is the same as in Section~\ref{sub:probabilistic-set-up}, in particular relations~\eqref{eq:(sub)critical} and~\eqref{eq:near-critical} hold, and we now turn to the asymptotic behavior of the chronological contour process $\C$. Under the assumption $\E(\y^*_p) \rightarrow \alpha^*<\infty$ and other mild conditions, we showed in Corollary~\ref{cor:H-fdd} that the genealogical and chronological height processes are essentially proportional to one another.

In this section, we study the contour process when this assumption is not enforced, which allows the chronological and genealogical processes to scale in different ways. We thus consider two sequences $\varepsilon_p$ and $\bar \varepsilon_p$, both converging to $0$, rescale the genealogical processes using $\bar \varepsilon_p$ as
\[ \cH_p(t) = \bar \varepsilon_p \cH([pt]), \ \cC_p(t) = \bar \varepsilon_p \cC(pt) \ \text{ and } \ S_p(t) = \frac{1}{p \bar \varepsilon_p} S([pt]), \]
and the chronological processes using $\varepsilon_p$ as
\[ \H_p(t) = \varepsilon_p \H([pt]) \ \text{ and } \ \C_p(t) = \varepsilon_p \C(pt). \]

\begin{remark}
When $\E(V^*)<\infty$, Theorem \ref{thm:H-fdd} ensures that the difference of scaling between the genealogical and the chronological height processes
can only occur when $\E(\y^*)=+\infty$. For instance, this will occur in the (non-triangular) case of
Poissonian birth events along the edges (as in \cite{Lambert10:0}) and when $\E((V^*)^2)=\infty$. 
\end{remark}

In the Galton-Watson case, it is well-known that $\cC_p$ is essentially obtained from $\cH_p$ by a deterministic time-change under rather mild assumptions (essentially conditions~\ref{A.epsilon}--\ref{A.X} below). We now show that a similar statement holds at the chronological level.

\begin{condition}{T-C1} \label{cond-C}
	We have $(V^*_p, \cP^*_p) \Rightarrow (V^*_\infty, \cP^*_\infty)$ for some $\L$-valued random variable $(V^*_\infty, \cP^*_\infty)$ with $\E(V^*_\infty) < \infty$ and $\E(\lvert \cP^*_\infty \rvert) = 1$.
\end{condition}

Let $V > 0$ be some random variable and $G$ be the additive subgroup generated by the support of its distribution. In the sequel we say that $V$ is \emph{non-arithmetic} if $G$ is dense in $\R$; otherwise, we say that $V$ is \emph{arithmetic} and in this case, there exists a unique $h > 0$, called the \emph{span} of $V$, such that $G = h \Z$. For a random variable $V > 0$ with finite mean, we define $\hat V$ as follows:
\begin{itemize}
	\item if $V$ is non-arithmetic, we define
	\[ \P(\hat V \geq x) = \frac{1}{\E(V)} \int_x^\infty \P(V \geq y) \d y, \ x \geq 0; \]
	\item if $V$ is arithmetic and $h$ is its span, we define
	\[ \P(\hat V = k h) = \frac{1}{\E(V)} \P(V > k h), \ k \in \N. \]
\end{itemize}

\begin{condition}{T-C2} \label{cond-C2}
	We have $\hat V^*_p \Rightarrow \hat V^*_\infty$ with $V^*_\infty$ as in Condition~\textnormal{\ref{cond-C}}, and moreover:
	\begin{itemize}
		\item if $V^*_\infty$ is non-arithmetic, then $V^*_p$ for each $p$ is non-arithmetic;
		\item if $V^*_\infty$ is arithmetic, then $V^*_p$ for each $p$ is arithmetic.
	\end{itemize}
\end{condition}

In the sequel, we will refer to the first case as the \emph{non-arithmetic case} and to the second case as the \emph{arithmetic case}. Note that, except for the integrability condition $\E(V^*_\infty) < \infty$, Conditions~\ref{cond-C} and \ref{cond-C2} as well as condition~\ref{A.V} below are automatically satisfied in the non-triangular case where the law of $(V^*_p, \cP^*_p)$ does not depend on~$p$.

\begin{theorem} \label{thm:C-fdd}
	Assume that Conditions~\textnormal{\ref{cond-C}} and~\textnormal{\ref{cond-C2}} hold and that:
	\begin{enumerate}[label={\textnormal{(C\arabic*)}}]
		\item \label{A.V} $\E(V^*_p) \to \beta^*$ with $\beta^* = \E(V^*_\infty)< \infty$;
		\item \label{A.epsilon} $\lim_{p \to \infty} p \varepsilon_p = \lim_{p \to \infty} p \bar \varepsilon_p = \infty$;
		\item\label{A.X} $S_p \Rightarrow S_\infty$ for some L\'evy process $S_\infty$ with infinite variation;
		\item\label{A.HC} $(\cH_p, \cC_p) \Rightarrow (\cH_\infty, \cC_\infty)$ for some (almost surely) continuous processes $\cH_\infty, \cC_\infty$ satisfying the condition $\P(\cH_\infty(t), \cC_\infty(t) \ > 0) = 1$ for every $t > 0$;
		\item \label{A.Hc} $\H_p \longfdd \H_\infty$ for some process $\H_\infty$ which is (almost surely) continuous at $0$ and satisfies the condition $\P(\H_\infty(t) > 0) = 1$ for every $t > 0$;
	\end{enumerate}
	and let $\varphi_\infty(t) = t / (2\beta^*)$. Then
	\begin{equation} \label{eq:conv-C-fdd}
		\left( \H_p, \C_p \right) \longfdd \left( \H_\infty, \H_\infty \circ \varphi_\infty \right).
	\end{equation}
\end{theorem}

Note that the three assumptions~(H3a)--(H3c) stated after Corollary~\ref{cor:H-fdd} actually imply~\ref{A.HC} with $\cC_\infty(t) = \cH_\infty(t/2)$. Moreover, instead of assuming~\ref{A.X}+\ref{A.HC}, we could merely assume~\ref{A.X} and that $\cH_p \Rightarrow \cH_\infty$ with $\cH_\infty$ continuous and with $\P(\cH_\infty(t) > 0) = 1$: indeed, results in~\cite{Duquesne02:0} show that this implies~\ref{A.HC} with $\cC_\infty$ as above.

Combining Theorems~\ref{thm:H-fdd} and~\ref{thm:C-fdd}, we obtain the following joint convergence.

\begin{corollary} \label{cor:C}
	Assume that except for~\ref{A.Hc}, the conditions of Theorems~\ref{thm:H-fdd} and~\ref{thm:C-fdd} hold with $\bar \varepsilon_p=\varepsilon_p$: then
	\[ \left( \cH_p, \cC_p, \H_p, \C_p \right) \longfdd \left( \cH_\infty, \cH_\infty(\, \cdot \,/2), \alpha^* \cH_\infty, \alpha^* \cH_\infty \circ \varphi_\infty \right). \]
\end{corollary}

We finally complement these results by showing that the trees themselves converge in the sense of finite-dimensional distributions. To do so, we only need considering the minimum of the contour process, see for instance Le Gall~\cite{Le-Gall05:0} for more details.

\begin{theorem}\label{thm:fdd-trees}
	Assume that except for~\ref{A.Hc}, the conditions of Theorems~\ref{thm:H-fdd} and~\ref{thm:C-fdd} hold with $\bar \varepsilon_p=\varepsilon_p$. Assume moreover that the sequence of random variables $(\y^*_p)$ is uniformly integrable: then for every $0 \leq u \leq v$ we have
	\[ \inf_{u \leq t \leq v} \C_p(t) - \alpha^* \inf_{u \leq t \leq v} \cC_p(2\varphi_\infty(t)) \Rightarrow 0. \]
\end{theorem}

\begin{remark}
	In~\cite{Sagitov95:0}, Sagitov investigated ({in the non-triangular setting}) the size of a CMJ process conditioned to survive at large time under the short edge assumption, i.e., when $\E(V^*_1)<\infty$ and $\E(\y^*_1) < \infty$ (see also Section~\ref{sec:example} and Green~\cite{Green77:0}). The population size is described in the limit in terms of a continuous state branching process where space and time are scaled analogously as in Corollary~\ref{cor:C}. As a consequence, the previous corollary can be seen as a genealogical version of~\cite{Sagitov95:0}. We also note that in~\cite{Sagitov95:0}, the results are obtained through an entirely different approach, namely analytic computations involving some non-trivial extension of the renewal theorem.
\end{remark}

In the rest of this section we discuss the proof of Theorem~\ref{thm:C-fdd}: the proof of Theorem~\ref{thm:fdd-trees}, provided in Section~\ref{sub:proof-trees}, uses essentially the same arguments, together with the additional result of Corollary~\ref{cor:formula-min}. In order to prove~\eqref{eq:conv-C-fdd} and in view of the assumption~\ref{A.Hc}, we only need to prove that
\begin{equation}
\label{eq:goal-proof-thm-C}
\forall t \geq 0, \ \C_p(t) - \H_p\circ\vp_\infty(t) \Rightarrow 0.
\end{equation}

To show this result, it is tempting to draw inspiration from the proof of Theorem~2.4.1 in Duquesne and Le Gall~\cite{Duquesne02:0}, where it is proved that $\sup_{0 \leq s \leq t} \lvert \cC_p(s) - \cH_p(s/2) \rvert \Rightarrow 0$ for each fixed $t \geq 0$. The proof of this result relies heavily on the assumption that the discrete height process converges weakly (i.e., in a functional sense) to its continuum counterpart. At the genealogical level, assuming weak convergence is not much stronger than assuming convergence of the finite-dimensional distributions, see~\cite[Theorem 2.3.1]{Duquesne02:0}. At the chronological level however, the simple example presented in the Section~\ref{sec:example} illustrates that the gap between these two modes of convergence is more significant. In Section~\ref{sec:overview} we give an overview of the main steps for proving~\eqref{eq:goal-proof-thm-C}, thereby highlighting key differences with the Galton-Watson case.

\subsection{Overview of the proof of Theorem~\ref{thm:C-fdd}} \label{sec:overview}

Except in Section~\ref{sec:example}, we assume in the rest of the paper that Conditions~\textnormal{\ref{cond-C}} and~\textnormal{\ref{cond-C2}} and Conditions~\ref{A.V}--\ref{A.Hc} of Theorem~\ref{thm:C-fdd} hold. The two conditions $V^*_p \Rightarrow V^*_\infty$ with $V^*_\infty$ integrable and $\E(V^*_p) \to \E(V^*_\infty)$ imply that the sequence $(V^*_p)$ is uniformly integrable (see for instance~\cite[Theorem $3.6$]{Billingsley99:0}), which implies the following triangular weak law of large numbers. It can be directly checked by computing Laplace transforms or by invoking \textsection22 in Gnedenko and Kolmogorov~\cite{Gnedenko68:0}.

\begin{lemma}\label{lemma:triangular-LLN}
	For any sequence $u_p \to \infty$, we have $\cV([u_p]) / u_p \Rightarrow \beta^*$. In particular, for any $s \geq 0$ we have $\cV([ps]) / p \Rightarrow \beta^* s$.
\end{lemma}

In view of the construction of the chronological contour process $\C$ in Section~\ref{subsub:chronological-processes}, we have
\begin{equation} \label{eq:C-H}
	\sup_{t \in [K_n, K_{n+1}]} \left \lvert \C(t) - \H(n) \right \rvert \leq \left \lvert \H(n+1) - \H(n) \right \rvert + V_n.
\end{equation}

Let $\varphi$ be the left-continuous inverse of $(K_{[t]}, t \geq 0)$, defined by
\begin{equation} \label{eq:def-varphi}
	\varphi(t) := \min \left\{ j \geq 0: K_{j} \geq t \right\}, \ t \geq 0.
\end{equation}

Then defining
\begin{equation}\label{vpn}
	\varphi_p(t) := \frac{1}{p} \varphi(pt),
\end{equation}
the inequality~\eqref{eq:C-H} translates after scaling to
\[ \left \lvert \C_p(t) - \H_p(\varphi_p(t)) \right \rvert \leq \varepsilon_p V_{\varphi(pt)} + \left \lvert \H_p(\varphi_p(t)+1/p) - \H_p(\varphi_p(t)) \right \rvert, \ t \geq 0, \]
%%%%%%%%%%%%%%
% DERIVATION %
%%%%%%%%%%%%%%
% \begin{align*}
% 	\left \lvert \C_p(t) - \H_p(\varphi_p(t)) \right \rvert & = \varepsilon_p \left \lvert \C(pt) - \H(\varphi(pt)) \right \rvert\\
% 	& = \varepsilon_p \sum_{k \geq 0} \left \lvert \C(pt) - \H(\varphi(pt)) \right \rvert \indicator{\varphi(pt) = k}\\
% 	& = \varepsilon_p \sum_{k \geq 0} \left \lvert \C(pt) - \H(k) \right \rvert \indicator{\varphi(pt) = k}\\
% 	& = \varepsilon_p \sum_{k \geq 0} \left \lvert \C(pt) - \H(k) \right \rvert \indicator{K_{k-1} < nt \leq K_k}\\
% 	& \leq \varepsilon_p \sum_{k \geq 0} \big( \left \lvert \H(k + 1) - \H(k) \right \rvert + V_k \big) \indicator{K_{k-1} < nt \leq K_k}\\
% 	& = \varepsilon_p \sum_{k \geq 0} \big( \left \lvert \H(k + 1) - \H(k) \right \rvert + V_k \big) \indicator{\varphi(pt) = k}\\
% 	& = \varepsilon_p \sum_{k \geq 0} \left( \left \lvert \H(\varphi(pt) + 1) - \H(\varphi(pt)) \right \rvert + V_{\varphi(pt)} \right) \indicator{\varphi(pt) = k}\\
% 	& = \varepsilon_p \left( \left \lvert \H(\varphi(pt) + 1) - \H(\varphi(pt)) \right \rvert + V_{\varphi(pt)} \right)\\
% 	& = \left \lvert \H_p(\varphi_p(t) + 1/n) - \H_p(\varphi_p(t)) \right \rvert + \varepsilon_p V_{\varphi(pt)}.
% \end{align*}
%%%%%%%%%%%%%%%%%%
% END DERIVATION %
%%%%%%%%%%%%%%%%%%
and so going back to~\eqref{eq:goal-proof-thm-C}, we obtain for any $t \geq 0$
\begin{multline} \label{eq:decomposition}
	\left \lvert \C_p(t) - \H_p(\varphi_\infty(t)) \right \rvert \leq \varepsilon_p V_{\varphi(pt)} + \left \lvert \H_p(\varphi_p(t) + 1/p) - \H_p(\varphi_\infty(t)) \right \rvert\\
	+ 2 \left \lvert \H_p(\varphi_p(t)) - \H_p(\varphi_\infty(t)) \right \rvert.
\end{multline}

The proofs of $\H_p(\varphi_p(t) + 1/p) - \H_p(\varphi_\infty(t)) \Rightarrow 0$ and of $\H_p(\varphi_p(t)) - \H_p(\varphi_\infty(t)) \Rightarrow 0$ proceed along similar lines, and so in the sequel we only focus on the latter convergence. The above relation shows that, asymptotically, the correct time-change should be the limit of $\varphi_p$, and we now explain why this is indeed $\varphi_\infty$. Plugging in the definition $K_n = 2 \cV(n-1) - \H(n)$ into the definition of $\varphi$, we obtain
\[ \varphi_p(t) = \frac{1}{p} \inf\left\{ j \geq 0: 2 \cV(j-1) - \H(j) \geq p t \right\}. \]

For large $p$, the triangular law of large numbers of Lemma~\ref{lemma:triangular-LLN} suggests the approximation $\cV(p) \approx \beta^* p$; while under assumptions~\ref{A.epsilon} and ~\ref{A.Hc} , $\H(p)$ for large $p$ is of the order of $1/\varepsilon_p \ll p$. These two observations thus give a rationale for the following result.

\begin{lemma} \label{lemma:phi}
	For every $t \geq 0$ we have $\varphi_p(t) \Rightarrow \varphi_\infty(t)$.
\end{lemma}

\begin{proof}
	Consider any $t' < \varphi_\infty(t)$: using the definition of $\varphi_p$, the fact that $\H(j) \geq 0$ and that $\cV$ is increasing, one obtains that
	\[ \P_p \left( \varphi_p(t) < t' \right) \leq \P_p \left( 2 \cV\big( n p' \big) < n p \right). \]
	
	Since $\cV(ps) / p \Rightarrow \beta^* s$ for any $s \geq 0$ by Lemma~\ref{lemma:triangular-LLN}, we obtain $\P_p \left( \varphi_p(t) < t' \right) \to 0$ for $t' < \varphi_\infty(t)$. Let now $t' > \varphi_\infty(t)$, and write
	\[ \P_p\left( \varphi_p(t) > t' \right) \leq \P_p \left( 2 \cV(pt') - \H(pt') \leq nt \right). \]
	
	Since the sequence $(\varepsilon_p \H([pt']), p \geq 1)$ is tight and $p \varepsilon_p \to \infty$, we obtain $\H(pt') / p \Rightarrow 0$ and so $(2 \cV(pt') - \H(pt')) / p \Rightarrow 2 \beta^* t'$ by Lemma~\ref{lemma:triangular-LLN}. Consequently, we obtain the convergence $\P_p \left( 2 \cV(pt') - \H(pt') \leq pt \right) \to 0$ which concludes the proof.
\end{proof}

In view of this result, a natural idea to prove $\H_p(\varphi_p(t)) - \H_p(\varphi_\infty(t)) \Rightarrow 0$ is to use a uniform control of the kind
\[ \left \lvert \H_p(\varphi_p(t)) - \H_p(\varphi_\infty(t)) \right \rvert \leq \sup \left\{ \left \lvert \H_p(s) - \H_p(\varphi_\infty(t)) \right \rvert : \left \lvert s - \varphi_\infty(t) \right \rvert \leq \eta_p \right\} \]
for some $\eta_p \to 0$ such that $\P_p(\lvert \varphi_p(t) - \varphi_\infty(t) \rvert \leq \eta_p) \to 1$. However, the example considered in Section~\ref{sec:example} strongly suggests that even for $\eta_p$ precisely of the order of $\lvert \varphi_p(t) - \varphi_\infty(t) \rvert$, the supremum of the previous upper bound may blow up. Such a control is therefore too rough and more care is needed.

One of the main obstacle for a finer control is the convoluted relation between $\H_p$ and $\varphi_p(t)$, whereby $\H_p$ appears in the definition of $\varphi_p(t)$; this is also the reason why it is not straightforward to prove the apparently innocuous convergence $\varepsilon_p V_{\varphi(pt)} \Rightarrow 0$ which is required in order to deal with the first term in the upper bound of~\eqref{eq:decomposition}.
\\

In order to circumvent this difficulty, we introduce a random time $\bar \varphi_p(t)$ close to $\varphi_p(t)$ and which will be easier to control. More precisely, we consider
\[ \bar \varphi_p(t) = \frac{1}{p} \bar \varphi(pt) \ \text{ with } \ \bar \varphi(t) = \inf \left\{ j \geq 0: 2 \cV(j) \geq t \right\} \]
the first passage time of the renewal process $2\cV$ above level $t$. Note that, since $V_n$ and $\H(n)$ are non-negative, we have $\bar \varphi(t) \leq \varphi(t)$ for every $t \geq 0$.

For fixed $p$, the renewal theorem provides an asymptotic description as $t \to \infty$ of the process $2 \cV$ shifted at time $\bar \varphi(t)$. In Section~\ref{sub:renewal} we will prove a triangular version of this result, and Condition~\textnormal{\ref{cond-C2}} is here to ensure that this extension of the renewal theorem to a triangular setting holds. We will for instance prove the following result.

\begin{lemma}\label{lemma:renewal-V}
	For any $t \geq 0$, we have $\varepsilon_p V_{\bar \varphi(pt)} \Rightarrow 0$.
\end{lemma}

\begin{proof}
	See forthcoming Corollary~\ref{lemma:spine-seen}.
\end{proof}

This result illustrates the fact that $\bar \varphi_p(t)$ is more convenient to work with compared to $\varphi_p(t)$.
Besides, $\varphi_p(t)$ and $\bar \varphi_p(t)$ are close: the triangular law of large numbers of Lemma~\ref{lemma:triangular-LLN} implies similarly as in the proof of Lemma~\ref{lemma:phi} that $\bar \varphi_p(t) \Rightarrow \varphi_\infty(t)$ and, to be more precise, the next result implies that their difference is at most of the order of $1/\varepsilon_p$. This result is a consequence of Proposition~\ref{prop:Delta} which will be proved in Section~\ref{sub:proof-tightness-Delta}.

\begin{lemma} \label{lemma:tightness-Delta}
	For any $t \geq 0$, the sequence of random variables $(\varepsilon_p (\varphi_p(t) - \bar \varphi_p(t)), p \geq 1)$ is tight.
\end{lemma}

\begin{proof}
	See forthcoming Proposition~\ref{prop:Delta}.
\end{proof}

Lemmas~\ref{lemma:renewal-V} and~\ref{lemma:tightness-Delta} allow to get rid of the first term in the upper bound~\eqref{eq:decomposition} as we show now.

\begin{corollary} \label{cor:V}
	For any $t \geq 0$, we have $\varepsilon_p V_{\varphi(pt)} \Rightarrow 0$.
\end{corollary}

\begin{proof}
	Since $\bar \varphi(pt) \leq \varphi(pt)$, for any $M, \eta > 0$ we have
	\begin{multline*}
		\P_p \left( \varepsilon_p V_{\varphi(pt)} \geq \eta \right) \leq \P_p \left( \varphi(pt) - \bar \varphi(pt) > M / \varepsilon_p \right)\\
		+ \P_p \left( \varepsilon_p \max \left\{ V_k: k = \bar \varphi(pt), \ldots, \bar \varphi(pt) + [M / \varepsilon_p ]\right\} \geq \eta \right)
	\end{multline*}
	which gives
	\begin{multline*}
		\P_p \left( \varepsilon_p V_{\varphi(pt)} \geq \eta \right) \leq \P_p \left( \varepsilon_p (\varphi_p(t) - \bar \varphi_p(t)) > M \right) + \P_p \left( \varepsilon_p V_{\bar \varphi(pt)} \geq \eta \right)\\
		+ \P_p \left( \varepsilon_p \max \left\{ V_{\bar \varphi(pt) + k}: k = 1, \ldots, [M / \varepsilon_p] \right\} \geq \eta \right).
	\end{multline*}
	
	Lemmas~\ref{lemma:renewal-V} and~\ref{lemma:tightness-Delta} imply that the two first terms vanish, while for the third term, we write
	\[ \P_p \left( \varepsilon_p \max \left\{ V_{\bar \varphi(pt) + k}: k = 1, \ldots, [M / \varepsilon_p] \right\} \geq \eta \right) \leq \frac{M}{\varepsilon_p} \P\left( \varepsilon_p V^*_p \geq \eta \right) \leq \frac{M}{\eta} \E \left( V^*_p; V^*_p \geq \frac{\eta}{\varepsilon_p} \right) \]
	where the first inequality follows from the fact that the $(V_{\bar \varphi(pt)+k}, k \geq 1)$ under $\P_p$ are i.i.d.\ with common distribution $V^*_p$. Since the $(V^*_p)$ are uniformly integrable, this last bound vanishes as $p \to \infty$, which completes the proof.
\end{proof}

In order to show that $\H_p(\varphi_p(t)) - \H_p(\varphi_\infty(t)) \Rightarrow 0$, we introduce $\H_p(\bar \varphi_p(t))$ and write
\begin{equation} \label{eq:two-terms}
	\left \lvert \H_p(\varphi_p(t)) - \H_p(\varphi_\infty(t)) \right \rvert \leq \left \lvert \H_p(\bar \varphi_p(t)) - \H_p(\varphi_\infty(t)) \right \rvert + \left \lvert \H_p(\varphi_p(t)) - \H_p(\bar \varphi_p(t)) \right \rvert.
\end{equation}

We will then study each term of this upper bound. We will control the first term $\H_p(\bar \varphi_p(t)) - \H_p(\varphi_\infty(t))$ by showing that the spine originated from the random time $\bar \varphi(pt)$ asymptotically looks like the spine originated from a deterministic time. To do so we prove an extension of the renewal theorem to a triangular setting and a macroscopic horizon in Section~\ref{sub:renewal}, thereby extending results of Miller~\cite{Miller74:0}.

To control the second term $\H_p(\bar \varphi_p(t)) - \H_p(\varphi_p(t))$, we introduce the shifted process $\H' = (\H(\bar \varphi(pt) + k) - \H(\bar \varphi(pt)), k \geq 0)$ and write $\H_p(\varphi_p(t)) - \H_p(\bar \varphi_p(t)) = \varepsilon_p \H'(\Delta)$ with $\Delta = \varphi(pt) - \bar \varphi(pt)$. The key idea is that $\H'$ turns out to be close in distribution to $\H$, and so elaborating on Lemma~\ref{lemma:tightness-Delta} which states that $\Delta$ is small macroscopically (since $p \gg 1/\eps_p$ by condition~\ref{A.epsilon}) will give the desired result.

\subsection{Organization of the rest of the paper}

The rest of the paper is organized as follows. In Section~\ref{sec:preliminary-results} we prove some preliminary results, namely some formulas on the height process which extend the right decomposition of the spine introduced in Section~\ref{sub:right-decomposition}, as well as some renewal type results: in particular, these results make it possible to prove Lemma~\ref{lemma:renewal-V}. Section~\ref{sec:proofs} contains the remaining proofs, namely the proof of Lemma~\ref{lemma:tightness-Delta}, the proof that each term in the upper bound of~\eqref{eq:two-terms} vanishes and finally the proof of Theorem~\ref{thm:fdd-trees}.

\section{Preliminary results} \label{sec:preliminary-results}

\subsection{Right decomposition of the spine continued.}

\begin{lemma}\label{lemma:decompo-m}
	For any $n\geq m \geq 0$ with $0 \leq \mrca{m}{n} < m$, we have
	\[ \spine_0^m = \left[ \spine_0^{\mrca{m}{n}}, \cQ \circ \bTm (\tau_{L(n-m)}) \circ \vartheta^m, \ldots, \cQ(1) \circ \vartheta^m \right]. \]
\end{lemma}
\begin{proof}
	By Lemma \ref{lemma:snake-3}, for every $k$ such that $T(k)\circ\vartheta^m\leq m$ we have
	\[ \spine_0^{m} = \left[ \spine_0^{m-T(k)\circ\vartheta^m}, \cQ(k)\circ\vartheta^m, \ldots, \cQ(1)\circ\vartheta^m \right]. \]
	Let $k = \bTm(\tau_{L(n-m)})$: then $T(k) = \tau_{L(n-m)}$ (as $T(\bTm(i)) = i$ for every $i \geq 0$) and so $T(k) \circ \dualoperator^m = m - \mrca{m}{n}$ by~\eqref{eq:identity-mrca}. Since by assumption $\mrca{m}{n} \geq 0$, we have $T(k) \circ \dualoperator^m \leq m$ and so the application of Lemma~\ref{lemma:snake-3} gives the result as $m - T(k) \circ \dualoperator^m = \mrca{m}{n}$.
\end{proof}

In the sequel, we consider the measurable function $D_\ell: \cM^* \to \R_+$ that satisfies $D_0\equiv0$ and for $\ell \in \N\setminus\{0\}$:
\begin{equation} \label{eq:def-D}
	D_\ell(\spine^n_0) = \left( \sum_{i: 0 < T(i) \leq \min(\tau_{\ell},n)} \y(i) - \Indicator{\tau_\ell \leq n} \pi(\mu_\ell) \right) \circ \dualoperator^n, n \in \N.
\end{equation}
The fact that the right hand side is measurable with respect to $\spine_0^n$ (and thus can be written as a function of $\spine_0^n$) is a consequence of Proposition~\ref{prop:formula-spine} and the fact that the random variables appearing in the formula are related to the dual Lukasiewicz path $S\circ\vartheta^n$.

Moreover, we leave the reader check that for any $Y \in \cM^*$ the sequence $(D_\ell(Y), \ell \in \N)$ is increasing. Actually, this comes from a more general fact, namely that $D_\ell(Y)$ for $Y \in \cM^*$ gives the distance between $\pi(Y)$ and the $\ell$-th stub of $Y$.

The following result relates the two shifts which play a key role in this paper : on the one hand, the canonical shift $\shift$ which acts on the initial sequence of sticks $((V_n, \cP_n), n \in \Z)$ through the term $\pi(\spine^n_m) = \pi(\spine^{n-m}_0) \circ \shift_m$, and on the other hand, the shift in time through the term $\H(n) - \H(m)$.

\begin{prop} \label{prop:shifts}
	For every $0 \leq m \leq n$ we have
	\begin{equation} \label{id-diff}
		 \H(n) - \H(m) = \pi(\spine^n_m) - D_{L(n-m) \circ \dualoperator^m}(\spine^m_0).
	 \end{equation}
\end{prop}

\begin{proof}
	Applying~\ref{eq:def-D} to the random $\ell = L(n-m) \circ \dualoperator^m$, we obtain (see Remark~\ref{rk:manipulation})
	\begin{equation} \label{eq:D-L}
		D_{L(n-m) \circ \dualoperator^m} (\spine^m_0) = \left( \sum_{i = 1}^{\bTm(\min(\tau_{L(n-m)}, m))} \y(i) - \Indicator{\tau_{L(n-m)} \leq m} \pi(\mu_{L(n-m)}) \right) \circ \dualoperator^m.
	\end{equation}
	To prove~\eqref{id-diff} we distinguish the two cases $\mrca{m}{n}<0$ and $\mrca{m}{n} \geq 0$.

	\noindent {\it Case 1: $\mrca{m}{n}<0$.} By~\eqref{eq:identity-mrca} this condition is equivalent to $\tau_{L(n-m)}\circ\vartheta^m > m$: in view of~\eqref{eq:D-L}, we thus need to show that 
	\[
		\H(n)-\H(m) = \pi(\spine_{m}^n) - \left(\sum_{i = 1}^{\bTm(m)} \y(i)\right)\circ \vartheta^m.
	\]
	Using the expression for $\H(n)$, $\H(m)$ and $\pi(\spine^n_m)$ provided by Proposition~\ref{prop:formula-H} and~\eqref{eq:formula-pi(rhonm)}, we see that in order to show the above relation we only have to show that $\bTm(n-m) \circ \dualoperator^n = \bTm(n) \circ \dualoperator^n$. This in turn follows from the fact that the condition $\mrca{m}{n}<0$ implies that $T(T^{-1}(n-m))\circ\vartheta^n>n$ (again by~\eqref{eq:identity-mrca}),
	which is equivalent to saying that the sets $\{T(i):i\in\N\}\circ\vartheta^n$ and $\{n-m, \ldots, n\}$ do not intersect and gives $\bTm(n-m) \circ \dualoperator^n = \bTm(n) \circ \dualoperator^n$. The proof in this case is thus complete.
	\\
	
	\noindent {\it Case 2: $\mrca{m}{n} \geq 0$.} The result is obvious in the case $\mrca{m}{n} = m$, while in the other case we can invoke Proposition~\ref{prop:snake-det} and Lemma~\ref{lemma:decompo-m} that give
	\[ \H(n) = \pi(\spine^{\mrca{m}{n}}_0) + \pi(\mu_{L(n-m)}) \circ \dualoperator^m + \pi(\spine^n_m) \ \text{ and } \ \H(m) = \pi(\spine^{\mrca{m}{n}}_0) + \left( \sum_{i=1}^{\bTm(\tau_{L(n-m)})} \y(i) \right) \circ \dualoperator^m. \]
	Taking the difference between these two expressions yield the result in view of~\eqref{eq:D-L} (recall that $\mrca{m}{n} \geq 0$ is equivalent to $\tau_{L(n-m)}\circ\vartheta^m \leq m$).
\end{proof}

The following lemma relates the shifted spine to the Skorohod reflection. 
\begin{lemma}\label{lem:sk}
	For any $0 \leq m \leq n$, we have $\pi(\spine_m^n) = \H(n) - \min_{k=m,\ldots,n} \H(k)$.
\end{lemma}
\begin{proof}
	It follows from~\eqref{eq:formula-pi(rhonm)} that
	\[ \pi(\spine_m^n) = \left(\sum_{i=1}^{\bTm(n)} \y(i) \right) \circ \vartheta^n - \left(\sum_{i=\bTm(n-m)+1}^{\bTm(n)} \y(i) \right) \circ \vartheta^n = \H(n) - \left(\sum_{i=\bTm(n-m)+1}^{\bTm(n)} \y(i) \right) \circ \vartheta^n. \]
	Next, we have from Proposition~\ref{prop:formula-spine} that
	\[ \spine^n_0 = \left( \cQ(\bTm(n)), \cQ(\bTm(n)-1), \ldots, \cQ(1) \right) \circ \dualoperator^n \]
	while Lemma~\ref{lemma:snake-3} with $k = \bTm(n-m)$ gives
	\[ \spine^n_0 = \left[ \spine^{n - T(\bTm(n-m)) \circ \vartheta^n}_0, \cQ(\bTm(n-m)) \circ \vartheta^n, \ldots, \cQ(1) \circ \vartheta^n \right]. \]
	Comparing the two expressions for $\spine_0^n$ ,we see that
	\[ \spine^{n - T(\bTm(n-m)) \circ \vartheta^n}_0 = \left( \cQ(\bTm(n)), \ldots, \cQ(\bTm(n-m)+1) \right) \circ \vartheta^n \]
	and in particular,
	\[ \left(\sum_{i=\bTm(n-m)+1}^{\bTm(n)} \y(i) \right) \circ \vartheta^n = \H \big( n-T(\bTm(n-m)) \circ \vartheta^n \big). \]
	We let the reader convince herself that $\H \big( n-T(\bTm(n-m)) \circ \vartheta^n \big) = \min_{\{m, \ldots, n\}} \H$ (again by comparing the number of ladder height times at $n-T(\bTm(n-m)) \circ \vartheta^n$ and $k\in\{m, \ldots, n\}$), so that gathering the previous relations we finally obtain the desired result.
\end{proof}

\begin{corollary}\label{cor:formula-min}
	For any $0 \leq m \leq n$,
	\begin{equation}\label{eq:min}
		\min_{K_m \leq t \leq K_n} \C(t) = \H(m) - D_{L(n-m)\circ\vartheta^m}(\spine^m_0).
	\end{equation}
\end{corollary}

\begin{proof}
	Let $I^n_m = \min_{[K_m, K_n]} \C$. Since $\H(n) - \H(m) = \pi(\spine^n_m) - D_{L(n-m)\circ\vartheta^m}(\spine^m_0)$ by Proposition~\ref{prop:shifts}, in order to prove~\eqref{eq:min} it is enough to prove that
	\[ \pi \left( \spine^n_m \right) = \H(n) - I^n_m. \]
	Local minima of $\C$ are by construction attained on the set $\{K_n \ : \ n\in\N\}$ and since $\H(k)=\C(K_k)$ for any $k \in \N$, this implies $I^n_m = \min_{k = m, \ldots, n} \H(k)$. The result then follows from Lemma \ref{lem:sk}.
\end{proof}

\subsection{Triangular renewal theorem on a macroscopic horizon} \label{sub:renewal}

By construction, $\pi(\spine^n_m)$ only depends on the finite vector $\cP^n_m = (\cP_k, k = m, \ldots, n-1)$, and we can thus for instance write $\pi(\spine^n_m) = \Xi_{n-m}(\cP^n_m)$ for some measurable mapping $\Xi_{n-m} : \cM^{n-m} \to [0,\infty)$. With this notation, Condition~\ref{A.Hc} on the convergence of the chronological height process precisely means that if we take a vector $\nu^p \in \cM^{[p \delta]}$ of $[p \delta]$ i.i.d.\ random measures with common distribution $\cP^*_p$, then $\Xi_{[p\delta]}(\nu^p)$ converges weakly to $\H_\infty(\delta)$. For instance, for any $0 < \delta < \varphi_\infty(t)$ we have $\Xi_{[p \delta]} \big( \cP^{[\varphi_\infty(pt)]}_{[\varphi_\infty(pt)] - [p \delta]} \big) \Rightarrow \H_\infty(\delta)$ and we want to extend this result by replacing the deterministic time $[\varphi_\infty(pt)]$ by the random one $\bar \varphi(pt)$.

Of course, the random variables $(\cP_k, k = \bar \varphi(pt) - [p \delta], \ldots, \bar \varphi(pt) - 1)$ are not i.i.d.\ and so we cannot directly invoke the same argument. However, the renewal theorem suggests that these random variables become asymptotically i.i.d.\ as $p \to \infty$, which gives a rationale for, e.g., the convergence $\Xi_{[p\delta]}\big( \cP^{\bar \varphi(pt)}_{\bar \varphi(pt) - [p \delta]} \big) = \pi \big( \spine^{\bar \varphi(pt)}_{\bar \varphi(pt) - [p \delta]} \big) \Rightarrow \H_\infty(\delta)$. Results with a similar flavor, i.e., renewal theorems on a macroscopic horizon, can be found in Miller~\cite{Miller74:0}.

Two technical difficulties prevent us from using Miller's or other standard results: $(1)$ we are in a triangular setting and $(2)$ we need to consider a growing number of terms (of the order of $p$). In addition, Miller~\cite{Miller74:0} typically assumes the almost sure convergence of $\Xi_{[p \delta]} \big( \cP^{[p\delta]}_0 \big)$ when we only have weak convergence.

In order to overcome these difficulties, we exploit the coupling between two random walks with the same step distribution but possibly different initial distributions constructed in the proof of Lemma~$9.21$ in Kallenberg~\cite{Kallenberg02:0}. This coupling leads to the following results proved in the Appendix~\ref{sec:coupling}.

\begin{prop}\label{Prop:renewal-1}
	Let $(\hat V_\infty^*, \hat \cP_\infty^*)$ have the following size-biased distribution: for every measurable function $f: \R_+ \times \cM \to \R_+$,
	\begin{itemize}
		\item if $V_\infty^*$ is non-arithmetic,
		\[ \E\left[ f(\hat V_\infty^*, \hat \cP_\infty^*) \right ] = \frac{1}{\E(V_\infty^*)} \int_0^\infty \E \left[ f(v,\cP_\infty^*) \mid V_\infty^* = v \right] \P ( V_\infty^* > v ) \d v; \]
		\item if $V_\infty^*$ is arithmetic with span $h$,
		\[ \E\left[ f(\hat V_\infty^*, \hat \cP_\infty^*) \right ] = \frac{1}{\E(V_\infty^*)} \sum_{i \geq 1} \E \left[ f(ih,\cP_\infty^*) \mid V_\infty^* = ih \right] \P ( V_\infty^* > ih ). \]
	\end{itemize}
	Then $\left(V_{\bvp(pt)},{\mathcal P}_{\bvp(pt)} \right) \Rightarrow \left(\hat V^*_\infty, \hat \cP^*_\infty\right)$ for every $t > 0$.
\end{prop}

\begin{prop}\label{Prop:renewal-2}
	For each $p \geq 1$ let $\Xi_p: \cM^p \to \R$ be a measurable mapping such that $\Xi_p(\cP_0^p) \Rightarrow \Xi_\infty$ for some random variable $\Xi_\infty$. Then $\Xi_{[\delta p]}(\cP^{\bvp(pt)}_{\bvp(pt)-[p\delta]}) \Rightarrow \Xi_\infty$ for any $0 < \delta < t / (2 \beta^*)$.
\end{prop}

Recall the exploration process $\rho^n_0 = \spine^n_0 \circ \cG$, which similarly as~\eqref{eq:S^n_m} is extended by setting $\rho^n_m = \spine^n_m \circ \cG = \rho^{n-m}_0 \circ \shift_m$. The following corollary to Propositions~\ref{Prop:renewal-1} and~\ref{Prop:renewal-2} gathers the results needed in the sequel.

\begin{corollary} \label{lemma:spine-seen}
	For $t \geq 0$, the three sequences $\varepsilon_p V_{\bar \varphi(pt)}$, $\varepsilon_p \pi(\cP_{\bar \varphi(pt)})$ and $\varepsilon_p \lvert \cP_{\bar \varphi(pt)} \rvert$ converge weakly to $0$ as $p \to \infty$. If in addition $0 < \delta < t / (2 \beta^*)$, then
	\[ \varepsilon_p\pi\left( \spine^{\bvp(pt)}_{\bvp(pt)-[p\delta]}\right) \Rightarrow \H_\infty(\delta), \ \bar \varepsilon_p \pi\left( \rho^{\bvp(pt)}_{\bvp(pt)-[p\delta]}\right) \Rightarrow \cH_\infty(\delta) \]
	and
	\[ \sup_{0 \leq u \leq \delta} S_p(u) \circ \dualoperator^{\bvp(pt)} \Rightarrow \sup_{0 \leq u \leq \delta} S_\infty(u). \]
\end{corollary}

\begin{proof}
	The convergence of the three sequences $\varepsilon_p V_{\bar \varphi(pt)}$, $\varepsilon_p \pi(\cP_{\bar \varphi(pt)})$ and $\varepsilon_p \lvert \cP_{\bar \varphi(pt)} \rvert$ is a direct consequence of Proposition~\ref{Prop:renewal-1} (note that, for point processes, the functionals $\pi$ and $\lvert \cdot \rvert$ are continuous for the weak topology).
	
	Let us now discuss the remaining convergence of $\varepsilon_p\pi\left( \spine^{\bvp(pt)}_{\bvp(pt)-[p\delta]}\right)$, $\bar \varepsilon_p \pi\left( \rho^{\bvp(pt)}_{\bvp(pt)-[p\delta]}\right)$ and $\sup_{[0, \delta]} S_p \circ \dualoperator^{\bvp(pt)}$. From their definition, each of these random variables can be expressed in the form $\Xi_{[\delta p]}\big( \cP^{\bar \varphi(pt)}_{\bar \varphi(pt) - [p \delta]} \big)$ for some measurable mappings $\Xi_p: \cM^p \to [0,\infty)$. Proposition~\ref{Prop:renewal-2} implies that $\Xi_{[\delta p]}\big( \cP^{\bar \varphi(pt)}_{\bar \varphi(pt) - [p \delta]} \big)$ converges if $\Xi\big( \cP^{[p \delta]}_0 \big)$ does, in which case they have the same limit. This means that we are brought back to the convergence of $\H_p(\delta)$, $\cH_p(\delta)$ and $\sup_{[0, \delta]} S_p$ and since each of these three terms convergences by assumption~\ref{A.X},~\ref{A.HC} and~\ref{A.Hc}, the result follows.
\end{proof}

\section{Proof of Theorems~\ref{thm:C-fdd} and~\ref{thm:fdd-trees}} \label{sec:proofs}

We now complete the proof of Theorems~\ref{thm:C-fdd} and~\ref{thm:fdd-trees}: Theorem~\ref{thm:C-fdd} is proved in Sections~\ref{sub:proof-remaining-1}--\ref{sub:proof-remaining-2} and Theorem~\ref{thm:fdd-trees} in Section~\ref{sub:proof-trees}. For Theorem~\ref{thm:C-fdd}, recall from the discussion in Section~\ref{sec:overview} that there remains to prove Lemma~\ref{lemma:tightness-Delta} as well as the fact that both terms in the upper bound of~\eqref{eq:two-terms} vanish, i.e., that
\begin{equation} \label{eq:remaining-1}
	\H_p(\bar \varphi_p(t)) - \H_p(\varphi_\infty(t)) \Rightarrow 0
\end{equation}
and
\begin{equation} \label{eq:remaining-2}
	\H_p(\varphi_p(t)) - \H_p(\bar \varphi_p(t)) \Rightarrow 0.
\end{equation}

Using the results of the previous section, we will first prove~\eqref{eq:remaining-1} in Section~\ref{sub:proof-remaining-1}. Then, we will use~\eqref{eq:remaining-1} to prove the following result in Section~\ref{sub:proof-tightness-Delta}.

\begin{prop}\label{prop:Delta}
	For any $t > 0$ and any $\eta > 1/(2 \beta^*)$,
	\[ \lim_{p \to \infty} \P_p \left( \varphi(pt) - \bar \varphi(pt) > \eta \, \H(\bar \varphi(pt)) \right) = 0. \]
\end{prop}

Combining~\eqref{eq:remaining-1} and Condition~\ref{A.Hc} implies that $\H(\bar \varphi(pt))$ is of the order of $1/\varepsilon_p$, and so Proposition~\ref{prop:Delta} directly implies Lemma~\ref{lemma:tightness-Delta}. Finally, we will use Proposition~\ref{prop:Delta} to prove~\eqref{eq:remaining-2} in Section~\ref{sub:proof-remaining-2}, which will achieve the proof of Theorem~\ref{thm:C-fdd}.

\subsection{Proof of~\eqref{eq:remaining-1}} \label{sub:proof-remaining-1}

We start with the following simple lemma.

\begin{lemma} \label{lemma:bound-m}
	For any $1 \leq m \leq n$,
	\begin{equation} \label{eq:bound-m}
		0 \leq \pi(\spine_{m-1}^n) - \pi(\spine_m^n) \leq \pi(\cP_{m-1}).
	\end{equation}
\end{lemma}

\begin{proof}
	Relation~\eqref{eq:formula-pi(rhonm)} gives
	\[ \pi(\spine^n_{m-1}) = \left( \sum_{k = 1}^{\bTm(n-m+1)} \y(k) \right) \circ \dualoperator^n. \]
	If $\bTm(n-m+1) \circ \dualoperator^n = \bTm(n-m) \circ \dualoperator^n$, then we obtain $\pi(\spine^n_{m-1}) = \pi(\spine^n_m)$ and so the result holds in this case. Otherwise, we have $\bTm(n-m+1) \circ \dualoperator^n = \bTm(n-m) \circ \dualoperator^n+1$ and so isolating the last term, we obtain
	\[ \pi(\spine^n_{m-1}) = \pi(\spine^n_m) + \y \big( \bTm(n-m+1) \big) \circ \dualoperator^n. \]
	Further, for any $k \in \N$ we have
	\[ \y\left(\bTm(k)\right) = \pi \circ \cQ\left(\bTm(k)\right) = \pi \circ \Upsilon_{\zeta_0}(\cP_{\tau_0-1}) \circ \theta_{T(\bTm(k))-1} \leq \pi(\cP_{\tau_0-1}) \circ \theta_{T(\bTm(k))-1}. \]
	As $\tau_0 \circ \theta_{T(\bTm(k))-1} = 1$, this gives $\y(\bTm(k)) \leq \pi\left(\cP_{T(\bTm(k)) - 1}\right)$ and consequently,
	\[ \y\left(\bTm(n-m+1)\right) \circ \dualoperator^n \leq \pi\left(\cP_{T(\bTm(n-m+1))-1}\right) \circ \dualoperator^n = \pi\left(\cP_{n - T(\bTm(n-m+1)) \circ \dualoperator^n}\right). \]
	The condition $\bTm(n-m+1) \circ \dualoperator^n = \bTm(n-m) \circ \dualoperator^n+1$ means that $n-m+1$ is a weak ascending ladder height time (for the dual process $S \circ \dualoperator^n$) and thus implies the relation $T(\bTm(n-m+1)) \circ \dualoperator^n = n-m+1$. Plugging in this relation in the previous display achieves the proof.
\end{proof}

Let for simplicity $m_p = \mrca{\bar \varphi(pt)}{[\varphi_\infty(pt)]}$. Since we have $\H(m_p) \leq \H(\bar \varphi(pt))$ as well as $\H(m_p) \leq \H([\varphi_\infty(pt)])$, the triangular inequality reads
\[ \left \lvert \H_p(\bar \varphi_p(t)) - \H_p(\varphi_\infty(t)) \right \rvert \leq \varepsilon_p \left( \H(\bar \varphi(pt)) - \H(m_p) \right) + \varepsilon_p \left( \H([\varphi_\infty(pt)]) - \H(m_p) \right) \]
and since $m_p \leq \min (\varphi(pt), [\varphi_\infty(pt)])$,~\eqref{id-diff} gives by neglecting the terms $D \geq 0$
\[ \left \lvert \H_p(\bar \varphi_p(t)) - \H_p(\varphi_\infty(t)) \right \rvert \leq \varepsilon_p \pi \left( \spine^{\bar \varphi(pt)}_{m_p} \right) + \varepsilon_p \pi \left( \spine^{[\varphi_\infty(pt)]}_{m_p} \right). \]
In particular, we only need to show that $\varepsilon_p \pi(\spine^{\phi_p}_{m_p}) \Rightarrow 0$ for $\phi_p = \bar \varphi(pt)$ or $[\varphi_\infty(pt)]$. Using the monotonicity of $\pi(\spine^n_m)$ in~$m$ given by Lemma~\ref{lemma:bound-m}, we obtain for any $0 < \delta < \varphi_\infty(t)$
\[ \P_p \left( \varepsilon_p \pi \left( \spine^{\phi_p}_{m_p} \right) \geq \eta \right) \leq \P_p \left( m_p \leq \phi_p - [p \delta] \right) + \P_p \left( \varepsilon_p \pi \left( \spine^{\phi_p}_{\phi_p - [p \delta]} \right) \geq \eta \right). \]

The second term converges to $\P_p \left( \H_\infty(\delta) \geq \eta \right)$: for $\phi_p = [\varphi_\infty(pt)]$ this is a consequence of~\ref{A.Hc}, and for $\phi_p = \bar \varphi(pt)$ this was proved in Corollary~\ref{lemma:spine-seen} for $\delta$ small enough. Since this 
inequality holds for every $\delta$ small enough and since $\H_\infty$ is almost surely continuous at $0$ by Condition~\ref{A.Hc}, in order to conclude the proof it remains to show that $\P_p(m_p \leq \phi_p - [p \delta]) \to 0$ as $p \to \infty$ for each fixed $0 < \delta < \varphi_\infty(t)$, which we do now.

By Assumption~\ref{A.HC}, the genealogical contour process $\cC_p$ converges weakly to a continuous process $\cC_\infty$. Since $\phi_p/p \Rightarrow \varphi_\infty(t)$, this implies that $\cC_p(t_p) - \inf_{I_p} \cC_p \Rightarrow 0$ with $t_p = \phi_p/p$ or $t_p = \varphi_\infty(t)$ and $I_p = [\min(\phi_p/p, \varphi_\infty(t)), \max(\phi_p/p, \varphi_\infty(t))]$. By classical arguments on discrete trees, this implies that the genealogical distance rescaled by $\bar \varepsilon_p$ between $\phi_p$ and $m_p$ converges to $0$, i.e., $\bar \varepsilon_p (\cH(\phi_p) - \cH(m_p)) \Rightarrow 0$. Therefore, for any $\eta > 0$ we obtain
\[ \limsup_{p \to \infty} \P_p \left( m_p \leq \phi_p - [p \delta] \right) \leq \limsup_{p \to \infty} \P_p \left( m_p \leq \phi_p - [p \delta], \bar \varepsilon_p (\cH(\phi_p) - \cH(m_p)) \leq \eta \right). \]

Since $L(n-m) \circ \dualoperator^m = 0$ if and only if $m = \mrca{m}{n}$, Proposition~\ref{prop:shifts} implies that $\cH(n) - \cH(m) = \pi(\spine^n_m) \circ \cG = \pi(\rho^n_m)$ for any $0 \leq m \leq n$ with $m \in \cA(n) \cap \R_+$. In particular, it follows by definition of $m_p$ that $\cH(\phi_p) - \cH(m_p) = \pi(\rho^{\phi_p}_{m_p})$. Since $\pi(\rho^n_m)$ is non-increasing in $m$ by Lemma~\ref{lemma:bound-m}, this gives
\[ \P_p \left( m_p \leq \phi_p - [p \delta], \bar \varepsilon_p (\cH(\phi_p) - \cH(m_p)) \leq \eta \right) \leq \P_p \left( \bar \varepsilon_p \pi\left(\rho^{\phi_p}_{\phi_p - [p \delta]} \right) \leq \eta \right). \]

Since this term converges to $\P(\cH_\infty(\delta) \leq \eta)$ (for $\phi_p = \bar \varphi(pt)$ this comes from Corollary~\ref{lemma:spine-seen} and for $\phi_p = [\varphi_\infty(pt)]$ this is the convergence of the genealogical height process assumed in~\ref{A.HC}) we finally obtain
\[ \limsup_{p \to \infty} \P_p \left( m_p \leq \phi_p - [p \delta] \right) \leq \P \left( \cH_\infty(\delta) \leq \eta \right). \]

Letting $\eta \to 0$ in the last display therefore concludes the proof thanks to Condition~\ref{A.HC}.

\subsection{Proof of Proposition~\ref{prop:Delta}} \label{sub:proof-tightness-Delta}

In order to prove this result, we introduce two intermediate height processes. We enrich the probability space with a random variable $\widetilde \cP$ which under $\P_p$ is equal in distribution to $\cP_1$ and independent from the sequence $(\cP_{\bar \varphi(pt) + k}, k \geq 1)$, and we consider $\widetilde \spine_{(p)} = (\widetilde \spine^n_{(p)}, n \geq 0)$ the spine process defined from the sequence $(\widetilde \cP, \cP_{\bar \varphi(pt) + 1}, \cdots)$. For $k \geq 0$ we then let
\[ \widehat \H^p(k) = \pi\left(\spine_{\bar \varphi(pt)}^{\bar \varphi(pt)+k}\right) \ \text{ and } \ \widetilde \H^p(k) = \pi\left(\widetilde \spine_{(p)}^k \right). \]

\begin{lemma}\label{lemma:H-H'}
	 $\widetilde \H^p$ under $\P_p$ is equal in distribution to $\H$ under $\P_p$. Moreover, we have $\varepsilon_p \sup_{k \geq 0} \lvert \widetilde \H^p(k) - \widehat \H^p(k) \rvert \Rightarrow 0$.
\end{lemma}

\begin{proof}
	The first part of the lemma directly follows from the strong Markov property. As for the second part, Lemma~\ref{lemma:bound-m} gives
	\[ 0 \leq \pi\left(\spine_{\bar \varphi(pt)}^{\bar \varphi(pt)+k}\right) - \pi\left(\spine_{\bar \varphi(pt)+1}^{\bar \varphi(pt)+k}\right) \leq \pi(\cP_{\bar \varphi(pt)}) \ \text{ and } \ 0 \leq \pi\left(\widetilde \spine_{(p)}^k\right) - \pi\left(\spine_{\bar \varphi(pt)+1}^{\bar \varphi(pt)+k}\right) \leq \pi(\tilde \cP) \]
	which gives $\lvert \widetilde \H^p(k) - \widehat \H^p(k) \rvert \leq \pi(\widetilde \cP) + \pi(\cP_{\bar \varphi(pt)})$. Since this bound is uniform in $k$ and both $\widetilde \cP$ and $\cP_{\bar \varphi(pt)}$ converge weakly (by Condition~\textnormal{\ref{cond-C}} and Corollary~\ref{lemma:spine-seen}), multiplying by $\varepsilon_p$ and letting $p \to \infty$ gives the result.
\end{proof}

We now turn to the proof of Proposition~\ref{prop:Delta}. Let in the rest of the proof $\Delta_p = \varphi(pt) - \bar \varphi(pt)$. Since by definition
\[ \varphi(pt) = \inf \left\{ k \geq 1: 2 \cV(k-1) - \H(k) \geq pt \right\} \ \text{ and } \ \bar \varphi(pt) = \inf \left\{ k \geq 1: 2 \cV(k) \geq pt \right\}, \]
it follows that
\[ \Delta_p = \inf \left\{ k \geq 0 : 2 \cV(\bar \varphi(pt)+k-1) - \H(\bar \varphi(pt)+k) \geq pt \right\}. \]

Defining $\bar \cV_p(k) = \cV(\bar \varphi(pt) + k) - \cV(\bar \varphi(pt))$ for $k \geq -1$, we obtain
\[ \Delta_p = \inf \left\{ k \geq 0 : 2 \bar \cV_p(k-1) - \H(\bar \varphi(pt)) \geq \H(\bar \varphi(pt)+k) - \H(\bar \varphi(pt)) - (2 \cV(\bar \varphi(pt)) - pt) \right\} \]
%%%%%%%%%%%
% DETAILS %
%%%%%%%%%%%
% \begin{align*}
% 	\Delta_p & = \inf \left\{ k \geq 1 : 2 \cV(\bar \varphi(pt)+k-1) - \H(\bar \varphi(pt)+k) \geq pt \right\}\\
% 	& = \inf \left\{ k \geq 1 : 2 \bar \cV_p(k-1) + 2 \cV(\bar \varphi(pt)) - \H(\bar \varphi(pt)+k) \geq pt \right\}
% \end{align*}
%%%%%%%%%%%%%%%
% END DETAILS %
%%%%%%%%%%%%%%%
and so according to Proposition~\ref{prop:shifts},
\begin{multline} \label{eq:formula-Delta}
	\Delta_p = \inf \left\{ k \geq 0 : \right.\\
	\left. 2 \bar \cV_p(k-1) - \H(\bar \varphi(pt)) \geq \pi \left( \spine^{\bar \varphi(pt) + k}_{\bar \varphi(pt)} \right) - D_{L(k) \circ \dualoperator^{\bar \varphi(pt)}} \left( \spine^{\bar \varphi(pt)}_0 \right) - (2 \cV(\bar \varphi(pt)) - pt) \right\}.
\end{multline}

Since $D_k(\nu) \geq 0$ and $2 \cV(\bar \varphi(pt)) \geq pt$, we obtain by definition of $\widehat \H^p$ that
\[ \Delta_p \leq \inf \left\{ k \geq 0 : 2 \bar \cV_p(k-1) - \H(\bar \varphi(pt)) \geq \widehat \H^p(k) \right \}. \]

In particular, if $\sigma_p = [\eta \H(\bar \varphi(pt))]$ then in order to prove the result it is enough to show that $\P_p \left(2 \bar \cV_p(\sigma_p - 1) - \H(\bar \varphi(pt)) \geq \widehat \H^p(\sigma_p) \right) \to 1$ which we rewrite as
\[ \P_p \left(2 \bar \cV_p(\sigma_p - 1) - \sigma_p / \eta \geq \widehat \H^p(\sigma_p) \right) \mathop{\longrightarrow}_{p \to \infty} 1. \]

Since for any $\gamma > 0$, we have
\[ \P_p \left(2 \bar \cV_p(\sigma_p - 1) - \sigma_p / \eta \geq \widehat \H^p(\sigma_p) \right) \geq \P_p \left(2 \bar \cV_p(\sigma_p - 1) - \sigma_p / \eta \geq \gamma / \varepsilon_p \geq \widehat \H^p(\sigma_p)\right) \]
the desired convergence is implied by the following two relations:
\begin{equation} \label{eq:goal-Delta}
	\varepsilon_p \widehat \H^p(\sigma_p) \Rightarrow 0 \ \text{ and } \ \liminf_{p \to \infty} \ \P_p \left( 2 \bar \cV_p(\sigma_p - 1) - \sigma_p / \eta \geq \gamma / \varepsilon_p \right) \mathop{\longrightarrow}_{\gamma \to 0} 1.
\end{equation}

Let us begin by proving the first relation $\varepsilon_p \widehat \H^p(\sigma_p) \Rightarrow 0$. Corollary~\ref{cor:H-fdd} combined with~\eqref{eq:remaining-1} shows that $\varepsilon_p \sigma_p \Rightarrow \eta \H_\infty(\varphi_\infty(t))$, and since $p \varepsilon_p \to \infty$ by~\ref{A.epsilon}, it follows that $\sigma_p / p \Rightarrow 0$. Since $\widetilde \H^p$ is equal in distribution to $\H$ by Lemma~\ref{lemma:H-H'} and $\sigma_p$ is independent of $\widetilde \H^p$, we obtain in view of Remark~\ref{rem:conv-seq} that $\varepsilon_p \widetilde \H^p(\sigma_p) \Rightarrow 0$. The second part of Lemma~\ref{lemma:H-H'} finally entails the desired result $\varepsilon_p \widehat \H^p(\sigma_p) \Rightarrow 0$.

We now prove the second convergence in~\eqref{eq:goal-Delta}. By construction, $\bar \cV_p$ is a renewal process independent of $\H(\bar \varphi(pt))$, and thus independent of $\sigma_p$: Lemma~\ref{lemma:triangular-LLN} thus implies that $\bar \cV_p(\sigma_p-1)/\sigma_p \Rightarrow \beta^*$ and since, as already mentioned, $\varepsilon_p \sigma_p \Rightarrow \eta \H_\infty(\varphi_\infty(t))$, we get
\[ \liminf_{p \to \infty} \P_p \left( 2 \bar \cV_p(\sigma_p - 1) - \sigma_p / \eta \geq \gamma / \varepsilon_p \right) \geq \P\left( (2 \beta^* \eta - 1) \H_\infty(\varphi_\infty(t)) \geq \gamma \right). \]

Since $(2 \beta^* \eta - 1)>0$ and $\cH_\infty(\varphi_\infty(t))>0$ a.s.\ by~\ref{A.Hc}, the result follows by letting $\gamma \to 0$.

\subsection{Proof of~\eqref{eq:remaining-2}} \label{sub:proof-remaining-2}

Let as in the previous subsection $\Delta_p = \varphi(pt) - \bar \varphi(pt)$. Proposition~\ref{prop:shifts} gives
\begin{equation} \label{eq:remaining-2-decomposition}
	\H_p(\varphi_p(t)) - \H_p(\bar \varphi_p(t)) = \varepsilon_p \pi\left(\spine_{\bar \varphi(pt)}^{\varphi(pt)}\right) - \varepsilon_p D_{L'_p(\Delta_p)} \left(\spine^{\bar \varphi(pt)}_0 \right),
\end{equation}
where we have defined $L'_p(k) = L(k) \circ \dualoperator^{\bar \varphi(pt)}$. We now show that each term of the right-hand side of~\eqref{eq:remaining-2-decomposition} vanishes, and we start with the second one, i.e., we show that
\begin{equation} \label{eq:D-vanish}
	\varepsilon_p D_{L'_p(\Delta_p)} \left(\spine^{\bar \varphi(pt)}\right) \Rightarrow 0.
\end{equation}
It is not hard to prove that $D_{L'_p(k)} (\spine^{\bar \varphi(pt)}_0)$ is non-decreasing in $k$ 
and the sequence $(\varepsilon_p \Delta_p, p \geq 1)$ is tight, it is enough to show that
\begin{equation} \label{eq:D-vanish-2}
	\varepsilon_p D_{L(t_p) \circ \dualoperator^{\bar \varphi(pt)}}\left(\spine^{\bar \varphi(pt)}_0 \right) \Rightarrow 0
\end{equation}
for some deterministic integer-valued sequence $(t_p)$ with $\varepsilon_p t_p \to \infty$: we will consider $t_p = [(p / \varepsilon_p)^{1/2}]$, which satisfies in addition $t_p / p \to 0$. In order to prove~\eqref{eq:D-vanish-2}, we fix until further notice $\gamma, \gamma' > 0$ and two integer-valued sequences $(\gamma_p)$, $(\gamma'_p)$ such that $\gamma_p / p \to \gamma$ (in particular $t_p / \gamma_p \to 0$) and $\gamma'_p / (p\bar \varepsilon_p) \to \gamma'$. Since both $D_k(\spine^{\bvp(pt)}_0)$ and $L(k) \circ \dualoperator^{\bar \varphi(pt)}$ are non-decreasing with $k$, it follows that for $p$ large enough such that $t_p \leq \gamma_p$, we have
\[ \P_p \left( \varepsilon_p D_{L(t_p) \circ \dualoperator^{\bar \varphi(pt)}}(\spine^{\bvp(pt)}_0) \geq \eta \right) \leq \P_p \left( L(\gamma_p) \circ \dualoperator^{\bar \varphi(pt)} \geq \gamma'_p \right) + \P_p \left( \varepsilon_p D_{\gamma_p'}(\spine^{\bvp(pt)}_0) \geq \eta \right). \]

By definition of $L$ and $S$, the first term is equal to
\[ \P_p \left( L(\gamma_p) \circ \dualoperator^{\bar \varphi(pt)} \geq \gamma'_p \right) = \P_p \left( \min_{i = 0, \ldots, \gamma_p} \sum_{k=\bar \varphi(pt)}^{\bar \varphi(pt) + i} (\lvert \cP_k \rvert - 1) \leq -\gamma'_p \right). \]

Isolating the term $\lvert \cP_{\bar \varphi(pt)} \rvert - 1$ and using that the $\cP_k$'s for $k \geq \bar \varphi(pt) + 1$ are i.i.d., we further get
\[ \P_p \left( L(\gamma_p) \circ \dualoperator^{\bar \varphi(pt)} \geq \gamma'_p \right) \leq \P_p \left( \lvert \cP_{\bar \varphi(pt)} \rvert \leq -\frac{\gamma'_p}{2} + 1 \right) + \P_p \left( \min_{i = 1, \ldots, \gamma_p} \sum_{k=1}^i (\lvert \cP_k \rvert - 1) \leq - \frac{\gamma'_p}{2} \right). \]

The first term vanishes by~\ref{A.epsilon} and Corollary~\ref{lemma:spine-seen}, and so rescaling the second term by $p \bar \varepsilon_p$ and using~\ref{A.X}, we obtain
\[ \limsup_{p \to \infty} \P_p \left( L(\gamma_p) \circ \dualoperator^{\bar \varphi(pt)} \geq \gamma'_p \right) \leq \P \left( \inf_{0 \leq t \leq \gamma} S_\infty(t) \leq -\frac{\gamma'}{2} \right). \]

By letting first $p \to \infty$ and then $\gamma \downarrow 0$, we thus have at this point
\[ \limsup_{p \to \infty} \P_p \left( \varepsilon_p D_{L(t_p) \circ \dualoperator^{\bar \varphi(pt)}}(\spine^{\bvp(pt)}_0) \geq \eta \right) 
\leq \limsup_{p \to \infty} \ \P_p \left( \varepsilon_p D_{\gamma'_p} (\spine^{\bvp(pt)}_0) \geq \eta \right). \]

Fix now some $0 < \delta < t / (2 \beta^*)$: by definition~\eqref{eq:def-D} of $D$,
\[ D_{\gamma'_p}(\spine^{\bar \varphi(pt)}_0) \leq \left( \sum_{i: 0 < T(i) \leq \tau_{\gamma'_p}} \y(i) \right) \circ \dualoperator^{\bar \varphi(pt)} \]
and so in the event $\{ \tau_{\gamma'_p} \circ \dualoperator^{\bar \varphi(pt)} \leq [p \delta] \}$, we get
\[ D_{\gamma'_p}(\spine^{\bar \varphi(pt)}_0) \leq \left( \sum_{i=1}^{\bTm([p \delta])} \y(i) \right) \circ \dualoperator^{\bar \varphi(pt)} = \pi \left( \spine^{\bar \varphi(pt)}_{\bar \varphi(pt) - [p \delta]} \right), \]
where we have used~\eqref{eq:formula-pi(rhonm)} to derive the last equality. In particular,
\[ \P_p \left( \varepsilon_p D_{\gamma_p'}(\spine^{\bvp(pt)}_0) \geq \eta \right) \leq \P_p \left(\tau_{\gamma'_p} \circ \dualoperator^{\bar \varphi(pt)} > [p \delta] \right) + \P_p \left( \varepsilon_p \pi\left( \spine^{\bvp(pt)}_{\bvp(pt)- [p \delta]} \right) > \eta\right) \]
and since by definition we have
\[ \P_p \left(\tau_{\gamma'_p} \circ \dualoperator^{\bar \varphi(pt)} > [p \delta] \right) = \P_p \left( \sup_{k = 0, \ldots, [p\delta]} S(k) \circ \dualoperator^{\bar \varphi(pt)} \leq \gamma_p' \right), \]
Corollary~\ref{lemma:spine-seen} implies that
\[ \limsup_{p \to \infty} \ \P_p \left( \varepsilon_p D_{\gamma'_p} (\spine^{\bvp(pt)}_0) \geq \eta \right) \leq \P\left(\sup_{0 \leq t \leq \delta} S_\infty(t) \leq \gamma' \right) + \P \left( \H_\infty(\delta) \geq \eta \right). \]

Letting first $\gamma' \to 0$ and then $\delta \to 0$ concludes the proof of~\eqref{eq:D-vanish-2}, and so also of~\eqref{eq:D-vanish}.
\\

We now show that the first term in the right-hand side of~\eqref{eq:remaining-2-decomposition} also vanishes. In view of~\eqref{eq:formula-Delta} and using $2 \cV(\bar \varphi(pt)) - pt \leq 2V_{\bar \varphi(pt)}$, we obtain
\[ \varepsilon_p \pi \left( \spine^{\varphi(pt)}_{\bar \varphi(pt)} \right) \leq \varepsilon_p \left( 2 \bar \cV_p(\Delta_p-1) - \H(\bar \varphi(pt)) \right) + \varepsilon_p D_{L(\Delta_p) \circ \dualoperator^{\bar \varphi(pt)}} \left( \spine^{\bar \varphi(pt)}_0 \right) + 2 \varepsilon_p V_{\bar \varphi(pt)}. \]

We have just proved that the second term vanishes (in law), and since the third term also vanishes by Lemma~\ref{lemma:renewal-V} it only remains to control the first term. Since $\bar \cV$ is an increasing sequence, for any $\gamma, \eta > 0$ we have
\begin{multline*}
	\P_p \left( \varepsilon_p \left( 2 \bar \cV_p(\Delta_p-1) - \H(\bar \varphi(pt)) \right) \geq \gamma \right) \leq \P_p \left( \Delta_p > \eta \H(\bar \varphi(pt)) \right)\\
	+ \P_p \left( \varepsilon_p \left( 2 \bar \cV_p([\eta \H(\bar \varphi(pt))]) - \H(\bar \varphi(pt)) \right) \geq \gamma \right).
\end{multline*}

Choose now $\eta > 1/(2 \beta^*)$, so that the first term vanishes by Proposition~\ref{prop:Delta}. For the second term, we note that $\bar \cV$ is independent from $\H(\bar \varphi(pt))$ to obtain with similar arguments as in the proof of Proposition~\ref{prop:Delta}
\[ \limsup_{p \to \infty} \P_p \left( \varepsilon_p \left( 2 \bar \cV_p([\eta \H(\bar \varphi(pt))]) - \H(\bar \varphi(pt)) \right) \geq \gamma \right) \leq \P \left( (2 \beta^* \eta - 1) \H_\infty(\varphi_\infty(t)) \geq \gamma \right). \]

Since $\P(\H_\infty(\varphi_\infty(t)) > 0) = 1$, letting $\eta \to 1/(2 \beta^*)$ concludes the proof.

\subsection{Proof of Theorem~\ref{thm:fdd-trees}} \label{sub:proof-trees}

In this section, we assume in addition to everything else that $(\y^*_p)$ is uniformly integrable and we prove Theorem~\ref{thm:fdd-trees}.

\begin{lemma}\label{lem:D-X}
	Let $(\ell(p), p \geq 0)$ be a deterministic sequence in $\R_+$ going to $\infty$. Then for every $t>0$ we have
	\[ \varepsilon_p \left( D_{\ell(p)}(\spine^{[pt]}_0) - \alpha^* D_{\ell(p)}(\spine^{[pt]}_0 \circ \cG) \right) \Rightarrow 0. \]
\end{lemma}

\begin{proof}
	Let $\bTm_p = \bTm(\min(\tau_{\ell(p)}, [pt]))$ and $R_p = \Indicator{0 < \tau_{\ell(p)} \leq [pt]} \pi(\mu_{\ell(p)})$, so that by definition~\eqref{eq:def-D} of $D$ we have
	\[ D_{\ell(p)}(\spine^{[pt]}_0) = \left( \sum_{i=1}^{\bTm_p} \y(i) \right) \circ \vartheta^{[pt]} - R_p \circ \vartheta^{[pt]}. \]
	Using the various facts that $D_{\ell(p)}(\spine^{[pt]}_0 \circ \cG) = D_{\ell(p)}(\spine^{[pt]}_0) \circ \cG$, that $\y(i) \circ \cG = \pi(\mu_\ell) \circ \cG = 1$, that $\bTm_p$ and $\tau_{\ell(p)}$ are genealogical quantities and finally that $\vartheta^{[pt]}$ and $\cG$ commute, composing on the right with $\cG$ in the previous display gives
	\[ D_{\ell(p)}(\spine^{[pt]}_0 \circ \cG) = \left( \bTm_p - \Indicator{\tau_{\ell(p)} \leq [pt]} \right) \circ \vartheta^{[pt]}. \]
	By duality, we therefore only have to show that the three quantities
	\[ \varepsilon_p R_p, \ \varepsilon_p \Indicator{\tau_{\ell(p)} \leq [pt]} \ \text{ and } \ \varepsilon_p \sum_{k = 1}^{\bTm_p} \left( \y(k) - \E(\y^*_p) \right) \]
	converge weakly to $0$. The second one obviously does since $\varepsilon_p \to 0$. For the third one we proceed similarly as in the proof of Theorem~\ref{thm:H-fdd}: indeed, $\varepsilon_p \bTm_p$ is tight (because it is smaller than $\varepsilon_p \bTm([pt])$ by monotonicity of $\bTm$, which is equal in distribution to $\cH_p(t)$), which is the only assumption necessary for the proof of Theorem~\ref{thm:H-fdd} to go through.
	
	We now prove that $\varepsilon_p R_p \Rightarrow 0$, which will conclude the proof. First of all, let $\Gamma$ such that $\tau_{\ell(p)} = T(\Gamma)$: then by definition, $\mu_{\ell(p)} = \mu_0 \circ \theta_{T(\Gamma-1)}$ and so $\pi(\mu_{\ell(p)}) = \pi \circ \mu_0 \circ \theta_{T(\Gamma-1)} = \y(\Gamma-1)$. In addition, if $\tau_{\ell(p)} \leq [pt]$ then $\Gamma \leq \bTm([pt])$ and so $R_p \leq \max_{k = 1, \ldots, \bTm([pt])} \y(k)$. Next, we fix some $N \geq 0$, consider $N_p = [N / \varepsilon_p]$ and use the previous inequality to write
	\begin{multline} \label{eq:end}
		\P_p \left( \varepsilon_p \max_{k = 1, \ldots, \bTm([pt])} \y(k) \geq \eta \right) \leq \P_p \left( \bTm([pt]) \geq N_p \right)\\
		+ \P_p \left( \max_{k=1, \ldots, \min( N_p, G-1)} \y(k) \geq \eta_p \right)
	\end{multline}
	where $G = \inf\{k \geq 0: T(k) = \infty\}$ and $\eta_p = \eta / \varepsilon_p$. For the first term of the right-hand side, we note that $\bTm([pt])$ is by duality equal in distribution to $\cH(pt)$ to get
	\[ \limsup_{p \to \infty} \P_p \left( \bTm([pt]) \geq N_p \right) = \limsup_{p \to \infty} \P_p \left( \cH_p(t) \geq \varepsilon_p N_p \right) \mathop{\longrightarrow}_{N \to \infty} 0. \]
	It remains to control the second term in the right-hand side of~\eqref{eq:end}: since the $(\y(k), k = 1, \ldots, G-1)$ are i.i.d.\ by Lemma~\ref{lemma:y}, we have
	\[ \P_p \left( \max_{k=1, \ldots, \min( N_p, G-1)} \y(k) \geq \eta_p \right) \leq 1 - \left[ 1 - \P \left( \y^*_p \geq \eta_p \right) \right]^{N_p}. \]
	This last bound vanishes because $N_p \P(\y^*_p \geq \eta_p) \to 0$ as a direct consequence of the uniform integrability of the $\y^*_p$ together with the following bound:
	\[ N_p \P \left( \y^*_p \geq \eta_p \right) \leq \frac{N}{\eta} \E\left(\y^*_p ; \y^*_p \geq \frac{\eta}{\varepsilon_p} \right). \]
	The proof is complete.
\end{proof}

In the sequel for $0 \leq u \leq v$ we define
\[ \cM(u,v) = \inf_{u \leq t \leq v} \cC(t) \ \text{ and } \ \M(u,v) = \inf_{u \leq t \leq v} \C(t). \]

\begin{corollary} \label{cor:end}
	For any $0<a<b$ we have
	\[ \varepsilon_p\left( \M(K_{[pa]},K_{[pb]}) - \alpha^* \cM(2pa,2pb) \right) \Rightarrow 0. \]
\end{corollary}

\begin{proof}
	First of all, we note that
	\[ \varepsilon_p \left( \cM(2pa,2pb) - \M(K_{[pa]}, K_{[pb]} ) \circ \cG \right) \Rightarrow 0. \]
	Indeed, this follows from rewriting $\cM(2\beta^*pa, 2\beta^*pb) = \inf \left\{ \cC_p(t): 2 a \leq t \leq 2 b \right\}$ and
	\[ \M(K_{[pa]}, K_{[pb]}) \circ \cG = \inf \left\{ \cC_p(t): \frac{1}{p} K_{[pa]} \circ \cG \leq t \leq \frac{1}{p} K_{[pb]} \circ \cG \right\}, \]
	together with the following two facts: 1) $\cC_p \Rightarrow \cC_\infty$ with $\cC_\infty$ continuous and 2) $p^{-1} K_{[pa]} \circ \cG \Rightarrow 2 a$. Therefore, in order to prove the result we only have to prove that
	\[ \varepsilon_p\left( \M(K_{[pa]},K_{[pb]}) - \alpha^* \M(K_{[pa]},K_{[pb]}) \circ \cG \right) \Rightarrow 0. \]
	
	To prove this, we define $L_p = L( [pb] - [pa] )\circ\vartheta^{[pa]}$ and apply Corollary~\ref{cor:formula-min} to write
	\begin{multline*}
		\varepsilon_p \left( \M(K_{[pa]},K_{[pb]}) - \alpha^* \M(K_{[pa]},K_{[pb]}) \circ \cG \right) = \varepsilon_p \left( \H([pa]) - \alpha^* \cH([pa]) \right)\\
		- \varepsilon_p\left( D_{L_p}(\spine^{[pa]}_0) - \alpha^* D_{L_p}(\spine^{[pa]}_0)\circ\cG \right).
	\end{multline*}
	The first term on the right-hand side vanishes by Theorem \ref{thm:H-fdd}, so we are left with the second term. Since $L_p$ is a genealogical quantity, this term is equal to
	\[ \varepsilon_p\left( D_{L_p}(\spine^{[pa]}_0) - \alpha^* D_{L_p}(\spine^{[pa]}_0)\circ\cG \right) = \varepsilon_p\left( D_{L_p}(\spine^{[pa]}_0) - \alpha^* D_{L_p}(\spine^{[pa]}_0\circ\cG) \right) \]
	and we can now invoke Lemma~\ref{lem:D-X} to conclude that this term vanishes, as $L_p$ is independent of $\spine^{[pa]}_0$ and converges weakly to $\infty$. This proves the result.
\end{proof}

\begin{proof}[Proof of Theorem~\ref{thm:fdd-trees}]
	In order to prove Theorem~\ref{thm:fdd-trees} we have to prove that
	\[ \varepsilon_p \left( \M(ps, pt) - \alpha^* \cM(2\varphi_\infty(ps), 2\varphi_\infty(pt)) \right) \Rightarrow 0. \]
	Since for any $t \in \R_+$ we have $p^{-1} K_{[pt]} \Rightarrow 2 \beta^* t$, for any $0 < \gamma < t$ we have $\P_p(E_p(t,\gamma)) \to 1$ as $p \to \infty$ where $E_p(t, \gamma)$ is the event
	\[ E_p(t, \gamma) = \left\{ K_{[\varphi_\infty(pt - p\gamma)]} \leq pt \leq K_{[\varphi_\infty(pt + p\gamma)]} \right\}. \]
	Thus in the sequel, for any $0 < \gamma < s < t$ we can assume that the event $E_p(s, \gamma) \cap E_p(t, \gamma)$ holds. By monotonicity, in this event we have
	\[ \M \left( K_{[\varphi_\infty(ps - p\gamma)]}, K_{[\varphi_\infty(pt + p\gamma)]} \right) \leq \M \left( ps, pt \right) \leq \M \left( K_{[\varphi_\infty(ps + p\gamma)]}, K_{[\varphi_\infty(pt - p\gamma)]} \right). \]
	Thus defining $a = \varphi_\infty(ps)$, $b = \varphi_\infty(pt)$, $a^\pm = [\varphi_\infty(ps \pm p \gamma)]$ and $b^\pm = \varphi_\infty(pt \pm p \gamma)$, we have
	\begin{multline*}
		\left \lvert \M(ps, pt) - \alpha^* \cM \left( 2\varphi_\infty(ps), 2\varphi_\infty(pt) \right) \right \rvert\\
		\leq \big \lvert \M \left( K_{a^+}, K_{b^-} \right) - \alpha^* \cM \left( 2a, 2b \right) \big \rvert + \left \lvert\M \left( K_{a^-}, K_{b^+} \right) - \alpha^* \cM \left( 2a, 2b \right) \right \rvert
	\end{multline*}
	and pursuing with the triangular inequality, we obtain
	\begin{multline*}
		\left \lvert \M(ps, pt) - \alpha^* \cM \left( 2\varphi_\infty(ps), 2\varphi_\infty(pt) \right) \right \rvert\\
		\leq \big \lvert \M \left( K_{a^+}, K_{b^-} \right) - \alpha^* \cM \left( 2a^+, 2b^- \right) \big \rvert + \left \lvert\M \left( K_{a^-}, K_{b^+} \right) - \alpha^* \cM \left( 2a^-, 2b^+ \right) \right \rvert\\
		+ \alpha^* \left \lvert \cM \left( 2a^+, 2b^- \right) - \cM \left(2a , 2b \right) \right \rvert + \alpha^* \left \lvert \cM \left( 2a^-, 2b^+ \right) - \cM \left(2a , 2b \right) \right \rvert.
	\end{multline*}
	Multiplying by $\varepsilon_p$, the two terms of the second line vanish as $p \to \infty$ by Corollary~\ref{cor:end}; letting then $\gamma \to 0$ makes the terms of the third line disappear by virtue of the convergence $\cC_p \Rightarrow \cC_\infty$ with $\cC_\infty$ continuous. The proof of Theorem~\ref{thm:fdd-trees} is complete.
\end{proof}

\section{Some examples where tightness fails} \label{sec:example}

In the Galton--Watson case, if the height process converges in the sense of finite-dimensional distributions toward a c\`adl\`ag process, then one actually only needs mild additional assumptions in order to get weak convergence in a functional sense of both the height and contour processes, essentially assumption~(H3c) discussed after Corollary~\ref{cor:H-fdd}. For instance, we automatically get weak convergence in the non-triangular case where the offspring distribution does not depend on $p$.

In this section we consider simple examples where the genealogical height and contour processes of the corresponding CMJ trees converge in the sense of finite-dimensional distributions but not necessarily in a functional sense. In contrast to the Galton--Watson case, we show that this can happen even in the non-triangular case. For these examples, all the assumptions of the main results of the present paper (namely Corollary~\ref{cor:H-fdd} and Theorem~\ref{thm:C-fdd}) hold, which shows that further conditions are called upon in order to strengthen these results to functional convergence.

Throughout this section, we assume that $(V^*_p, \cP^*_p)$ is equal in distribution to $(V^*, \cP^*)$, independent of $p$. We let $\xi = \lvert \cP^* \rvert$ and assume that its distribution is a critical offspring distribution in the domain of attraction of an $\alpha$-stable law with $\alpha \in (1,2)$. Then, it is known that for the choice $\varepsilon_p = \bar \varepsilon_p = p^{-(1-1/\alpha)}$, assumptions~(H3a)--(H3c) and~\ref{A.epsilon}--\ref{A.HC} hold. In particular, $S_p$ has jumps of the order of one which means that, typically, some nodes have of the order of $p \varepsilon_p = p^{1/\alpha}$ children: these nodes are called macroscopic.

\subsection{First family of examples}

To start with, consider the case
\[ \left( V^*, \cP^* \right) = \left( 1+\xi, \xi \delta_1 \right), \]
so that $\E(V^*) = 2$ and
\[ \E(\y^*) = \E \left( \int_0^\infty u \cP^*(\d u) \right) = \E \left( \xi \right) = 1. \]
In particular, assumptions~(H2) and~\ref{A.V} hold. The corresponding CMJ tree is then almost a Galton-Watson tree with offspring distribution the distribution of $\xi$, except that each edge is extended by a length equal to the number of children of the corresponding individual. Since $\H$ only depends on the $\cP_p$ but not on the $V_p$, we have $\H = \cH$ and so $\H_p$ converges weakly. On the other hand, macroscopic nodes have, by construction, edges with length of the order of $p^{1/\alpha}$. When the particle traveling along the edges meets such an edge, this makes $\C$ go up and then down at rate $\pm 1$ for a duration $p^{1/\alpha}$, so that during this time interval $\C$ has variation of the order of $p^{1/\alpha}$. Because of the scaling $\C_p(t) = p^{-(1-1/\alpha)} \C(pt)$, such a time interval corresponds for $\C_p$ to a time interval of size $p^{1/\alpha} \times (1/p) = p^{-(1-1/\alpha)}$, during which $\C_p$ has variation of the order of $p^{1/\alpha} \times p^{-(1-1/\alpha)} = p^{2/\alpha-1}$. Since $\alpha \in (1,2)$, in the limit we see that each macroscopic node should induce an infinite jump of $\C_p$. Since macroscopic nodes are dense, this strongly proscribes the tightness of $\C_p$.

\subsection{Second family of examples}

Let us now consider a variation of the above example, where both $\H_p$ and $\C_p$ fail to converge weakly: here we consider
\[ \left( V^*, \cP^* \right) = \left( 1+\xi, (\xi-1) \delta_1 + \delta_\xi \right), \]
so that $\E(V^*) = 2$ and $\E(\y^*) = \E(2\xi-1) = 1$. Again, the corresponding CMJ tree is almost a Galton-Watson tree, with the difference that all but one child are born at time $1$, and one child is born at a time equal to the number of children. The crucial difference with the first family of examples is that now, a macroscopic node also induces an infinite jump of $\H_p$ for the exact same reason as before.
\\

Let us now push this example a little further, and discuss the claim made in Section~\ref{sec:overview} that a uniform control of the kind
\begin{equation} \label{eq:rough-bound}
	\left \lvert \H_p(\varphi_p(t)) - \H_p(\varphi_\infty(t)) \right \rvert \leq \sup \left\{ \left \lvert \H_p(s) - \H_p(\varphi_\infty(t)) \right \rvert : \left \lvert s - \varphi_\infty(t) \right \rvert \leq \eta_p \right\}
\end{equation}
for some $\eta_p \to 0$ such that $\P_p(\lvert \varphi_p(t) - \varphi_\infty(t) \rvert \leq \eta_p) \to 1$ is too rough. Actually, we will discuss this with $\bar \varphi_p(t)$ instead of $\varphi_p(t)$ but since these two quantities are close (recall Lemma~\ref{lemma:tightness-Delta}), this discussion is equally insightful. In this case, classical results show that $\bar \varphi(pt) - \varphi_\infty(pt)$ is of the order of $p^{1/\alpha}$. Undoing the scaling, we see that we want to understand the order of magnitude for the variations of $\H$ on time scales of the order of $p^{1/\alpha}$, and in particular to see how these variations compare to the space scale $p^{1-1/\alpha}$.

Since $S_p$ converges to a stable process, it follows from the previous discussion that on the time scale $p^{1/\alpha}$, $S$ makes jumps of size $(p^{1/\alpha})^{1/\alpha} = p^{1/\alpha^2}$. As before, these jumps correspond to ``mesoscopic'' individuals with of the order of $p^{1/\alpha^2}$ children, which also have edge lengths of the same order. In particular, if the space scale $p^{1-1/\alpha}$ is negligible compared to $p^{1/\alpha^2}$, i.e., if
\[ 1 - \frac{1}{\alpha} < \frac{1}{\alpha^2} \Longleftrightarrow \alpha < \frac{1 + \sqrt 5}{2}, \]
then it is reasonable to expect the right-hand side of~\eqref{eq:rough-bound} to blow up, although we have proved that left-hand side vanishes.

To conclude, we mention that such examples could be generalized by considering
\[ \left( V^*, \cP^* \right) = \left( 1+\xi, (\xi-1) \delta_1 + \delta_{f(\xi)} \right) \]
for some function $f: [0,\infty) \to [0, \infty)$ such that $\E(f(\xi)) < \infty$. This extended family of examples then allows to decrease the above threshold involving the golden number, and also to show that even if $\alpha = 2$, i.e., the offspring distribution has finite variance, $\H_p$ may fail to be tight even though its finite-dimensional distributions converge.

\appendix

\section{Proof of Lemma~\ref{lemma:useful-identities}} \label{appendix:proof-useful-identities}

In this section we prove Lemma~\ref{lemma:useful-identities}: first consider the following lemma.

\begin{lemma}\label{lemma:lemma-useful-identities}
	For any $n \geq 0$ with $L(n) \circ \shift_n > 0$, we have
	\begin{equation} \label{eq:goal-lemma}
		T(T^{-1}(n)) = n + \tau_{L(n)} \circ \shift_n \ \text{ and } \ \cQ(T^{-1}(n)) = \mu_{L(n)} \circ \shift_n.
	\end{equation}
\end{lemma}

Considering this lemma with $n = n - m$, composing to the right with $\dualoperator^n$ and using $\shift_{n-m} \circ \dualoperator^n = \dualoperator^m$ by~\eqref{eq:sigma-dual}, this lemma gives Lemma~\ref{lemma:useful-identities} except for the fact that $n - T(T^{-1}(n-m)) \circ \dualoperator^n$ (or $m - \tau_{L(n-m)} \circ \dualoperator^m$) is equal to $\mrca{m}{n}$. Thus, in order to prove Lemma~\ref{lemma:useful-identities} we first prove Lemma~\ref{lemma:lemma-useful-identities} and then prove the equality with $\mrca{m}{n}$.

\subsection{Proof of Lemma~\ref{lemma:lemma-useful-identities}}

Let $n \geq 0$ with $L(n) \circ \shift_n > 0$. Simple computation shows that
\begin{equation} \label{eq:L-shift}
	L(n) \circ \shift_n = \max_{i = 0, \ldots, n} S(-i) \circ \shift_n = \max_{\{0, \ldots, n\}} S - S(n)
\end{equation}
and so $L(n) \circ \shift_n > 0$ means that $n$ is not a weak ascending ladder height time of $S$, in which case by definition of $T^{-1}(n)$ we have
\[ T(T^{-1}(n)) = \inf \left\{ k > n: S(k) \geq \max_{\{0, \ldots, n\}} S \right\}. \]
The right-hand side is always equal to $n + \tau_{L(n)} \circ \shift_n$: indeed,
\begin{align*}
	\tau_{L(n)} \circ \shift_n & = \inf \left\{ k > 0: S(k) \geq L(n) \right\} \circ \shift_n\\
	& = \inf \left\{ k > 0: S(k) \circ \shift_n \geq L(n) \circ \shift_n \right\}\\
	& = \inf \left\{ k > 0: S(n+k) - S(n) \geq \max_{\{0, \ldots, n\}} S - S(n) \right\}\\
	& = \inf \left\{ k > n: S(k) \geq \max_{\{0, \ldots, n\}} S \right\} - n.
\end{align*}
This proves the first identity in~\eqref{eq:goal-lemma}, and we now prove the second one. Define the random time $\Gamma = T(T^{-1}(n)-1)$: recalling the definition $\cQ(k) = \mu_0 \circ \shift_{T(k-1)}$, we see that we have to prove that $\mu_0 \circ \shift_\Gamma = \mu_{L(n)} \circ \shift_n$ (under the assumption $L(n) \circ \shift_n > 0$). Going back to the definition of $\mu_k = \Upsilon_{\zeta_k}(\cP_{\tau_k-1})$, we see that
\[ \mu_0 \circ \shift_\Gamma = \Upsilon_{\zeta_0 \circ \shift_\Gamma} \left( \cP_{\tau_0 \circ \shift_\Gamma + \Gamma - 1} \right) \ \text{ and } \ \mu_{L(n)} \circ \shift_n = \Upsilon_{\zeta_{L(n)} \circ \shift_n} \left( \cP_{\tau_{L(n)} \circ \shift_n + n - 1} \right) \]
and so it is enough to show that
\[ \zeta_0 \circ \shift_\Gamma = \zeta_{L(n)} \circ \shift_n \ \text{ and } \ \tau_0 \circ \shift_\Gamma + \Gamma = \tau_{L(n)} \circ \shift_n + n. \]
We first show the second identity. Since $\tau_0 = T(1)$ and $T(1) \circ \shift_{T(k)} + T(k) = T(k+1)$ for any $k \geq 0$, considering $k = T^{-1}(n)-1$ yields $\tau_0 \circ \shift_\Gamma + \Gamma = T(T^{-1}(n))$ which is equal to $\tau_{L(n)} \circ \shift_n + n$ as has been argued above.

Using this equality, we now prove that $\zeta_0 \circ \shift_\Gamma = \zeta_{L(n)} \circ \shift_n$ which will conclude the proof of Lemma~\ref{lemma:lemma-useful-identities}. Since $\zeta_{L(n)} = L(n) - S(\tau_{L(n)}-1)$, $L(n) \circ \shift_n = \max_{\{0, \ldots, n\}} S - S(n)$ by~\eqref{eq:L-shift} and
\[ S(\tau_{L(n)}-1) \circ \shift_n = S\big( \tau_{L(n)} \circ \shift_n + n - 1 \big) - S(n) = S \big( T(T^{-1}(n)) - 1 \big) - S(n), \]
we get
\[ \zeta_{L(n)} \circ \shift_n = \max_{\{0, \ldots, n\}} S - S \big( T(T^{-1}(n)) - 1 \big). \]
Moreover,
\[ \zeta_0 \circ \shift_\Gamma = - S(\tau_0 - 1) \circ \shift_\Gamma = S(\Gamma) - S(\tau_0 \circ \shift_\Gamma + \Gamma - 1) = S(\Gamma) - S \big( T(T^{-1}(n)) - 1 \big). \]
Since the condition $L(n) \circ \theta_n > 0$ means that $n$ is not a weak ascending ladder height time of $S$, we have $T^{-1}(n) = \bTm(n)+1$ and in particular, $\Gamma = T(\bTm(n))$. Thus, $S(\Gamma) = \max_{\{0, \ldots, n\}} S$ by definition of $\bTm(n)$ which concludes the proof.

\subsection{Proof of the identity for $\mrca{m}{n}$}

Let $0 \leq m \leq n$ with $L(n-m) \circ \dualoperator^m > 0$ and define $\kappa = n - T(T^{-1}(n-m)) \circ \vartheta^n$: in order to conclude the proof of Lemma~\ref{lemma:useful-identities}, we now prove that $\mrca{m}{n} = \kappa$. Since on the one hand $\kappa = n - T(T^{-1}(n-m)) \circ \dualoperator^n$, it follows from the definition of $\cA(n)$ that $\kappa \in \cA(n)$. Moreover, $\tau_{L(n-m)}$ is by definition a weak ascending ladder height time, i.e., for every $k \geq 0$ there exists $\Gamma$ such that $\tau_k = T(\Gamma)$: in particular, $\kappa = m - \tau_{L(n-m)} \circ \dualoperator^m$ also belongs to $\cA(m)$. In order to conclude the proof it remains to show that $\kappa \geq \alpha$ for any $\alpha \in \cA(m) \cap \cA(n)$. By definition, we can write such an $\alpha$ as
\[ \alpha = n - T(\Gamma) \circ \dualoperator^n = m - T(\Gamma') \circ \dualoperator^m \]
for some $\Gamma$, $\Gamma' \geq 0$. In particular,
\[ T(\Gamma) \circ \dualoperator^n = n - m + T(\Gamma') \circ \dualoperator^m \geq n - m \]
and so by definition of $T^{-1}$, we have $\Gamma \circ \dualoperator^n \geq T^{-1}(n-m) \circ \dualoperator^n$. Since the weak ascending ladder height times form an increasing sequence, this implies $T(\Gamma) \circ \dualoperator^n \geq T(T^{-1}(n-m)) \circ \dualoperator^n$ and so $\alpha \leq \kappa$, which concludes the proof.

\section{Proof of Lemma~\ref{lemma:chi}} \label{appendix:proof-identity-chi}

Let $n \geq 0$, $m = n - \tau_0 \circ \dualoperator^n$, $i = \zeta_0 \circ \dualoperator^n$ and assume that $m \geq 0$: we have to prove that $i \in \{0, \ldots, \lvert \cP_m \rvert - 1 \}$ and $\chi(m,i) = n$. Let us first prove that $i \in \{0, \ldots, \lvert \cP_m \rvert - 1 \}$. Since $\cP_m = \cP_{n - \tau_0 \circ \dualoperator^n} = \cP_{\tau_0 - 1} \circ \dualoperator^n$, this follows from the fact that $\zeta_0 = -S(\tau_0 - 1) \leq S(\tau_0) - S(\tau_0 - 1) = \lvert \cP_{\tau_0 - 1} \rvert - 1$ and then composing on the right with $\dualoperator^n$.

Let us now prove that $\chi(m, i) = n$. By definition of $\chi$ and since $S$ only makes negative jumps of size $-1$, we have to prove that
\begin{equation} \label{eq:goal-1}
	S(n) = S(m+1) - i
\end{equation}
and that
\begin{equation} \label{eq:goal-2}
	S(\ell) > S(m+1) - i, \ \ell = m+1, \ldots, n-1.
\end{equation}
Let us first prove~\eqref{eq:goal-1}. By definition of $m$ and $i$ we have
\[ S(m+1) - i = S \big( n - \tau_0 \circ \dualoperator^n + 1 \big) - \zeta_0 \circ \dualoperator^n = S \big( n - \tau_0 \circ \dualoperator^n + 1 \big) + S(\tau_0 - 1) \circ \dualoperator^n \]
which by~\eqref{eq:S(Gamma)-dual} (applied with $\Gamma = \tau_0 - 1$) implies~\eqref{eq:goal-1}. Let us now prove~\eqref{eq:goal-2}: in view of~\eqref{eq:goal-1} we have to prove that
\[ \min \left\{S(k) : k = m+1, \ldots, n-1 \right\} > S(n) \]
which directly follows from the fact that
\[ \min \left\{S(k) : k = m+1, \ldots, n-1 \right\} = S(n) - \max \left\{S \left( k \right) : k = 1, \ldots, T(1) - 1 \right\} \circ \dualoperator^n. \]
%%%%%%%%%%%
% DETAILS %
%%%%%%%%%%%
% \begin{align*}
% 	\min \left\{S(k) : k = m+1, \ldots, n-1 \right\} & = \min \left\{S(m+k) : k = 1, \ldots, n-m-1 \right\}\\
% 	& = \min \left\{S \left( n - (T(1)-k) \circ \dualoperator^n \right) : k = 1, \ldots, T(1) \circ \dualoperator^n - 1 \right\}\\
% 	& = \min \left\{S(n) - S \left( T(1)-k \right) \circ \dualoperator^n : k = 1, \ldots, T(1) \circ \dualoperator^n - 1 \right\}\\
% 	& = S(n) - \max \left\{S \left( T(1)-k \right) \circ \dualoperator^n : k = 1, \ldots, T(1) \circ \dualoperator^n - 1 \right\}\\
% 	& = S(n) - \max \left\{S \left( T(1)-k \right) : k = 1, \ldots, T(1) - 1 \right\} \circ \dualoperator^n\\
% 	& = S(n) - \max \left\{S \left( k \right) : k = 1, \ldots, T(1) - 1 \right\} \circ \dualoperator^n.
% \end{align*}
%%%%%%%%%%%%%%%
% END DETAILS %
%%%%%%%%%%%%%%%

\section{Proof of Propositions~\ref{Prop:renewal-1} and~\ref{Prop:renewal-2}} \label{sec:coupling}

\subsection{Coupling between random walks} \label{sub:coupling}

We present here the coupling of Lemma $9$.$12$ in Kallenberg~\cite{Kallenberg02:0} between two random walks with the same step distribution and possibly different initial distributions. The coupling starts from the following stochastic primitives, which are assumed to be mutually independent:
\begin{itemize}
	\item $\alpha$ and $\alpha'$, two independent real-valued random variables;
	\item $(\xi_k)$, i.i.d.\ sequence of real-valued random variables;
	\item $(\varrho_k)$, i.i.d.\ sequence with $\P(\varrho_k = \pm 1) = 1/2$.
\end{itemize}

Let
\[ \widetilde W(n)= \alpha' - \alpha + \sum_{k = 1}^n \varrho_k \xi_k, \ n \geq 0, \]
so that $\widetilde W$ is a critical random walk with initial distribution $\alpha' - \alpha$ and step distribution $\varrho_1 \xi_1$. Fix in the rest of this subsection $\varepsilon > 0$ and define the following quantities:
\begin{itemize}
	\item $A_\varepsilon = \inf \big\{ n \geq 0: \widetilde W(n)\in [0, \varepsilon] \big\}$;
	\item $\varrho'_k = (-1)^{\Indicator{k \leq A_\varepsilon}} \varrho_k$;
	\item $\kappa_1 < \kappa_2 < \cdots$ the values of $k$ with $\varrho_k = 1$ and $\kappa'_1 < \kappa'_2 < \cdots$ the values of $k$ with $\varrho'_k = 1$;
	\item and finally
	\[ W(n)= \alpha + \sum_{j = 1}^n \xi_{\kappa_j} \ \text{ and } \ W'(n)= \alpha' + \sum_{j = 1}^n \xi_{\kappa'_j}. \]
\end{itemize}

\begin{lemma}\label{lem:}
	$W$, respectively $W'$, is a random walk with step distribution $\xi_1$ and initial distribution $\alpha$, respectively $\alpha'$.
\end{lemma}

\begin{proof}
	See the proof of Lemma~$9.21$ in Kallenberg~\cite{Kallenberg02:0}.
\end{proof}

Thus we have constructed a coupling of two random walks with the same step distribution, as promised. The interest of this coupling lies in the following result, which exhibits an event in which many increments of $W$ and $W'$ are equal. In the sequel we define:
\begin{itemize}
	\item $\sigma = \lvert \{j: \kappa_j\leq A_\varepsilon\}\rvert$, $\sigma'=\lvert\{j: \kappa_j'\leq A_\varepsilon\}\rvert$ and
	\[ \gamma = \max\left(\max_{k = 0, \ldots, \sigma} W(k), \ \max_{k = 0, \ldots, \sigma'} W'(k)\right); \]
	\item $\psi(t) = \inf\{n \geq 0: W(n)\geq t\}$ and $\psi'(t) = \inf\{n \geq 0: W'(n)\geq t\}$ for $t \geq 0$;
	\item $\Delta_k = W(k) - W(k-1)$ and $\Delta'_k = W'(k) - W'(k-1)$ for $k \geq 1$.
\end{itemize}

\begin{lemma}\label{lemma:A1}
	For any $m \in \N$ and $t, \varepsilon \in \R_+$, in the event
	\[ \left\{ \gamma < t \right\} \cap \left\{ \psi(t) > A_\varepsilon + m \right\} \cap \left\{ W'(\psi'(t)) \geq t + 2 \varepsilon \right\}, \]
	we have $\Delta_{\psi(t) - k} = \Delta'_{\psi'(t) - k}$ for any $k = 0, \ldots, m$.
\end{lemma}

%%%%%%%%%%%
% DETAILS %
%%%%%%%%%%%
% Let us first prove the desired equality at time $0$, i.e., that $S_{\sigma} = S'_{\sigma'} - \widetilde S_\sigma$. We have
% \[ \sum_{j \leq \sigma} \varrho_j \xi_j = \sum_{j \leq \sigma: \varrho_j = 1} \xi_j - \sum_{j \leq \sigma: \varrho_j = -1} \xi_j \]
% and since for $j \leq \sigma$ we have $\varrho'_j = - \varrho_j$, this can be rewritten as
% \[ \sum_{j \leq \sigma} \varrho_j \xi_j = \sum_{j \leq \sigma: \varrho_j = 1} \xi_j - \sum_{j \leq \sigma: \varrho'_j = 1} \xi_j. \] By definition of $S$, $S'$ and $\widetilde S$, we have to prove that
% \[ \sum_{j \leq \sigma} \varrho_j \xi_j = \sum_{j \leq \sigma} \xi_{\kappa_j} - \sum_{j \leq \sigma'} \xi_{\kappa_j'}. \]
%
% By definition of $\kappa_j$ and $\kappa'_j$, this means precisely that
% \[ \sum_{j \leq \sigma} \varrho_j \xi_j = \sum_{j \leq \sigma} \xi_{\kappa_j} - \sum_{j \leq \sigma'} \xi_{\kappa_j'}. \]
% which is exactly the relation $S_{\sigma} = S'_{\sigma'} - \widetilde S_\sigma$ which we wanted to prove. Let us now prove that the increments of the shifted processes are equal, i.e., that $S_{\sigma + n+1} - S_{\sigma + n} = S'_{\sigma'+n+1} - S'_{\sigma'+n}$ for $n \geq 0$. By definition, this is equivalent to showing that $\xi_{\kappa_{\sigma + n + 1}} = \xi_{\kappa'_{\sigma' + n + 1}}$ for $n \geq 0$. Since $\varrho'_j = \varrho_j$ for $j > \sigma$, we obtain $\kappa_{\sigma+j} = \kappa'_{\sigma'+j}$ for $j \geq 1$ which gives the desired result.
%%%%%%%%%%%%%%%
% END DETAILS %
%%%%%%%%%%%%%%%

\begin{proof}
	Let $\overline W$ and $\overline W'$ be the processes $W$ and $W'$ shifted at time $\sigma$ and $\sigma'$, respectively, i.e., defined by $\overline W(n)= W(\sigma + n)$ and $\overline W'(n)= W'(\sigma' + n)$. Then for any $n \geq 0$, we have
	\begin{equation} \label{eq:equality-shift}
		\overline W(n)= \overline W'(n)- \widetilde W(A_\varepsilon),
	\end{equation}
	see~\cite[Lemma~$9.21$]{Kallenberg02:0} for details. Assume in the rest of the proof that $\gamma < t$, $\psi(t) > A_\varepsilon + m$ and $W'({\psi'(t)}) \geq t + 2 \varepsilon$. By definition, $\gamma < t$ implies that $\psi(t) \geq \sigma$ and so we can write
	\[ \psi(t) = \inf \left \{ n \geq \sigma: W(n) \geq t \right \} = \sigma + \inf \left \{ n \geq 0: \overline W(n) \geq t \right \}. \]
	In particular, using~\eqref{eq:equality-shift} we obtain
	\[ \psi(t) = \sigma + \inf \left \{ n \geq 0: \overline W'(n) \geq t - \widetilde W(A_\varepsilon) \right \} \]
	and since $\gamma < t$ implies $\psi'(t) \geq \sigma'$ as well, a symmetric reasoning finally entails
	\[ \psi(t) - \sigma = \psi'\big( t + 2 \widetilde W(A_\varepsilon) \big) - \sigma'. \]
	Since by definition $\widetilde W(A_\varepsilon) \leq \varepsilon$ and since we assume $W'(\psi'(t)) \geq t + 2\varepsilon$, we further get that $\psi'(t + 2 \widetilde W(A_\varepsilon)) = \psi'(t)$, which finally proves that $\psi(t) - \sigma = \psi'(t) - \sigma'$. Consider now any $k = 0, \ldots, m+1$, so that $\psi(t) \geq k + \sigma$ and $\psi'(t) \geq k + \sigma'$ as a consequence of the assumption $\psi(t) > \sigma + m$ and the fact that $\psi(t) - \sigma = \psi'(t) - \sigma'$: then we have
	\begin{align*}
		W(\psi(t) - k) & = \overline W(\psi(t) - \sigma - k) & \text{(by definition of $\overline W$)}\\
		& = \overline W(\psi'(t) - \sigma' - k) & \text{(by $\psi(t) - \sigma = \psi'(t) - \sigma'$)}\\
		& = \overline W'(\psi'(t) - \sigma' - k) - \widetilde W(A_\varepsilon) & \text{(by~\eqref{eq:equality-shift})},
	\end{align*}
	which finally gives
	\[ W(\psi(t) - k) = W'(\psi'(t) - k) - \widetilde W(A_\varepsilon) \]
	by definition of $\overline W'$. This last equality readily implies the desired result.
\end{proof}

\subsection{Stationary renewal processes on $\R$}

In this subsection and the following one, we fix some $p \geq 1$. We enrich the probability space $\L^\Z$ to $\L^\Z \times \R_+ \times (0,\infty)$ and denote by $(\omega, d_-, d_+) \in \L^\Z \times \R_+ \times (0,\infty)$ the canonical sequence. We then define
\[ W_\pm(n) = d_\pm + \sum_{k=1}^n V_{\pm k}, \ n \geq 0, \]
as well as the following point process on $\R \times \cM$:
\[ Z = \sum_{n \geq 0} \epsilon_{(W_+(n), \cP_n)} + \sum_{n \geq 0} \epsilon_{(-W_-(n), \cP_{-n-1})}. \]

For $\chi$ a probability distribution on $\R_+ \times (0,\infty) \times \cM$, let $\P^\chi_p$ be the probability measure under which:
\begin{itemize}
	\item $((V_n, \cP_n), n \in \Z \setminus \{0\})$ are i.i.d.\ with common distribution $(2 V^*_p, \cP^*_p)$;
	\item $(d_-, d_+, \cP_0)$ is independent from this sequence and has distribution $\chi$.
\end{itemize}
Under $\P^\chi_p$, $V_0$ will not play a role. We consider $\Theta_t$ the shift operator acting on measures on $\R \times \cM$ as follows: for any measure $\nu$ on $\R \times \cM$ and any Borel sets $B \subset \R$ and $M \subset \cM$,
\[ \Theta_t \nu(B \times M) = \nu((t + B) \times M). \]
Note that $Z$ uniquely characterizes the canonical sequence $(\omega, d_-, d_+)$ and so with a slight abuse of notation, we will sometimes consider that we are working on the canonical space of locally finite point measures on $\R \times \cM$, that $Z$ is the canonical measure and that $\omega$ and $d_\pm$ are functional thereof, e.g., $d_+ = \inf\{t > 0: Z(\{t\} \times \cM) > 0\}$. In particular, the notation $\P^\chi_p \circ \Theta^{-1}_p$ makes sense, which is rigorously to be understood as the law of $\Theta_t Z$ under $\P^\chi_p$.

Recall that $\hat V^*_p$ follows the size-biased distribution of $V^*_p$, and let $(\hat V^*_p, U^*_p, \hat \cP^*_p)$ be such that, conditionally on $\hat V^*_p = v$:
\begin{itemize}
	\item $\hat \cP^*_p$ is independent from $U^*_p$ and is distributed like $\cP^*_p$ conditionally on $V^*_p = v$;
	\item if $V^*_p$ is non-arithmetic, $U^*_p$ is a uniform random variable on $[0, v]$;
	\item if $V^*_p$ is arithmetic with span $h$, $U^*_p$ is a uniform random variable on $\{0, h, \ldots, v\}$.
\end{itemize}
Let $\chi_p$ be the law of $(0, 2V^*_p, \cP^*_p)$ and $\hat \chi_p$ the law of $(2(\hat V^*_p - U^*_p), 2U^*_p, \hat \cP^*_p)$. The following result corresponds to Theorem $2.1$ in Miller~\cite{Miller74:0}.

\begin{theorem} \label{thm:shift-invariance}
	The measure $\P^{\hat \chi_p}_p$ is shift invariant, i.e.:
	\begin{itemize}
		\item if $V^*_p$ is non-arithmetic, then $\P^{\hat \chi_p}_p \circ \Theta^{-1}_t = \P^{\hat \chi_p}_p$ for every $t \in \R$;
		\item if $V^*_p$ is arithmetic with span $h$, then $\P^{\hat \chi_p}_p \circ \Theta^{-1}_{ih} = \P^{\hat \chi_p}_p$ for every $i \in \Z$.
	\end{itemize}
\end{theorem}

The coupling presented in the previous section can be extended to the case of a marked random walk to give the following result. In the sequel, let $h_p \in \R_+$ be the span of $V^*_p$ in the arithmetic case, and with a slight abuse in notation let $h_p: \R_+ \to \R_+$ be the function such that $h_p(t) = t$ in the non-arithmetic case and $h_p(t) = h_p [t / h_p]$ in the arithmetic case.

\begin{lemma} \label{lemma:coupling-R}
	For any $m \in \N$, $t \in \R_+$ and $f: \R \times \cM^{m+1} \to [0,1]$ measurable, the inequality
	\begin{multline*}
		\left \lvert \E^{\chi_p}_p \circ \Theta_t^{-1} \left[ f \left( d_+ + d_-, \cP_0, \cP_{-1}, \ldots, \cP_{-m} \right) \right] - \E^{\hat \chi_p}_p \left[ f \left( d_+ + d_-, \cP_0, \cP_{-1}, \ldots, \cP_{-m} \right) \right] \right \rvert\\
		\leq \P \left( U^*_p < \varepsilon \right) + \P \left( U^*_p \geq t' \right) + 3 \P_p \big( \cV(m+n) \geq h_p(t)/2 - t' \big) + 2 \P \left( A^p_\varepsilon \geq n \right),
	\end{multline*}
	holds for any $n \in \N$ and $t', \varepsilon \in \R_+$, where $A^p_\varepsilon$ is the random variable defined in Section~\ref{sub:coupling} for $\alpha$ and $\xi_1$ equal in distribution to $2 V^*_p$ and $\alpha'$ to $2 U^*_p$.
\end{lemma}

\begin{proof}
	First of all, note that by definition of $h_p(t)$ the law of $(d_+ + d_-, \cP_0, \cP_{-1}, \ldots, \cP_{-m})$ is the same under $\P^{\chi_p}_p \circ \Theta^{-1}_t$ and $\P^{\chi_p}_p \circ \Theta^{-1}_{h_p(t)}$. In particular, we can assume without loss of generality that $t = h_p(t)$, which allows us to use Theorem~\ref{thm:shift-invariance} to get $\P^{\hat \chi_p}_p \circ \Theta^{-1}_t = \P^{\hat \chi_p}_p$.
	
	Next, considering the notation of Section~\ref{sub:coupling}, we consider the coupling described there with $\alpha$ and $\xi_1$ equal in distribution to $2 V^*_p$ and $\alpha'$ to $2 U^*_p$. We modify this coupling in two ways: $(1)$ we extend $W(n)$ and $W'(n)$ for $n \leq -1$ arbitrarily; $(2)$ we consider an additional sequence $(\nu_k)$ of marks, whereby $W(n)$, resp.\ $W'(n)$, is given the mark $m_n = \nu_{\kappa_n}$, resp.\ $m'_n = \nu_{\kappa'_n}$. This way, in addition to the conclusions of Lemma~\ref{lemma:A1} we obtain that $m_{\psi(t) - k} = m'_{\psi'(t) - k}$ for any $k = 0, \ldots, m$ in the event described there. In particular, if marks take value in $\cM$ then for any measurable function $f: \R \times \cM^{m+1} \to [0,1]$ we obtain
\begin{multline*}
	\left \lvert \E \left[ f(\Delta_{\psi(t)}, \nu_{\psi(t)}, \ldots, \nu_{\psi(t)-m}) \right] - \E \left[ f(\Delta'_{\psi'(t)}, \nu'_{\psi'(t)}, \ldots, \nu'_{\psi'(t)-m}) \right] \right \rvert\\
	\leq \P \left( \gamma \geq t \right) + \P \left( \psi(t) \leq A^p_\varepsilon + m + 1 \right) + \P\left( W'(\psi'(t)) < t + 2 \varepsilon \right).
\end{multline*}

When $\alpha$, $\alpha'$ and $\xi_1$ are as prescribed above and the $((\xi_k, \nu_k), k \in \N)$ are i.i.d.\ with common distribution $(V^*_p, \cP^*_p)$, we get the identities
\[ \E \left[ f(\Delta_{\psi(t)}, \nu_{\psi(t)}, \ldots, \nu_{\psi(t)-m}) \right] = \E^{\chi_p}_p \circ \Theta_t^{-1} \left[ f \left( d_+ + d_-, \cP_0, \cP_{-1}, \ldots, \cP_{-m} \right) \right] \]
and
\[ \E \left[ f(\Delta'_{\psi'(t)}, \nu'_{\psi'(t)}, \ldots, \nu'_{\psi'(t)-m}) \right] = \E^{\hat \chi_p}_p \circ \Theta_t^{-1} \left[ f \left( d_+ + d_-, \cP_0, \cP_{-1}, \ldots, \cP_{-m} \right) \right]. \]

Since $\P^{\hat \chi_p}_p$ is shift-invariant, we thus get the bound
\begin{multline*}
	\left \lvert \E^{\chi_p}_p \circ \Theta_t^{-1} \left[ f \left( d_+ + d_-, \cP_0, \cP_{-1}, \ldots, \cP_{-m} \right) \right] - \E^{\hat \chi_p}_p \left[ f \left( d_+ + d_-, \cP_0, \cP_{-1}, \ldots, \cP_{-m} \right) \right] \right \rvert\\
	\leq \P \left( \gamma \geq t \right) + \P \left( \psi(t) \leq A^p_\varepsilon + m \right) + \P\left( W'(\psi'(t)) < t + 2 \varepsilon \right)
\end{multline*}
and so in order to conclude the proof, it remains to show that
\begin{multline*}
	\P \left( \gamma \geq t \right) + \P \left( \psi(t) \leq A^p_\varepsilon + m \right) + \P\left( W'(\psi'(t)) < t + 2 \varepsilon \right)\\
	\leq \P \left( U^*_p < \varepsilon \right) + \P \left( U^*_p \geq t' \right) + 3 \P_p \big( \cV(m+n) \geq t/2 - t' \big) + 2 \P \left( A^p_\varepsilon \geq n \right).
\end{multline*}

First of all, by definition we have $\P(W'(\psi'(t)) < t + 2 \varepsilon) = \P^{\hat \chi_p}_p \circ \Theta^{-1}_t(d_+ < 2 \varepsilon)$ and so since $\P^{\hat \chi_p}_p$ is shift-invariant, we obtain
\[ \P\left( W'(\psi'(t)) < t + 2 \varepsilon \right) = \P^{\hat \chi_p}_p(d_+ < 2 \varepsilon) = \P \left( 2 U^*_p < 2 \varepsilon \right) = \P \left( U^*_p < \varepsilon \right). \]

Further, since in the present case $W$ and $W'$ are increasing and $\sigma, \sigma' \leq A_\epsilon$ by construction, we get
\[ \P \left( \gamma \geq t \right) = \P \left( W(\sigma) \geq t \ \text{ or } \ W'(\sigma') \geq t \right) \leq \P \left( A^p_\varepsilon \geq n \right) + \P \left( W(n) \geq t \right) + \P \left( W'(n) \geq t \right). \]

Since $W(n)$ is equal in distribution to $2 \cV(n)$ (under $\P_p$), $W'(0)$ is equal in distribution to $2 U^*_p$ and $W'(n) - W'(0)$ is equal in distribution to $2 \cV(n-1)$, we obtain
\[ \P \left( W(n)\geq t \right) + \P \left( W'(n) \geq t \right) \leq \P \left( U^*_p \geq t' \right) + 2 \P_p \left( \cV(m+n) \geq t/2 - t' \right). \]

Finally, since
\[ \P \left( \psi(t) \leq A^p_\varepsilon + m \right) \leq \P \left( \psi(t) \leq m+n \right) + \P \left( A^p_\varepsilon \geq n \right) \]
and $\P \left( \psi(t) \leq m+n \right) = \P(W(m+n) \geq t) = \P_p(2 \cV(m+n) \geq t) \leq \P_p(\cV(m+n) \geq t/2 - t')$, gathering the previous inequalities gives the desired result.
\end{proof}

\subsection{Proof of Propositions~\ref{Prop:renewal-1} and~\ref{Prop:renewal-2}}
Let
\[ Y_p = \left( 2 V_{\bar \varphi(pt)}, \cP_{\bar \varphi(pt)}, \ldots, \cP_{\bar \varphi(pt) - [p \delta]} \right) \ \text{ and } \ \hat Y_p = \left( 2 \hat V^*_p, \hat \cP_p^*, \cP^*_p(1), \ldots, \cP^*_p([p \delta]) \right). \]

Then by definition of $\P^{\chi_p}_p$ and $\P^{\hat \chi_p}_p$, we have
\[ \E_p \left[ f \left( Y_p \right) \right] = \E^{\chi_p}_p \circ \Theta_{pt}^{-1} \left[ f \left( d_+ + d_-, \cP_0, \cP_{-1}, \ldots, \cP_{-[p \delta]} \right) \right] \]
and
\[ \E \left[ f\left( \hat Y_p \right) \right] = \E^{\hat \chi_p}_p \left[ f \left( d_+ + d_-, \cP_0, \cP_{-1}, \ldots, \cP_{-[p \delta]} \right) \right] \]
and so for any $n \in \N$, $t', \varepsilon \in \R_+$ and $f: \R_+ \times \cM^{[p \delta]+1} \to [0,1]$ measurable, Lemma~\ref{lemma:coupling-R} gives
\begin{multline} \label{eq:coupling-bound-2}
	\left \lvert \E_p \left[ f\left(Y_p \right) \right] - \E \big[ f\big(\hat Y_p \big) \big] \right \rvert \leq \P \left( U^*_p < \varepsilon \right) + \P \left( U^*_p \geq t' \right)\\
	+ 3 \P_p \big( \cV([p \delta] + n) \geq h_p(pt)/2 - t' \big) + 2 \P \left( A^p_\varepsilon \geq n \right).
\end{multline}
Let $p \to \infty$, and assume for a moment that the previous upper bound vanishes by suitably playing on the free parameters $\varepsilon$, $t'$ and $n$ (after having taken the limit $p \to \infty$): by considering
\[ f(v, \nu_0, \ldots, \nu_{[p \delta]}) = f(v/2, \nu_0) \]
with $f: \R_+ \times \cM \to \R_+$ continuous bounded for Proposition~\ref{Prop:renewal-1}, and $f = g \circ \Xi_{[p \delta]}$ with $g: \R \to \R$ continuous and bounded for Proposition~\ref{Prop:renewal-2}, this would give the desired result. We now explain how to make the upper bound in~\eqref{eq:coupling-bound-2} vanish.

First of all, note that $h_p(t) \sim pt$ as $p \to \infty$: in the non-arithmetic case this is trivial, while in the arithmetic case, this follows from the fact that $\sup_p h_p < \infty$ (which follows from the assumption $V^*_p \Rightarrow V^*_\infty$ with $V^*_\infty$ arithmetic). Therefore, Lemma~\ref{lemma:triangular-LLN} implies that $\P_p \big( \cV([p \delta] + n) \geq h_p(pt)/2 - t' \big) \to 0$ as $p \to \infty$, for $\delta < t / (2 \beta^*)$ and fixed $n$ and $t'$.

To deal with the other terms, define $U^*_\infty$ and $A^\infty_\varepsilon$ from $V_\infty^*$ similarly as $U^*_p$ as $A^p_\varepsilon$ from $V^*_p$, respectively. From $\hat V^*_p \Rightarrow \hat V^*_\infty$ we obtain that $U^*_p \Rightarrow U^*_\infty$ and so letting $t' \to \infty$ after $p \to \infty$, we obtain
\begin{equation} \label{eq:coupling-bound-3}
	\left \lvert \E_p \left[ f\left(Y_p \right) \right] - \E \big[ f\big(\hat Y_p \big) \big] \right \rvert \leq \limsup_{p \to \infty} \P \left( U^*_p < \varepsilon \right) + 2 \limsup_{p \to \infty} \P \left( A^p_\varepsilon \geq n \right).
\end{equation}
We further distinguish the arithmetic and non-arithmetic cases.
\\

\noindent \textit{Arithmetic case.} In this case, we have $A^p_0 \Rightarrow A^\infty_0$ and since $U^*_p \geq 0$, considering~\eqref{eq:coupling-bound-3} with $\varepsilon = 0$ gives
\[ \limsup_{p \to \infty} \left \lvert \E_p \left[ f_p\left(Y_p \right) \right] - \E \big[ f_p\big(\hat Y_p \big) \big] \right \rvert \leq 2 \P \left( A^\infty_0 \geq n \right). \]
Since $A^\infty_0$ is almost surely finite, letting $n \to \infty$ gives the result.
\\

\noindent \textit{Non-arithmetic case.} In this case, we have $A^p_\varepsilon \Rightarrow A^\infty_\varepsilon$ for any $\varepsilon > 0$ and since $U^*_\infty$ is absolutely continuous with respect to Lebesgue measure,~\eqref{eq:coupling-bound-3} with $\varepsilon > 0$ gives 
\[ \limsup_{p \to \infty} \left \lvert \E_p \left[ f_p\left(Y_p \right) \right] - \E \big[ f_p\big(\hat Y_p \big) \big] \right \rvert \leq \P \left( U^*_\infty < \varepsilon \right) + 2 \P \left( A^\infty_\varepsilon \geq n \right). \]
Since $A^\infty_\varepsilon$ is almost surely finite (for $\varepsilon > 0$) and $U^*_\infty$ does not put mass at $0$, letting first $n \to \infty$ and then $\varepsilon \to 0$ finally achieves the proof.

\end{document}